\documentclass[11pt, twoside, reqno]{amsart}
\usepackage{amsmath}
\usepackage{amsfonts}
\usepackage{amssymb}
\usepackage{amsthm}
\usepackage[colorlinks=true, linkcolor=blue, citecolor=purple]{hyperref}
\usepackage{graphicx}
\usepackage{cite}
\usepackage{caption}
\usepackage{color}
\usepackage[all]{xy}
\usepackage[colorinlistoftodos]{todonotes}
\usepackage{mathtools}
\usepackage{tikz}
\usetikzlibrary{decorations.pathreplacing}
\usetikzlibrary{arrows}
\usetikzlibrary{shapes.misc}

\usepackage[top=3.35cm, bottom=3.35cm, left=2.5cm, right=2.5cm, headsep=0.2in]{geometry}

\usepackage{fancyhdr}

\DeclareMathOperator{\Tr}{tr}

\DeclareMathOperator{\im}{Im}
\DeclareMathOperator{\re}{Re}
\DeclareMathOperator{\rk}{rank}
\DeclareMathOperator{\ind}{ind}
\DeclareMathOperator{\ch}{ch}
\DeclareMathOperator{\dd}{d}
\DeclareMathOperator{\supp}{supp}
\DeclareMathOperator{\vol}{vol}
\DeclareMathOperator{\res}{Res}

\DeclareMathOperator{\id}{Id}
\DeclareMathOperator{\E}{\mathcal{E}}

\DeclareMathOperator{\Op}{Op}

\DeclareMathOperator{\Bd}{Bd}
\DeclareMathOperator{\ran}{ran}
\DeclareMathOperator{\spann}{span}

\theoremstyle{plain}
\newtheorem{theorem}{Theorem}[section]
\newtheorem{definition}[theorem]{Definition}
\newtheorem{lemma}[theorem]{Lemma}
\newtheorem{rem}[theorem]{Remark}
\newtheorem{prop}[theorem]{Proposition}

\newtheorem{corollary}[theorem]{Corollary}

\numberwithin{equation}{section}

\newcommand{\Lapl}{\mathcal{L}}

\begin{document}

\title{Resonant spaces for volume preserving Anosov flows}
\author[M. Ceki\'{c}]{Mihajlo Ceki\'{c}}
\date{\today}
%\address{Max-Planck Institute for Mathematics, Vivatsgasse 7, 53111, Bonn, Germany}
\address{Laboratoire de Math\'{e}matiques d'Orsay, CNRS, Universit\'{e} Paris-Saclay, 91405 Orsay, France}
\email{mihajlo.cekic@u-psud.fr}
\author[G.P. Paternain]{Gabriel P. Paternain}
%\date{\today}
\address{Department of Pure Mathematics and Mathematical Statistics, University of Cambridge, Cambridge CB3 0WB, UK}
\email{g.p.paternain@dpmms.cam.ac.uk}

\begin{abstract} We consider Anosov flows on closed 3-manifolds preserving a volume form $\Omega$.
Following \cite{DyZw17} we study spaces of invariant distributions with values in the bundle of exterior forms whose wavefront set is contained in the dual of the unstable bundle. Our first result computes the dimension of these spaces in terms of the first Betti number of the manifold, the cohomology class $[\iota_{X}\Omega]$ (where $X$ is the infinitesimal generator of the flow) and the helicity. These dimensions coincide with the Pollicott-Ruelle
resonance multiplicities under the assumption of {\it semisimplicity}. We prove various results regarding semisimplicity on 1-forms, including an example showing that it may fail for time changes of hyperbolic geodesic flows.
We also study non null-homologous deformations of contact Anosov flows and we show that there is always
a splitting Pollicott-Ruelle resonance on 1-forms and that semisimplicity persists in this instance.
These results have consequences for the order of vanishing at zero of the Ruelle zeta function. Finally our analysis also incorporates a flat unitary twist in both, the resonant spaces and the Ruelle zeta function.

%We compute the order of vanishing at zero of the Ruelle zeta function twisted with a flat unitary connection for an arbitrary volume preserving Anosov flow on a closed 3-manifold. The final result combines topological information with dynamical information in the form of the winding cycle of the flow.
\end{abstract}

\maketitle

\section{Introduction}

In this paper we study resonant spaces of invariant distributions with values in the bundle of exterior forms for volume preserving Anosov flows on 3-manifolds. One of the main motivations for looking at these spaces is that when a natural restriction is placed on the wave front set of the distributions, their dimensions are related to the Pollicott-Ruelle resonance multiplicities which in turn determine the order of vanishing at zero of the Ruelle zeta function. For the case of contact Anosov flows this analysis was carried out by Dyatlov and Zworski in \cite{DyZw17} and here we show that the transition from ``contact" to ``volume preserving" presents some new features, making the overall picture more involved, partially due to the non-smoothness of the stable plus unstable bundle.

Let $(M,\Omega)$ be a closed 3-manifold equipped with a volume form $\Omega$ and let $\varphi_{t}$ be a volume preserving Anosov flow with infinitesimal generator $X$. If we write the Anosov splitting as $TM=\mathbb{R} X\oplus E_{s}\oplus E_{u}$, then we define the spaces $E_{0}^*$, $E_{s}^*$ and $E_{u}^*$ as the duals of $\mathbb{R} X$, $E_{u}$ and $E_{s}$ respectively. In particular, this means that for each $x\in M$, $E_u^*(x)$ is the annihilator of $\mathbb{R}X(x) \oplus E_u(x)$ and $E_u^*  \subset T^*M$, a closed conic subset.
We denote by $\mathcal{D}'_{E_u^*}(M; \Omega^k)$ the space of distributions with values in the bundle of exterior $k$-forms and with wave front set contained in $E_{u}^*$ (see Section \ref{section:prelim} for background on these notions). The resonant spaces that we are interested in are:
\[\textnormal{Res}_k(0):= \{u \in \mathcal{D}'_{E_u^*}(M; \Omega^k):\; \iota_{X}u=0,\;\iota_{X}du=0\}.\]
The dimensions of the spaces can be considered as {\it geometric multiplicities}. We note that the work \cite{DR17}
studies generalized resonant spaces of forms (at zero) for arbitrary Anosov flows and these have a good cohomology theory (cf. Remark \ref{rem:DR17} for more details and definitions) but in principle these generalized resonant forms are not in the kernel of $\iota_{X}$ and might only be in the kernel of some power of the Lie derivative.

Our first result computes the dimension of these geometric spaces in terms of the first Betti number $b_{1}(M)$ of the manifold $M$ and two natural characteristics of the flow that we now recall.

Since $X$ preserves the volume form $\Omega$, its Lie derivative $\mathcal L_{X}\Omega=0$. Hence the 2-form
$\omega:=\iota_{X}\Omega$ must be closed.

\begin{definition} {\rm We say that $X$ is {\it null-homologous} if the cohomology class $[\omega]=0$, i.e. $\omega$ is exact. For a null-homologous $X$, its {\it helicity} is the number
\[\mathcal H (X):=\int_{M}\tau(X)\Omega,\]
where $\tau$ is any 1-form such that $d\tau=\omega$.

}
\end{definition}

It is easy to check that this definition is independent of the choice of primitive $\tau$. The helicity (also referred to as the {\it asymptotic Hopf invariant}) measures how much in average field lines wrap and coil around one another.
We refer to \cite{AK98} for a complete account of this concept as well as its interpretation as an average self-linking number.

We can now state our first result:

\begin{theorem}\label{thm:volpres} Let $(M,\Omega)$ be a closed 3-manifold with volume form $\Omega$ and let $\varphi_{t}$ be a volume preserving Anosov flow. Then
\begin{enumerate}
\item $\text{\rm dim\,Res}_{0}(0)=\text{\rm dim\,Res}_{2}(0)=1$.
\item If $[\omega]\neq 0$, $\text{\rm dim\,Res}_{1}(0)=b_{1}(M)-1$.
\item If $[\omega]=0$, then
\[\text{\rm dim\,Res}_{1}(0)=\left\{\begin{array}{ll}
b_{1}(M)&\mbox{\rm if}\;\mathcal H (X)\neq 0\\
b_{1}(M)+1&\mbox{\rm if}\;\mathcal H (X)=0.\\
\end{array}\right.\]

\end{enumerate}

\label{thm:dimension}

\end{theorem}

This result generalizes \cite[Proposition 3.1]{DyZw17} as a contact Anosov flow fits into $[\omega]=0$ and $\mathcal H (X)\neq 0$, since in that case we can take $\tau$ to be the contact $1$-form and $\tau(X)=1$. In Section \ref{section:ex} we give some examples to illustrate the various cases in Theorem \ref{thm:dimension}, but we should point out right away that we do not know of any example of a volume preserving Anosov flow with zero helicity.

We note that all the notions involved in Theorem \ref{thm:dimension} are invariant under times changes. Namely, if $f$ is a positive smooth function, the flow of $fX$ is also Anosov and with the {\it same} $E^*_{u}$. Hence the resonant spaces $\textnormal{Res}_k(0)$ are the same for all such flows. Also the notion of being null-homologous or having non-zero helicity is unaffected by time changes.

As mentioned before, the dimensions of $\textnormal{Res}_k(0)$ are important since they are related to the Pollicott-Ruelle resonance multiplicities $m_{k}(0)$.
In general $m_{k}(0)\geq \textnormal{dim\,Res}_k(0)$, and equality holds under the following condition (cf. Lemma \ref{lemma:genres}):

\begin{definition} $X$ or $\varphi_t$ is said to be $k$-semisimple if given $u\in \mathcal{D}'_{E_u^*}(M; \Omega^k)$ with $\iota_{X}u=0$ and $\iota_{X}du\in \textnormal{Res}_k(0)$, then $u\in \textnormal{Res}_k(0)$, i.e. $\iota_{X}du=0$.
\label{definition:semisimple}
\end{definition}

Semisimplicity for $k=0,2$ will be easy to establish, but $1$-semisimplicity does not always hold. In the case of contact Anosov flows, 1-semisimplicity was proved in \cite[Lemma 3.5]{DyZw17}. For general volume preserving Anosov flows the bundle $E_u \oplus E_s$ is only H\"older continuous \cite{FoHa} and thus the $1$-form adapted to the flow, defined to be zero on $E_u \oplus E_s$ and $1$ on the generator $X$ is only H\"older continuous. As a consequence the computations done in \cite[Lemma 3.5]{DyZw17} are no longer viable due to this lack of smoothness.

Our next two results show that the picture for volume preserving Anosov flow is rather more subtle. Let $\mathcal X_{\Omega}$ denote the set of vector fields that preserve $\Omega$ and let $\mathcal X^{0}_{\Omega}\subset \mathcal X_{\Omega}$ denote those which are null-homologous.

\begin{theorem}\label{thm:localresonance} Let $(M,\Omega)$ be a closed 3-manifold with volume form $\Omega$.  Consider a smooth 1-parameter family $X_{\varepsilon}$ of volume preserving Anosov vector fields with $X_0$
1-semisimple.
\begin{enumerate}
\item If $X_{\varepsilon}\in\mathcal X^{0}_{\Omega}$ for every $\varepsilon$ and $\mathcal H(X_0)\neq 0$, then $X_{\varepsilon}$ is 1-semisimple for all $\varepsilon$ sufficiently small.
\item If $X_0$ is not null-homologous, then $X_{\varepsilon}$ is 1-semisimple for all $\varepsilon$ sufficiently small.
%\item There exists a perturbation $X_\varepsilon = X + \varepsilon Y$ of the geodesic flow on a negatively curved surface, such that
\end{enumerate}
%There is a time change of the geodesic flow of a hyperbolic surface which is not 1-semisimple.
For any hyperbolic surface, there is a time change of the geodesic flow which is not $1$-semisimple.
\label{thm:surprise}
\end{theorem}

%\todo[inline]{G: I switched to families as it is more in line with what we prove}

%Consider now the unit sphere bundle $SM$ of a negatively curved surface $M$ and the geodesic vector field $X$. In particular, by Theorem \ref{thm:localresonance} we know that $1$-semisimplicity persists in $\mathcal{X}^0_\Omega$ and near $X$. The next theorem gives us a local picture for what happens near $X$ and away from $\mathcal{X}^0_\Omega$.

Consider now a contact Anosov flow $X$ with contact form $\alpha$ on a closed $3$-manifold $M$. In particular, by Theorem \ref{thm:localresonance} we know that $1$-semisimplicity persists in $\mathcal{X}^0_\Omega$ and near $X$, where $\Omega = -\alpha \wedge d\alpha$. The next theorem gives us a local picture for what happens near $X$ and away from $\mathcal{X}^0_\Omega$.

%To state the theorem, consider $V$ the vector field on $SM$ generating rotations in the vertical fibers. Define the horisontal vector field $H = [V, X]$. Let $\theta$ be a harmonic one form on $M$ and consider $\theta$ also as a function on $SM$ linear in velocities. Let $Y = \theta X + V\theta H$ and consider $X_\varepsilon = X + \varepsilon Y$. One can easily check that $X_\varepsilon \in \mathcal{X}_\Omega$ and is not null-homologous for $\varepsilon \neq 0$.

\begin{theorem}\label{thm:perturbationflow}
	Consider $Y \in \mathcal{X}_\Omega \setminus \mathcal{X}_\Omega^0$. Then for sufficiently small $\varepsilon$, the flow $X_\varepsilon = X + \varepsilon Y$ is $1$-semisimple. Moreover, there is a splitting Pollicott-Ruelle resonance $-i\lambda_\varepsilon = O(\varepsilon^2)$ of $-i\Lapl_{X_\varepsilon}$ acting on $\Omega^1\cap\ker \iota_{X_\varepsilon}$ with $\lambda_\varepsilon < 0$ for $\varepsilon \neq 0$, with Pollicott-Ruelle multiplicity one (see Figure \ref{fig:splitting}). %, splitting to the upper half-plane
\end{theorem}
%{\color{blue} M: You're right, $\lambda_\varepsilon < 0$, it must've stayed there from an older version. I went for $S\Sigma$ since we use $M$ for the $3$-manifold; we should probably stick to this notation later on..}

%\todo[inline]{G: Much better. I agree we should have $S\Sigma$ all throughout the paper.}

\begin{figure}
             \centering
\begin{tikzpicture}%[scale = 0.7, everynode/.style={scale=0.5}]
\tikzset{cross/.style={cross out, draw=black, minimum size=2*(#1-\pgflinewidth), inner sep=0pt, outer sep=0pt},
%default radius will be 1pt. 
cross/.default={1pt}}

%the left graph
      	\draw[thick, ->] (0,-3) -- (0,3) node[above] {\small $\im s$};
		\draw[thick, ->] (-3.5,0) -- (3.5,0) node[below] {\small $\re s$};
      	
      	%\re \lambda = -\frac{1}{2}
      	\draw[thick] (-1.5, -3) node[left] {\small $\re s = -\frac{1}{2}$}-- (-1.5, 3);
		\draw[thick] (-3, 0) node[cross=3pt,rotate=0,red]{} node[above]{\small $-1$};%furthest left
		\draw[thick] (-1.5, 0) node[cross=3pt,rotate=0,red]{}; %\lambda = -1/2
		\draw[thick] (-1.3, 0) node[cross=3pt,rotate=0,blue]{};%small 1
		\draw[thick] (-1.1, 0) node[cross=3pt,rotate=0,blue]{};%small 2
		\draw[thick] (-1.7, 0) node[cross=3pt,rotate=0,blue]{};%small 1'
		\draw[thick] (-1.9, 0) node[cross=3pt,rotate=0,blue]{};%small 2'
		
		\draw[thick] (-1.5, 1) node[cross=3pt,rotate=0,green]{}; %\lambda = -1/2 + i
    	\draw[thick] (-1.5, 1.5) node[cross=3pt,rotate=0,green]{}; %\lambda = -1/2 + i  	
    	\draw[thick] (-1.5, 2.3) node[cross=3pt,rotate=0,green]{}; %\lambda = -1/2 + i
    	\draw[thick] (-1.5, -1) node[cross=3pt,rotate=0,green]{}; %\lambda = -1/2 + i
    	\draw[thick] (-1.5, -1.5) node[cross=3pt,rotate=0,green]{}; %\lambda = -1/2 + i  	
    	\draw[thick] (-1.5, -2.3) node[cross=3pt,rotate=0,green]{}; %\lambda = -1/2 + i
    	
      	%\re \lambda = \frac{1}{2}
      	\draw[thick] (1.5, -3) node[right] {\small $\re s = \frac{1}{2}$} -- (1.5, 3);
		\draw[thick] (3, 0) node[cross=3pt,rotate=0,blue]{} node[above]{\small $1$};%furthest right
		\draw[thick] (1.5, 0) node[cross=3pt,rotate=0,red]{};%special \lambda = 1/2
		\draw[thick] (1.3, 0) node[cross=3pt,rotate=0,blue]{};%small 1
		\draw[thick] (1.1, 0) node[cross=3pt,rotate=0,blue]{};%small 2
		\draw[thick] (1.7, 0) node[cross=3pt,rotate=0,blue]{};%small 1'
		\draw[thick] (1.9, 0) node[cross=3pt,rotate=0,blue]{};%small 2'
		
		\draw[thick] (1.5, 1) node[cross=3pt,rotate=0,green]{}; %\lambda = -1/2 + i
    	\draw[thick] (1.5, 1.5) node[cross=3pt,rotate=0,green]{}; %\lambda = -1/2 + i  	
    	\draw[thick] (1.5, 2.3) node[cross=3pt,rotate=0,green]{}; %\lambda = -1/2 + i
    	\draw[thick] (1.5, -1) node[cross=3pt,rotate=0,green]{}; %\lambda = -1/2 + i
    	\draw[thick] (1.5, -1.5) node[cross=3pt,rotate=0,green]{}; %\lambda = -1/2 + i  	
    	\draw[thick] (1.5, -2.3) node[cross=3pt,rotate=0,green]{}; %\lambda = -1/2 + i

		\draw[thick] (0, 0) node[cross=3pt,rotate=0,red]{} node[below left]{\small $0$};% node[above]{\tiny $m(0) = 2b_1$}; %\lambda = 0

		\draw[dashed] (0,0) circle [radius=1]; %dashed circle

		\draw[thick] (3.5, 2) node[cross=3pt,rotate=0,green]{} node[right]{\tiny $= \mathrm{large\, eigenvalues}$};
		\draw[thick] (3.5, 1.7) node[cross=3pt,rotate=0,blue]{} node[right]{\tiny $= \mathrm{small\, eigenvalues}$};
		\draw[thick] (3.5, 1.4) node[cross=3pt,rotate=0,red]{} node[right]{\tiny $= \mathrm{special\, points}$};

		%the right graph
		
		\draw[thick, ->] (5.5,0) -- (12.5,0) node[below] {\small $\re s$};
      	\draw[thick, ->] (9,-3) -- (9,3) node[above] {\small $\im s$};
		
		\draw[dashed] (9, 0) circle [radius=1];
		\draw[thick] (9, 0) node[cross=3pt,rotate=0,red]{} node[below left]{\small $0$};
		\draw[thick] (9.6, 0) node[cross=3pt,rotate=0,purple]{} node[below]{\small $-\lambda_\varepsilon$};
		
		\draw[thick] (10, 2) node[cross=3pt,rotate=0,purple]{} node[right]{\tiny $= \mathrm{splitting\, resonance}$};
\end{tikzpicture}
             \caption{\small On the left: resonance spectrum of $\Lapl_X$ acting on $\Omega^1(S\Sigma)$, for a closed hyperbolic surface $\Sigma$. According to \cite{GHW18, DFG15} and Remark \ref{rem:speconeformshyperbolic} below, the green crosses correspond to (large) eigenvalues $\mu \geq \frac{1}{4}$ of $-\Delta_\Sigma$, the blue ones correspond to (small) eigenvalues $\mu \leq \frac{1}{4}$ and the red ones are ``special". On the right: resonance spectrum of $\Lapl_{X_\varepsilon}$ acting on $\Omega^1(S\Sigma)$ and the splitting resonance, according to Theorem \ref{thm:perturbationflow}. We remark that the resonances in the rest of this paper will often be given by $\lambda = is$, i.e. obtained by a rotation of $\pi/2$ from this picture.}
             \label{fig:splitting}
\end{figure}
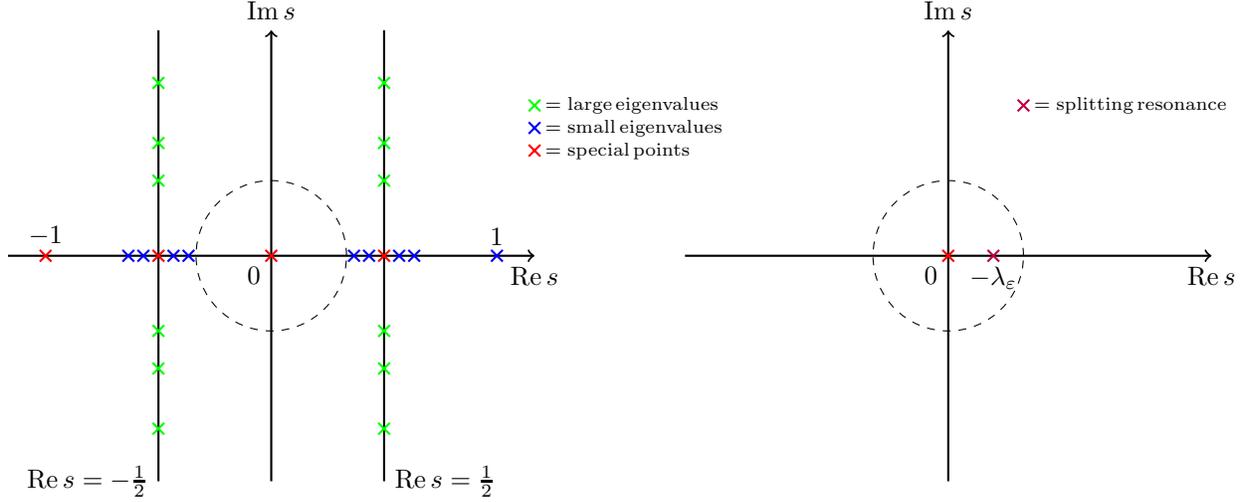
\subsection{Ruelle zeta function} We denote the set of primitive closed orbits of $X$ by $\mathcal{G}_0$ (i.e. the ones that are not powers of a closed orbit in $M$); the period of $\gamma \in \mathcal{G}_0$ is denoted by $l_\gamma$. The Ruelle zeta function is defined as:
\begin{align}\label{DefRuelleZeta}
\zeta (s):= \prod_{\gamma \in \mathcal{G}_0}{\big(1 - e^{-sl_\gamma}\big)}.
\end{align}
The infinite product converges for $\text{\rm Re}\, s\gg1$ and its meromorphic continuation to all $\mathbb{C}$ was first established in \cite{GiLiPo} in full generality and subsequently in \cite{DyZw16}, where a microlocal approach was employed (cf. \cite{Pnotes} for a survey of dynamical zeta functions). Moreover, it was shown in \cite{DyZw16} that there is a factorisation (assuming that $E_{s}$ and $E_{u}$ are orientable)
\begin{align}\label{factorisationruelle}
\zeta(s) = \frac{\zeta_{1}(s)}{\zeta_{0}(s) \zeta_{2}(s)},
\end{align}
where $\zeta_{k}(s)$ is an entire function with the order of vanishing at each $s \in \mathbb{C}$ equal to $m_{k}(is)$, for $k = 0, 1, 2$. Here $m_{k}(\lambda)$ is the Pollicott-Ruelle resonance multiplicity (see Section \ref{section:prelim} for more details). Hence the order of vanishing of $\zeta$ at $s=0$ is determined by
$m(0):=m_{1}(0)-m_{0}(0)-m_{2}(0)$. Using this and Theorem \ref{thm:dimension} we derive

\begin{corollary}
\label{corollary:zeta} Let $(M,\Omega)$ be a closed 3-manifold with a volume preserving Anosov flow $\varphi_t$ whose stable and unstable bundles are orientable. Then
\[s^{n(M,X)}\zeta(s)\]
is holomorphic close to zero, where
\begin{align*}
n(M,X)&=3-b_{1}(M),\;\text{if}\;\;[\omega]\neq 0,\\
n(M,X)&=2-b_{1}(M),\;\text{if}\;\;[\omega]= 0,\;and\;\,\mathcal H (X)\neq 0,\\
n(M,X)&=1-b_{1}(M),\;\text{if}\;\;[\omega]=0\;and\;\,\mathcal H (X)= 0.\\
\end{align*}
Moreover, if $\varphi_t$ is 1-semisimple, then $s^{n(M,X)}\zeta(s)|_{s=0}\neq 0$.
\end{corollary}

The Ruelle zeta function for the suspension of a hyperbolic toral automorphism $A \in SL(2, \mathbb{Z})$ is equal to $\zeta(s) = \frac{(e^{-s} - \lambda)(e^{-s} - \frac{1}{\lambda})}{(e^{-s} - 1)^2}$, where $\lambda$ and $\frac{1}{\lambda}$ are eigenvalues of $A$. This has a pole of order $2$ at $s=0$ which of course matches the computation in Corollary \ref{corollary:zeta} since $b_{1}(M)=1$. However the corollary asserts that {\it any} other volume preserving non null-homologous Anosov flow on $M$ will have $\zeta$ with the same behaviour at $s=0$ since 1-semisimplicity holds trivially given that $\textnormal{Res}_1(0)$ is zero dimensional.
An interesting class of Anosov flows with $[\omega]\neq 0$ is given in \cite{BL18}. These examples have a transverse torus, but they are not conjugate to suspensions. We do not know if they are 1-semisimple.

Magnetic flows are also examples to which the previous corollary applies.
They are null-homologous (cf. Section \ref{section:ex}), but they are generically {\it not} contact, (cf. \cite{DP})  hence they were not covered by the main result 
in \cite{DyZw17}.  In this setting, magnetic flows can be described by a vector field of the form $X+(\lambda\circ\pi) V$, where $X$ is the geodesic vector field, $V$ the vertical vector field of the circle fibration $\pi:S\Sigma\to \Sigma$, and $\lambda\in C^{\infty}(\Sigma)$ (here $M=S\Sigma$, the unit circle bundle of the orientable surface $\Sigma$). They are volume preserving since $X$ and $V$ preserve the canonical volume form. Suppose the geodesic flow is Anosov. Thanks to item (1) in Theorem \ref{thm:surprise}, if $\lambda$ is small enough, the magnetic flows remain Anosov and 1-semisimple and hence the order of vanishing of the zeta function at zero is the same as for Anosov geodesic flows, namely $-\chi(\Sigma)$.

The last statement in Theorem \ref{thm:surprise} and Theorem \ref{thm:perturbationflow} have consequences for the zeta function. The failure of 1-semisimplicity means that $m_{1}(0)\geq b_{1}(M)+1$ and hence the order of vanishing at zero of the zeta function is {\it strictly bigger} than that of the geodesic flow case. Hence time changes can a priori produce alterations in the properties of $\zeta$ near zero. Similarly the cohomology class $[\omega]$ can also produce alterations.
For the particular construction of Theorem \ref{thm:surprise} we do not know the precise order of vanishing at zero.

\begin{corollary}\label{cor:zeta'}
	The order of vanishing of the zeta function $\zeta_{X_\varepsilon}(s)$ of the flow $X_\varepsilon$ from Theorem \ref{thm:perturbationflow} at zero, for $\varepsilon \neq 0$, is equal to $b_{1}(M)-3$. Moreover, for the time-change $fX$ of the geodesic flow on the hyperbolic surface constructed in Theorem \ref{thm:localresonance}, the order of vanishing is greater than or equal to $-\chi(\Sigma)+ 1$.
\end{corollary}

\subsection{Flat unitary twists}It is possible (and natural)  to introduce a unitary twist in the discussion above.
Consider $(M,\Omega)$ a closed 3-manifold with volume form $\Omega$ and $X$ a volume preserving Anosov vector field. Let $\mathcal{E}$ a Hermitian vector bundle over $M$, equipped with a unitary connection $A$. 
We consider $\mathcal{D}'_{E_u^*}(M; \Omega^k\otimes \mathcal E)$ the space of distributions with values in the bundle of $\mathcal E$-valued exterior $k$-forms and with wave front set contained in $E_{u}^*$.
We replace the exterior differential $d$ by the covariant derivative $d_{A}$ (induced by the connection $A$) acting on $\mathcal E$-valued
differential forms. 
Thus we can define resonant spaces
\[\textnormal{Res}_{k,A}(0):= \{u \in \mathcal{D}'_{E_u^*}(M; \Omega^k\otimes\mathcal E):\; \iota_{X}u=0,\;\iota_{X}d_{A}u=0\}.\]
We shall compute the dimensions of these spaces in analogy to Theorem \ref{thm:dimension} under the assumption that $A$ is {\it flat and unitary}, i.e. $d_{A}^2=0$ and $d_A$ is compatible with the Hermitian inner product on $\E$. Recall that flat unitary connections are in 1-1 correspondence with
representations of $\pi_{1}(M)$ into the unitary group.
Under this condition, one can define twisted Betti numbers $b_{i}(M,\mathcal E)$ in the standard way (we note that these numbers may depend on $A$).
The upshot is a theorem similar to Theorem \ref{thm:dimension} where the Betti numbers $b_{i}(M)$
are replaced by $b_{i}(M,\mathcal E)$, cf. Theorem \ref{thm:dimensionA} for the full statement.
With this information in hand we can study a {\it twisted Ruelle zeta function}:

\begin{align}\label{DefRuelleZetaA}
\zeta_{A} (s):= \prod_{\gamma \in \mathcal{G}_0} \det{\big(\id - \alpha_\gamma e^{-sl_\gamma}\big)}.
\end{align}

Here, given a point $x_0$ on $\gamma \in \mathcal{G}_0$, we denote by $\alpha_\gamma$ the parallel transport map (i.e. an element of the holonomy group) along the loop determined by $\gamma$. It is easy to check that the product is independent of the choice of $x_0$ on $\gamma$, as this amounts to conjugating $\alpha_\gamma$ by a linear map. Note that if $\mathcal{E} = M \times \mathbb{C}$ and $d_A = d$, the expression in \eqref{DefRuelleZetaA} reduces to that in \eqref{DefRuelleZeta}. %Sometimes we will use $\mathcal{G}$ to denote all closed orbits in $M$ and denote the minimal period of $\gamma \in \mathcal{G}$ by $l_{\gamma}^\#$.

If the connection $A$ is flat, we recover the definition of the twisted Ruelle zeta function considered by Fried \cite{F1}; also by Adachi and Sunada \cite{A88, AS87}, who called functions of this type $\emph{L-functions}$ in analogy with number theory.  Fried conjectured that the coefficient at zero of $\zeta_{A}$ for an acyclic connection (i.e. one that has vanishing Betti numbers) is related to the analytic torsion, but proved it only for hyperbolic manifolds. For recent progress on this conjecture and more information, see \cite{DGRS18,Shen,Z}.

The notion of semisimplicity extends naturally to the twisted case (just replace $d$ by $d_{A}$ in Definition \ref{definition:semisimple}). In that case we will say a flow $\varphi_t$ or $X$ is $1$-semisimple with respect to $d_A$. Putting everything together we shall derive the following corollary:

\begin{corollary}\label{corollary:zetaA} Let $(M,\Omega)$ be a closed 3-manifold with a volume preserving Anosov flow $\varphi_t$ whose stable and unstable bundles are orientable.  Let $\mathcal E$ be a Hermitian vector bundle equipped with a unitary flat connection $A$. Then
\[s^{n(M,X,A)}\zeta_{A}(s)\]
is holomorphic close to zero, where
\begin{align*}
n(M,X,A)&=3b_{0}(M,\mathcal E)-b_{1}(M,\mathcal E),\;\text{if}\;\;[\omega]\neq 0,\\
n(M,X,A)&=2b_{0}(M,\mathcal E)-b_{1}(M,\mathcal E),\;\text{if}\;\;[\omega]= 0,\;and\;\,\mathcal H (X)\neq 0,\\
n(M,X,A)&=b_{0}(M,\mathcal E)-b_{1}(M,\mathcal E),\;\text{if}\;\;[\omega]=0\;and\;\,\mathcal H (X)= 0.\\
\end{align*}
Moreover, if $X$ is $1$-semisimple with respect to $d_A$, then $s^{n(M,X,A)}\zeta_{A}(s)|_{s=0}\neq 0$.
\end{corollary}

%{\color{blue} M: should we maybe introduce a notation like $(\varphi_t, A)$ or $(X, A)$ is $1$-semisimple, to avoid clashes to $\varphi_t$ or $X$ being $1$-semisimple? Also in Corollary 1.9, don't we want both null-homologous and non-zero helicity?}

%\todo[inline]{G: Good point. We could say $\varphi_t$ 1-semisimple with respect to $d_{A}$.
%Yes, in 1.9 we also need non-zero helicity. We do know 1-semisimplicity wrt to $d_{A}$ for geodesic flows, don't we? Does regular 1-semisimplicity imply $d_{A}$-semisimplcity?
%}
%{\color{blue} M: We know $1$-semisimplicity w.r.t. $d_A$ holds for flows that have the associated one form smooth -- so for contact and suspensions, that's what we did in the previous version of the paper. That's a good question, I don't know the answer: one can easily get that the ``part corresponding to parallel sections is semisimple", if $X$ is $1$-semisimple, but I'm not sure about the rest. Please check again Cor. 1.9 below.}

A particular instance of the corollary arises when we consider $A$ to be the pull-back of a flat connection on a surface $\Sigma$. In this case it is easy to check that (cf. Lemma \ref{lemma:twistedeuler}):
\[2b_0(M, \mathcal{E}) - b_1(M, \mathcal{E})=\text{rank}(\mathcal E)\,\chi(\Sigma).\]
%\[b_{0}(M;\mathcal{E})=\text{dimension\;of\;parallel\;sections\;on}\;\Sigma.\]
Thus:

\begin{corollary}\label{thm:pullback}
Let $\mathcal{E}$ be a Hermitian vector bundle over an oriented closed  Riemannian surface $(\Sigma,g)$, equipped with a unitary flat connection $A$. We consider $M = S\Sigma$ with footpoint map $\pi$ and any Anosov flow, $1$-semisimple with respect to $d_{\pi^*A}$, null-homologous with non-zero helicity, preserving the volume form of $S\Sigma$. We consider the pullback bundle $\pi^*\mathcal E$ with the pull-back connection $\pi^*A$. Then in a neighbourhood of zero we have $s^{\rk(\mathcal{E}) \cdot \chi(\Sigma)} \cdot \zeta_{\pi^*A}(s)$ holomorphic, such that
\begin{equation*}
    s^{\rk(\mathcal{E}) \cdot \chi(\Sigma)} \cdot \zeta_{\pi^*A}(s) |_{s = 0} \neq 0.
\end{equation*}
%the order of the pole of $\zeta_{R, A}$ at zero is equal to %$\rk(\mathcal{E}) \cdot \chi(\Sigma)$.
\end{corollary}

We remark that Corollary \ref{thm:pullback} applies in particular to contact flows, since for those $1$-semisimplicity holds with respect to any flat and unitary $d_A$.

\bigskip

This paper is organised as follows. Section \ref{section:prelim} gives preliminary information, recalls the Pollicott-Ruelle resonances and proves some necessary lemmas. In Section \ref{section:zeta} we recall the factorisation of the twisted zeta function in terms of some traces of operators on $\mathcal E$-valued $k$-forms. In Section \ref{section:resonances}, we compute the dimension of the resonant spaces $\textnormal{Res}_{k,A}(0)$ and obtain Theorem \ref{thm:dimension} as a particular case.
Corollary \ref{corollary:zetaA} is also proved in this section.
Section \ref{section:ex} gives examples and develops material needed for the study of time changes.
Section \ref{section:perturbations} discusses perturbations and proves the main result needed for items (1) and (2) in Theorem \ref{thm:surprise}. Theorem \ref{thm:perturbationflow} is proved in Section \ref{section:closetocontact}. Finally, Section \ref{section:timechange} exhibits a time change of the geodesic flow of a hyperbolic surface for which 1-semisimplicity fails, thus completing the proof of Theorem \ref{thm:surprise}.

%\todo[inline]{G: Ideas in the proofs, non-smoothness of $\alpha$}
%{\color{blue}M: what shall we add here?}

\medskip

\noindent {\bf Acknowledgements.} We would like to thank Semyon Dyatlov, Colin Guillarmou, Charles Hadfield and Maciej Zworski for very helpful conversations. MC would like to thank Maciej Zworski for the warm hospitality during his visit at UC Berkeley. We would also like to thank the referee for helpful remarks and corrections.

MC was supported by Trinity College (Cambridge),  the Max-Planck Institute (Bonn) and the European Research Council (ERC) under the European Union’s Horizon 2020 research and innovation programme
(grant agreement No. 725967).  GPP was supported by EPSRC grant EP/R001898/1. 

This work was partially supported by the National Science Foundation under Grant No. DMS-1440140 while the authors were in residence at the Mathematical Sciences Research Institute in Berkeley, California, during the Fall 2019 semester.

\section{Preliminary results}\label{section:prelim} In this section we review the necessary tools to prove the results stated in the introduction. In particular, we recall the Pollicott-Ruelle resonances and put forward some preparatory lemmas.

\subsection{Microlocal analysis}
Here we outline the microlocal tools necessary for our proofs. For more information on distribution spaces and properties of wavefront sets see \cite[Chapter 7]{GrSj94} or \cite[Chapters 6,8]{HoI+II} and for more about pseudodifferential operators \cite[Chapter 3]{GrSj94} or \cite[Chapter 18]{HoIII+IV}. 

Let $M$ be a closed manifold and $\mathcal{E}$ a smooth complex vector bundle. We consider the space of infinitely differentiable smooth sections and the space of distributional sections, respectively:
\[C^\infty(M; \mathcal{E}) \quad \text{ and } \quad \mathcal{D}'(M; \mathcal{E}).\]
We recall the notion of the \emph{wavefront set} of a distribution, which keeps track of the directional singularities. Given $u \in \mathcal{D}'(\mathbb{R}^n)$, then $(x, \xi) \not\in WF(u) \subset T^*\mathbb{R}^n \setminus 0 = \mathbb{R}^n \times \big(\mathbb{R}^n \setminus 0\big)$ if there exists $\varphi \in C_0^\infty(\mathbb{R}^n)$ with $\varphi(x) \neq 0$ and an open conical neighbourhood $U$ of $\xi$, such that
\[|\widehat{\varphi u}(\eta)| = O\big(\langle{\eta}\rangle^{-\infty}\big)\]
for $\eta \in U$. Here we denote $\langle{\eta}\rangle = (1 + |\eta|^2)^\frac{1}{2}$ and by $O(\langle{\eta}\rangle^{-\infty})$ we mean an expression bounded by $C_N \langle{\eta}\rangle^{-N}$ for every $N$. A vector valued distribution $u \in \mathcal{D}'(\mathbb{R}^n; \mathbb{R}^m)$ for some $m \in \mathbb{N}$ may be identified with a vector $u = (u_1, \dotso, u_m)$ with $u_i \in \mathcal{D}'(\mathbb{R}^n)$. Then
\[WF(u) := \cup_{i = 1}^m WF(u_i).\]
It is standard that these definitions are coordinate invariant, so for $u \in \mathcal{D}'(M; \mathcal{E})$ we have
\[WF(u) \subset T^*M\setminus 0.\]
It is moreover true that for any pseudodifferential operator $A$, we have
\[WF(Au) \subset WF(A) \cap WF(u) \subset WF(u),\]
a fact that will be used later on. Then, we introduce for a closed conic set $\Gamma \subset T^*M \setminus 0$ the space
\[\mathcal{D}'_{\Gamma}(M; \mathcal{E}) = \{u \in \mathcal{D}'(M; \mathcal{E}) \mid WF(u) \subset \Gamma\}.\]
Note that by the above relation on wavefront sets, the spaces $\mathcal{D}'_\Gamma(M; \mathcal{E})$ are invariant under the action of pseudodifferential operators.
\subsection{Pollicott-Ruelle resonances}
% without singularities
Let us now quickly recall the microlocal approach to Pollicott-Ruelle resonances, as in \cite{DyZw17}. Let $M$ be a compact smooth manifold without boundary and $X$ be a smooth vector field. We assume that the flow $\varphi_{t}$ of $X$ is Anosov, i.e that there is a splitting of the tangent space
\[T_x M = \mathbb{R} X(x) \oplus E_u(x) \oplus E_s(x)\]
for each $x \in M$, where $E_u(x)$ and $E_s(x)$ depend continuously on $x$ and are invariant under the flow and moreover, that for some constants $C, \nu > 0$ and a fixed metric on $M$,
\begin{align*}
|d\varphi_{t}(x) \cdot v| \leq Ce^{-\nu |t|} \cdot |v|, \quad \begin{cases} t \geq 0,\quad v \in E_s(x)\\
t \leq 0,\quad v \in E_u(x).
\end{cases}
\end{align*}

We call $E_s(x)$ the \emph{stable} bundle or direction and $E_u(x)$ the \emph{unstable} bundle or direction. It is a well-known fact that the geodesic flow on the unit tangent bundle $M = SN$ for $N$ with negative sectional curvature is Anosov. 

Let us define the spaces $E_0^*(x), E_u^*(x), E_s^*(x)$ as the duals of $E_0(x) := \mathbb{R} X(x), E_s(x), E_u(x)$ respectively. Explicitly, $E_u^*(x)$ is the annihilator of $\mathbb{R}X(x) \oplus E_u(x)$, $E_s^*(x)$ is the annihilator of $\mathbb{R}X(x) \oplus E_s(x)$ and $E_0^*(x)$ is the annihilator of $E_{s}(x) \oplus E_u(x)$.
The continuous vector bundle $E_u^* := \cup_{x \in M} E_u^*(x) \subset T^*M$ is a closed conic subset.

%we don't need unitary+hermitian here.
Let us consider a complex vector bundle $\mathcal{E}$ over $M$, equipped with a connection $A$ (which defines the covariant derivative $d_{A}$) and a smooth potential $\Phi$ (section of the endomorphism bundle of $\mathcal{E}$). This defines a first order operator
\begin{align}\label{eq:op}
P = -i\iota_{X}d_{A} + \Phi
\end{align}
acting on sections of $\mathcal{E}$, denoted by $C^\infty(M; \mathcal{E})$. Later on we will dispense with $\Phi$, but for the moment it can be included without trouble.

For $\lambda \in \mathbb{C}$ with sufficiently large $\im{\lambda} > C_0 > 0$, we have the integral
\begin{align}\label{eq:laplace}
R(\lambda) := i \int_0^\infty e^{i \lambda t} e^{-i t P}dt: L^2(M; \mathcal{E}) \to L^2(M; \mathcal{E})
\end{align}
%\todo[inline]{M: How is $e^{itP}$ defined exactly?}
converges and defines a bounded operator, holomorphic in $\lambda$ and moreover, $R(\lambda) = (P - \lambda)^{-1}$ on $L^2$. The propagator $e^{itP}$ is defined by solving the appropriate first order PDE and the constant $C_0$ depends on $P$.
%(here $e^{itP}$ is unitary and defined by the spectral theorem for unbounded operators)

Faure and Sj\"ostrand \cite{FaSj11} (see also \cite{DyZw16}) proved that the operator $R(\lambda)$ has a meromorphic extension to the entire complex plane %(see \cite{GiLiPo, DyZw16} and references therein)
\begin{align}
    R(\lambda): C^\infty(M; \mathcal{E}) \to \mathcal{D}'(M; \mathcal{E})
\end{align}
for $\lambda \in \mathbb{C}$ and the poles of this continuation are the \emph{Pollicott-Ruelle resonances}.

%We briefly review the construction of this meromorphic extension. We may define the Pollicott-Ruelle resonances as eigenvalues of $P$ acting on a specially constructed \emph{anisotropic Sobolev space} $\mathcal{H}_C$ lying outside the essential spectrum (here $C>0$ is a suitable ``weight" that controls the size of the essential spectrum). In more detail, these spaces are constructed using ``escape functions" (or Lyapunov functions) on the cotangent space, with suitable asymptotic behaviour near the stable/unstable directions. These functions are then used to define a pseudodifferential operator to conjugate $P$ (c.f. Egorov's theorem) and obtain a suitable operator with a discrete spectrum on any previously fixed compact set. See \cite{FaSj11} for more details.
%{\color{blue} M: we could probably erase this paragraph now, but let's keep it for now}

%Another viewpoint (??) is to prove that $P + \lambda$ is a Fredholm operator for all $\lambda$ and to using the analytic Fredholm theory to show the meromorphic continuation of $R(\lambda)$. To prove this Fredholmness one proves some functional-analytic estimates (c.f. Dyatlov talk).

We proceed to define the multiplicity of a Pollicott-Ruelle resonance $\lambda_0$. By definition, there is a Laurent expansion of $R(\lambda)$ at $\lambda_0$ (cf. \cite[Appendix C]{DyZwbook})% or \cite[Proposition 3.1]{DyZw16}), given by% equation (2.10) of \cite{DyZw17} (see also \cite[Proposition 3.1]{DyZw16}):
\begin{align}\label{expansion}
    R(\lambda) = R_H(\lambda) - \sum_{j = 1}^{J(\lambda_0)} \frac{(P - \lambda_0)^{j - 1} \Pi}{(\lambda - \lambda_0)^j}, \quad \Pi, R_H(\lambda): \mathcal{D}'_{E_u^*}(M; \mathcal{E}) \to \mathcal{D}'_{E_u^*}(M; \mathcal{E})
\end{align}
where $R_H(\lambda)$ is the holomorphic part at $\lambda_0$ and $\Pi = \Pi_{\lambda_0}$ is a finite rank projector given by
\begin{equation}\label{eq:projector}
\Pi_{\lambda_0} = \frac{1}{2\pi i} \oint_{\lambda_0} (\lambda - P)^{-1} d\lambda.
\end{equation}
Here, the integral is along a small closed loop around $\lambda_0$ and it can be easily checked that $\Pi_{\lambda_0}^2 = \Pi_{\lambda_0}$, $[\Pi_{\lambda_0}, P] = 0$. The fact that $R_H(\lambda)$ and $\Pi$ can be extended to continuous operators on $\mathcal{D}'_{E_u^*}$ follows from the restrictions on the wave front sets given in \cite[Proposition 3.3]{DyZw16} and \cite[Theorem 7.8]{GrSj94}. The \emph{Pollicott-Ruelle multiplicity} of $\lambda_0$, denoted by $m_P(\lambda_0)$ is defined as the dimension of the range of $\Pi_{\lambda_0}$. 

By applying $P - \lambda$ to \eqref{expansion}, we obtain $(P-\lambda_0)^{J(\lambda_0)} \Pi_{\lambda_0} = 0$ and so $\ran \Pi_{\lambda_0} \subset \ker (P - \lambda_0)^{J(\lambda_0)}$. The elements of $\ran \Pi_{\lambda_0}$ are called \emph{generalised resonant states} and we will denote, for $j \in \mathbb{N}$
%\subsection{Generalised resonant states.}\label{subsec:genres} As in the previous section, we consider $X, M, \E$ and a suitable first order operator $P$. We call $u \in \mathcal{D}'_{E_u^*}(M; \E)$ a \emph{generalised resonant state at} $\lambda_0 \in \mathbb{C}$ if $(P - \lambda_0)^j u = 0$ for some positive integer $j$ and we write
\begin{equation}
	\res^{(j)}_{P}(\lambda_0) = \{u \in \mathcal{D}'_{E_u^*}(M; \E) : (P - \lambda_0)^j u = 0\}.
\label{eq:resj}
\end{equation}
We also write
\[\res_{P}(\lambda_0) = \{u \in \mathcal{D}'_{E_u^*}(M; \E) : (P - \lambda_0)^{J(\lambda_0)} u = 0\}.\]
%Here recall $J(\lambda_0)$ comes from the expansion \eqref{expansion}. Recall also the definition of the residue $\Pi_{\lambda_0}$
%\[\Pi_{\lambda_0} = \frac{1}{2\pi i} \oint_{\lambda_0} (\lambda - P)^{-1} d\lambda\]
%and extends to $\mathcal{D}'_{E_u^*}(M)$ due to a wavefront set condition. 
%We also recall that $(P-\lambda_0)^{J(\lambda_0)} \Pi_{\lambda_0} = 0$, which follows from applying $P - \lambda$ to the expansion \eqref{expansion}.
In fact, we may show that $\res_P(\lambda_0)$ is equal to the range of $\Pi_{\lambda_0}$ and we may think of $J(\lambda_0)$ as the size of the largest Jordan block. 
%a generalised resonant state at $\lambda_0$, 
\begin{lemma}\label{lemma:genres}
	Let $u \in \mathcal{D}'_{E_u^*}(M; \E)$ be such that $(P - \lambda_0)^{j_0} u = 0$ with $j_0 \in \mathbb{N}_0$ the minimal such number. Then $j_0 \leq J(\lambda_0)$, $\Pi_{\lambda_0} u = u$ and $\ker (P - \lambda_0)^{J(\lambda_0)} = \ran \Pi_{\lambda_0}$. 
\end{lemma}
\begin{proof}
	Assume that $j_0 > J(\lambda_0)$ for the sake of contradiction. Since Sobolev spaces filter out $\mathcal{D}'(M; \E)$, there is an $s > 0$ such that $u \in H^{-s}(M; \E)$. Recalling the definition of the anisotropic space $\mathcal{H}_{rG}(M; \E)$ for $r > 0$ (see \eqref{eq:defanisotropicspace} below), we get%We now use the theory of anisotropic Sobolev space (see Section \ref{section:perturbations} below) and Lemma \ref{lemma:anisotropic} to see
	\[\mathcal{D}'_{E_u^*}(M; \E) \cap H^{-r}(M; \E) \subset \mathcal{H}_{rG}(M; \E)\]
since $\mathcal{H}_{rG}$ is microlocally equivalent to $H^{-r}$ near $E_u^*$. Therefore $u \in \mathcal{H}_{rG}(M; \E)$ for $r > s$ and by Lemma \ref{lemma:anisotropic} below $(P - \lambda)^{-1}: \mathcal{H}_{rG}(M; \E) \to \mathcal{H}_{rG}(M; \E)$ is meromorphic near $\lambda_0$ for $r \gg s$.

Let us set $v := (P - \lambda_0)^{j_0 - 1}u$. Then $(P - \lambda)^{-1} v = (\lambda_0 - \lambda)^{-1} v$ and by applying \eqref{eq:projector} to $v$ we get $\Pi_{\lambda_0} v= v$. Note that \eqref{expansion} also implies $(P - \lambda_0)^{J(\lambda_0)} \Pi_{\lambda_0} t = 0$ for all $t \in \mathcal{H}_{rG}$. But all this implies
\begin{equation}\label{eq:commuteproj}
	(P - \lambda_0)^{j_0 - 1}u = \Pi _{\lambda_0}(P - \lambda_0)^{j_0 - 1} u = (P - \lambda_0)^{j_0 - 1} \Pi_{\lambda_0} u = 0.
\end{equation}
This contradicts the minimality of $j_0$ and proves the first claim.

For the second claim, take some $u \in \res^{(j_0)}_{P}(\lambda_0)$ and use induction on $j_0$. Note that the first two equalities of \eqref{eq:commuteproj} show $\Pi_{\lambda_0} u = u$ for $j_0 = 1$ and more generally that
\[(P - \lambda_0)^{j_0 - 1}(\Pi_{\lambda_0}u - u) = 0.\]
The fact that $\Pi_{\lambda_0}$ is a projector and the induction hypothesis show $\Pi_{\lambda_0}u = u$, proving the claim.

Lastly, if $u \in \ran \Pi_{\lambda_0}$ then $\Pi_{\lambda_0} u = u$ and so $(P - \lambda_0)^{J(\lambda_0)}u = 0$ by \eqref{expansion}, which together with the previous paragraph shows $\ker (P - \lambda_0)^{J(\lambda_0)} \cap \mathcal{D}'_{E_u^*}(M; \E) = \ran \Pi_{\lambda_0}$.
\end{proof}

\begin{rem}\rm  Generalized resonant spaces of forms (at zero) have a good cohomology theory, cf. \cite[Theorem 2.1]{DR17}. We emphasize that here we study resonant spaces at zero with $j=1$ in \eqref{eq:resj} and such that the elements are in the kernel of $\iota_{X}$, as well as conditions under which there are no Jordan blocks.
\label{rem:DR17}
\end{rem}

%\todo[inline]{G:Is this enough/OK? Someone may complaint about our untidy notation for the zoology of resonant spaces...But let's leave it as it is for the time being.}

%We now briefly discuss a kind of orthogonality of generalised resonances. To this end, let $\mu_1 \neq \mu_2$ be two poles of $(P - \lambda)^{-1}$ and define the projectors for $i = 1, 2$
%\[\Pi_i := \frac{1}{2\pi i} \oint_{\mu_i} (\lambda - P)^{-1} d\lambda\]
%%Here $\gamma_i$ are small contours around $\mu_i$ for $i = 1, 2$. 
%We consider the composition of these projections as maps $\mathcal{H}_{rG}(M; \E)$ for some large $r$ and write
%\begin{align}
%	\Pi_2 \Pi_1 &= \frac{1}{(2\pi i)^2} \oint_{\mu_2} \oint_{\mu_1} (\lambda_1 - P)^{-1} (\lambda_2 - P)^{-1} d\lambda_1 d\lambda_2\\
%	 &= \frac{1}{(2\pi i)^2} \oint_{\mu_2} \oint_{\mu_1} \frac{(\lambda_1 - P)^{-1} - (\lambda_2 - P)^{-1}}{\lambda_2 - \lambda_1} d\lambda_1 d\lambda_2\\
%	 &= 0 = \Pi_1 \Pi_2
%\end{align}
%Here we used Cauchy's theorem so that $0 = \oint_{\mu_1} \frac{1}{\lambda_2 - \lambda_1} d\lambda_1 = \oint_{\mu_2} \frac{1}{\lambda_2 - \lambda_1} d\lambda_2$, and the resolvent identity. Therefore we have the ``disjointness" of $\Pi_1$ and $\Pi_2$, which serves as an orthogonality statement for the generalised eigenfunctions.

\subsection{Co-resonant states.}\label{subsec:rescores} Here we study the connection between the semisimplicity and a suitable pairing between resonant and co-resonant states. We start off with a lemma relating the adjoint of the spectral projector and the spectral projector of the adjoint.

\begin{lemma}\label{lemma:adjoints}
	Let $P$ be a first order differential operator acting on sections of $\E$ with principal symbol $-i \sigma(X)\times \id_{\E}$ and consider the adjoint operator $P^*$. %with the same principal symbol. 
	Denote the spectral projector of $P$ at $\lambda_0 \in \mathbb{C}$ by $\Pi_{\lambda_0}$ and of $P^*$ by $\Pi_{\lambda_0}'$. Also, denote the resolvent by $R_{P}(\lambda) = (P - \lambda)^{-1}$. Then\footnote{Here we interpret $-R_{-P^*}(-\overline{\lambda}): C^\infty(M; \E) \to \mathcal{D}'(M; \E)$ as the operator obtained by meromorphic continuation, but with respect to the flow generated by $-X$.} %if we denote the resolvent by $R_{P}(\lambda) = (P - \lambda)^{-1}$
	\[R_{P}(\lambda)^* = -R_{-P^*}(-\overline{\lambda}) \quad \mathrm{and} \quad \Pi_{\lambda}^* = \Pi'_{-\overline{\lambda}}.\]
\end{lemma}
\begin{proof}
	Firstly note that for $\im \lambda \gg 1$ and all $u, v \in L^2(M; \E)$, by \eqref{eq:laplace} we have the identity
	\begin{equation}\label{eq:adjoints'}
		\langle{R_P(\lambda)u, v}\rangle_{L^2} = \langle{u, -R_{-P^*}(-\overline{\lambda})v}\rangle_{L^2}.
	\end{equation}
	%Moreover, by continuity the equality also holds for $\im \lambda \gg 1$ and all $u \in \mathcal{D}'_{E_u^*}(M; \E)$, $v \in \mathcal{D}'_{E_s^*}(M; \E)$. 
	Then by analytic continuation we have the equality in \eqref{eq:adjoints'} for any $u, v \in C^\infty$ for all $\lambda \in \mathbb{C}$. Moreover, by continuity and the mapping properties of $R_P(\lambda): \mathcal{D}'_{E_u^*}(M; \E) \to \mathcal{D}'_{E_u^*}(M; \E)$ and $R_{-P^*}(-\overline{\lambda}): \mathcal{D}'_{E_s^*}(M; \E) \to \mathcal{D}'_{E_s^*}(M; \E)$ outside the poles, we have \eqref{eq:adjoints'} for all $u \in \mathcal{D}'_{E_u^*}$ and $v \in \mathcal{D}'_{E_s^*}$. This proves the first claim.% i.e. \eqref{eq:adjoints'} holds for all $u \in \mathcal{D}'_{E_u^*}(M; \E)$, $v \in \mathcal{D}'_{E_s^*}(M; \E)$ outside the poles.
	
%	Now recall the identity \eqref{expansion}. By taking the adjoints and by using the paragraph above we may write for $\lambda$ near $\lambda_0$
%	\begin{equation}\label{eq:adjoints}
%		R_H(\lambda)^* - \sum_{j = 1}^{J(\lambda_0)} \frac{(P^* - \overline{\lambda}_0)^{j - 1} \Pi^*_{\lambda_0}}{(\overline{\lambda} - \overline{\lambda}_0)^{j}} = -R_H'(-\overline{\lambda}) + \sum_{i = 1}^{J'(\lambda_0)} \frac{(-P^* + \overline{\lambda}_0)^{j - 1} \Pi'_{-\overline{\lambda}_0}}{(-\overline{\lambda} + \overline{\lambda}_0)^j}
%	\end{equation}
	%Here the primed quantities are corresponding to the operator $-P^*$. By equating the coefficients we get the desired claim.
	Now let $u \in \mathcal{D}'_{E_u^*}(M; \E)$ and $v \in \mathcal{D}'_{E_s^*}(M; \E)$. We may write
	\[\langle{\Pi_{\lambda_0} u, v}\rangle = -\frac{1}{2\pi i} \oint_{\lambda_0} \langle{R_P(\lambda)u, v}\rangle d\lambda = \frac{1}{2\pi i} \oint_{\lambda_0} \langle{u, R_{-P^*}(-\overline{\lambda})v}\rangle d\lambda = \langle{u, \Pi_{-\overline{\lambda}_0}' v}\rangle.\]
	This proves $\Pi_{\lambda_0}^* = \Pi_{-\overline{\lambda}_0}'$.
\end{proof}

We proceed to define the \emph{co-resonant states}. Given an operator $P$ as in Lemma \ref{lemma:adjoints} and a resonance $\lambda_0 \in \mathbb{C}$, the space of co-resonant states at $\lambda_0$ is $\res_{-P^*}(-\overline{\lambda}_0) \subset \mathcal{D}'_{E_s^*}(M; \E)$. By the wavefront set conditions, notice that we may multiply resonances and co-resonances in the scalar case, or form inner products (see e.g. \cite[Proposition 7.6]{GrSj94}). We are now ready to re-interpret the semisimplicity in terms of the pairing
\begin{equation}\label{eq:pairingfull}
	\res_{P}(\lambda_0) \times \res_{-P^*}(-\overline{\lambda}_0) \to \mathbb{C}, \,\, (u, v) \mapsto \langle{u, v}\rangle_{L^2}.
\end{equation}
Observe that the pairing \eqref{eq:pairingfull} is non-degenerate: we have $\langle{u, v}\rangle = 0$ for all $v \in \res_{-P^*}(-\overline{\lambda}_0)$ if and only if $\langle{u, \Pi'_{-\overline{\lambda}_0}\varphi}\rangle = 0$ for all $\varphi \in C^\infty(M; \E)$. Then by Lemma \ref{lemma:adjoints} and since $\Pi_{\lambda_0} u = u$, this holds if and only if $u \equiv 0$; by an analogous argument for the other entry, we obtain the non-degeneracy. In particular, $m_P(\lambda_0) = m_{-P^*}(-\overline{\lambda}_0)$ and also $J(\lambda_0) = J'(-\overline{\lambda}_0)$. %\footnote{Is necessarily $\dim \res^{(1)}_{P}(\lambda_0) = \dim \res^{(1)}_{-P^*}(-\overline{\lambda}_0)$?} 
Here $J'(\mu)$ denotes the size of the largest Jordan block of $-P^*$ at $\mu$. %In what follows, we denote by $\res^{(k)}_{P}(\lambda)$ for some positive integer $k$ the set of generalised resonant states in the kernel of $(P - \lambda)^k$.

\begin{lemma}\label{lemma:pairingsemisimple}
	Assume $P$ satisfies the assumptions of Lemma \ref{lemma:adjoints}. Then we have that the semisimplicity for $P$ at $\lambda_0$ holds if and only if the semisimplicity for $-P^*$ at $-\overline{\lambda}_0$ holds. Moreover, $P$ is semisimple at $\lambda_0$ if and only if the pairing
\begin{equation}\label{eq:pairingrescores}
	\res^{(1)}_{P}(\lambda_0) \times \res^{(1)}_{-P^*}(-\overline{\lambda}_0) \to \mathbb{C}, \,\, (u, v) \mapsto \langle{u, v}\rangle_{L^2}.
\end{equation}
is non-degenerate.
\end{lemma}
\begin{proof}
	For the first claim, simply note that by the previous paragraph we have $J(\lambda_0) = J'(-\overline{\lambda}_0)$.
	
	For the second claim, assume first that the pairing \eqref{eq:pairingrescores} is non-degenerate. Assume we have $u, u' \in \mathcal{D}'_{E_u^*}(M; \E)$ with $(P - \lambda_0)u = u'$ with $u' \in \res_{P}^{(1)}(\lambda_0)$. We want to show $u' = 0$. We have, for any $v \in \res^{(1)}_{-P^*}(-\overline{\lambda}_0)$%(M; \E)$
	\[\langle{u', v}\rangle = \langle{(P - \lambda_0)u, v}\rangle = \langle{u, (P^* - \overline{\lambda}_0)v}\rangle = 0.\]
	Now non-degeneracy implies $u' = 0$.
	
	Assume next the semisimplicity holds for $P$ at $\lambda_0$ and let $u \in \res^{(1)}_{P}(\lambda_0)$ satisfy $\langle{u, v}\rangle = 0$ for all $v \in \res^{(1)}_{-P^*}(-\overline{\lambda}_0)$. Then we have, for all $\varphi \in C^\infty(M; \E)$
		\[\langle{u, \varphi}\rangle = \langle{\Pi_{\lambda_0}u, \varphi}\rangle = \langle{u,\Pi'_{-\overline{\lambda}_0}\varphi}\rangle = 0.\]
	Here we used Lemma \ref{lemma:adjoints} and the assumption. Thus $u \equiv 0$. The fact that $-P^*$ is semisimple at $-\overline{\lambda}_0$ and an analogous argument for the other entry proves the non-degeneracy and finishes the proof.
\end{proof}

\subsection{Further preparatory results}
%The Pollicott-Ruelle resonance multiplicity can be computed using Lemma 2.2 from \cite{DyZw17}:% (what is it a special case of???):

%\begin{lemma}\label{lemma1}
%Define the space of resonant states at $\lambda_0 \in \mathbb{C}$,
%\begin{align}
%\textnormal{Res}_P(\lambda_0) = \{u \in \mathcal{D}'_{E_u^*}(M; \mathcal{E}): (P - \lambda_0)u = 0\}
%\end{align}
%Then $m_P(\lambda_0) \geq \dim \textnormal{Res}_P(\lambda_0)$. Moreover, we have $m_P(\lambda_0) = \dim \textnormal{Res}_P(\lambda_0)$ under the following semisimplicity condition:
%\begin{align}\label{eq:semicimplicity}
%    u \in \mathcal{D}'_{E_u^*}(M; \mathcal{E}), \quad (P - \lambda_0)^2 u = 0 \quad \implies \quad (P - \lambda_0)u = 0
%\end{align}
%\end{lemma}

We start by quoting an important technical result, see \cite[Lemma 2.3]{DyZw17}. %(the proof of this looks a bit technical).

\begin{lemma}\label{lemma2}
Suppose there exist a smooth volume form on $M$ and a smooth inner product on the fibers of $\mathcal{E}$ for which $P^* = P$ on $L^2(M; \mathcal{E})$. Suppose that $u \in \mathcal{D}'_{E_u^*}(M; \mathcal{E})$ satisfies\footnote{The inner product in this paper is complex conjugate in the second variable.}
\[Pu \in C^\infty(M; \mathcal{E}), \quad \im \langle{Pu, u}\rangle_{L^2} \geq 0.\]
Then $u \in C^\infty(M; \mathcal{E})$. In particular, the conclusion of the lemma holds for $u$ a resonant state with the eigenvalue $\lambda \in \mathbb{R}$ -- just swap $P$ with $P - \lambda$.
\end{lemma}

%In order to compute $\text{Res}_{1, A}(0)$, we need a few more auxiliary results. 
We also need a simple regularity result analogous to \cite[Lemma 2.1]{DyZw17}. We give it here for completeness

\begin{lemma}\label{lemma:regularity}
	Assume $d_{A}$ is flat and let $\Gamma \subset T^*M \setminus 0$ be a closed conic set. Assume that $u \in \mathcal{D}'_{\Gamma}(M; \Omega^k \otimes \mathcal{E})$ and $d_{A} u \in C^\infty(M; \Omega^{k+1} \otimes \mathcal{E})$.
	Then there exists $v \in C^\infty(M; \Omega^k \otimes \mathcal{E})$ and $w \in \mathcal{D}'_{\Gamma}(M; \Omega^{k-1} \otimes \mathcal{E})$ such that $u = v +d_{A} w$.
\end{lemma}
\begin{proof}
	The proof follows formally by replacing $d$ with $d_{A}$ and $\delta$ with $d_{A}^*$ in the proof \cite[Lemma 2.1]{DyZw17}.
\end{proof}

%Furthermore, we need a solvability result along the lines of Proposition 3.3. in \cite{DyZw17}.

%seems also to work for $A$ non-unitary..

\subsection{Cohomology in a flat bundle.}
Given a manifold $M$ of dimension $n$ and a Hermitian vector bundle $\mathcal{E}$ with a flat connection $A$, we may consider the complex given by
\begin{equation}\label{eq:chain}
    \xymatrix{
        0 \ar[r]^-{d_{A}} & C^\infty(M; \mathcal{E}) \ar[r]^-{d_{A}} & C^\infty(M; \Omega^1 \otimes \mathcal{E}) \ar[r]^-{d_{A}} & \cdots \ar[r]^-{d_{A}} & C^\infty(M; \Omega^n \otimes \mathcal{E}) \ar[r]^-{d_{A}} & 0.
        }
        % \ar[d]_g & B \ar[d]^{g'} \\
        %D \ar[r]_{f'}       & C }
\end{equation}
Here we extend, as usual, the action of $d_{A}$ to vector valued differential forms by asking that the Leibnitz rule holds. The homology of this complex will be denoted by $H^k_A(M; \mathcal{E})$ for $k = 0, \dotso, n$. %As for the de Rham cohomology or singular cohomology, we have topological results about it such as the Gysin sequence, Poincar\'e duality, properties of the Euler characteristic, etc. (see \cite{BT}). %It is related to cohomology with coefficients in a local system and to sheaf cohomology of a flat bundle.

%Now we prove an auxiliary topological result that will help us deal with pull-backs. 
Consider now $\Sigma$ an oriented Riemannian surface and let $\mathcal E$ be a Hermitian vector bundle over $\Sigma$ equipped with a unitary, flat
connection $A$. We can pull-back the bundle $\mathcal E$ to the unit sphere bundle $\pi:S\Sigma\to\Sigma$ to obtain $\pi^*\mathcal E$, equipped with a unitary, flat connection $\pi^*A$. 

% The result we want to prove here is the following:

% \begin{lemma}\label{lemma:twistedeuler}
% 		The Euler characteristic of the chain complex \eqref{eq:chain}, denoted by $\chi_{A}(M; \mathcal{E})$ is equal to:
% 		\[\chi_{A}(M; \mathcal{E}) = \rk(\mathcal{E}) \times \chi(M)\]
% \end{lemma}
% \begin{proof}
%     See the Appendix.
% \end{proof}

% We also have the following lemma:

\begin{lemma}\label{lemma:H^1}
		Assume $\Sigma$ has genus $g \neq 1$. %Assume $A$ is flat. 
		Then the following map is an isomorphism:
		\begin{equation}
				\pi^*: H^1_A(\Sigma; \mathcal{E}) \to H^1_{\pi^*A}(S\Sigma; \pi^*\mathcal{E}).
		\end{equation}
\end{lemma}
\begin{proof}
        There is a vertical vector field $V$ that generates the rotation in the fibres of $S\Sigma$. We first check $\pi^*$ is injective, so assume $\pi^* \theta = d_{\pi^*A} F$, where $\theta \in C^\infty(\Sigma; \Omega^1 \otimes \mathcal{E})$ is $d_{A}$-closed and $F \in C^\infty(S\Sigma; \mathcal{E})$. This implies $\iota_{V}d_{\pi^*A} F = 0$. Note that if $x \in \Sigma$, there is a small ball $B$ with $x \in B$, over which $\mathcal{E}$ is trivial. Thus $\iota_{V}d_{\pi^*A}F = 0$ implies $VF = 0$ (since $\iota_{V}\pi^*A=0$) and so $F = \pi^* f$ locally; this is easily seen to extend to $F = \pi^*f$ globally for some $f \in C^\infty(\Sigma; \mathcal{E})$. This implies $\pi^*(d_{A} f - \theta) = 0$ and so $d_{A} f = \theta$.
        
        For surjectivity, take $u \in C^\infty(S\Sigma; \Omega^1 \otimes \pi^*\mathcal{E})$ with $d_{\pi^*A} u = 0$. We want to prove there are $v$ and $F$ such that $u = \pi^* v + d_{\pi^*A}F$, where $v$ is $d_{A}$-closed. This implies 
        \begin{equation}\label{eq:wanttosolve}
            \iota_V u =\iota_{V}d_{\pi^*A}F.%=VF
        \end{equation}
        If we solve equation \eqref{eq:wanttosolve}, then $w = u - d_{\pi^*A}F$ satisfies: $d_{\pi^*A} w = 0$ and $\iota_V w = 0$. By going to local trivialisations where $A = 0$, a computation implies $w = \pi^*v$  for some $1$-form $v$ locally. Again, by uniqueness this may be easily extended to some global $v \in C^\infty(\Sigma; \E)$ with $d_{A} v = 0$. We now focus on \eqref{eq:wanttosolve} and finding such $F$.
        
        To this end, we introduce the pushforward map $\pi_*: C^\infty(S\Sigma; \Omega^1 \otimes \pi^*\mathcal{E}) \to C^\infty(\Sigma; \mathcal{E})$, by integrating along the fibers
        \begin{equation}
            \pi_*: \alpha(x, v) \mapsto \beta(x) = \int_{S_x \Sigma} \alpha.
        \end{equation}
        One can show that the pushforward is well-defined and that it intertwines $d_{A}$ and $d_{\pi^*A}$; after going to a trivialisation where $A = 0$, this reduces to showing commutation with $d$, which follows from \cite[Proposition 6.14.1]{BT}. Thus $\pi_*$ descends to cohomology, i.e. we have $\pi_*: H_{\pi^*A}^1(S\Sigma; \E) \to H_A^0(\Sigma; \E)$.
        
        Now observe that the equation \eqref{eq:wanttosolve} can be solved if and only if $\pi_* u = 0$. We introduce the section $s \in C^\infty(\Sigma; \mathcal{E})$ with $s(x) = \pi_* u$. Note that $d_{A}s=0$. Moreover, we have for $K$ the Gaussian curvature of $\Sigma$:
        \begin{align}\label{eq:identityparallel}
            \int_{S\Sigma} \big\langle{u, \pi^*(sKd\vol_\Sigma)}\big\rangle = \int_\Sigma \big\langle{\pi_* u, sKd\vol_\Sigma}\big\rangle = \int_\Sigma \lVert{s}\rVert^2 K d\vol_\Sigma = \lVert{s}\rVert^2 2\pi \chi(\Sigma).
        \end{align}
        Here we used that $\lVert{s}\rVert^2$ is constant, since $s$ is parallel and $A$ is unitary, and we applied Gauss-Bonnet theorem. In the first equality we use a generalisation of \cite[Proposition 6.15.]{BT}. We use the convention that $\langle{s\alpha, s'\beta}\rangle = \langle{s, s'}\rangle_{\E} \alpha \wedge \overline{\beta}$, where $\alpha$ and $\beta$ are forms of complementary degree and $s, s'$ are sections.
        
        On the other hand, we have $\pi^*(Kd\vol_\Sigma) = - d\psi$, where $\psi$ is the connection one form on $S\Sigma$. Therefore we have the pointwise identity, as $d_{\pi^*A}u = 0$ and $d_A s = 0$%(see Section \ref{sec:structure}).
        \[\big\langle{u, \pi^*(sKd\vol_\Sigma)}\big\rangle = d\big\langle{u, (\pi^*s) \psi}\big\rangle.\]
        So by Stokes' theorem we obtain that the first integral in \eqref{eq:identityparallel} is zero. Since $g \neq 1$, we have $\chi(\Sigma) \neq 0$ and so $s = 0$. Therefore $\pi_*u = 0$, which concludes the proof.
        %So by Stokes' theorem:
        %\begin{equation}
          %  \int_{S\Sigma} \langle{u, \pi^*s \otimes \pi^*(Kdvol_\Sigma)}\rangle = 0
%        \end{equation}
\end{proof}

\begin{rem}\rm
    Alternatively, we could have proved Lemma \ref{lemma:H^1} more abstractly using a version of the Gysin sequence for twisted de Rham complexes, see \cite[p. 177]{BT} for more details.
\end{rem}

%\subsection{Euler characteristic.} 
We now compute the Euler characteristic of the twisted de Rham complex. This shows that, although the twisted Betti numbers, i.e. dimensions of $H_A^k(M; \mathcal{E})$ can jump by changing $A$, the Euler characteristic is independent of the choice of flat connection. We could not find an appropriate reference for this result.

\begin{lemma}\label{lemma:twistedeuler}
		The Euler characteristic of the chain complex \eqref{eq:chain}, denoted by $\chi_{A}(M; \mathcal{E})$ is equal to:
		\[\chi_{A}(M; \mathcal{E}) = \rk(\mathcal{E}) \chi(M).\]
\end{lemma}
\begin{proof}
        A way to prove this is given by an application of the Atiyah-Singer index theorem; we sketch the proof here. It starts by noting that, as with the usual non-twisted forms we have
        \begin{equation}
            d_A + d_A^*: C^{\infty}(M; \Omega^{odd} \otimes \mathcal{E}) \to C^{\infty}(M; \Omega^{even} \otimes \mathcal{E}).
        \end{equation}
        Here $\Omega^{even} = \oplus_i \Omega^{2i}$ and $\Omega^{odd} = \oplus_i \Omega^{2i + 1}$ are the bundles of even and odd differential forms, respectively. Let us introduce the twisted Hodge laplacian, $\Delta_A = d_A^*d_A + d_A d_A^*$. By Hodge theory, we have $H_A^k(M; \mathcal{E}) \cong \ker \Delta_A|_{\Omega^k \otimes \mathcal{E}}$. Therefore, we also have $\ind(d_A + d_A^*) = \chi_A(M; \mathcal{E})$, where by $\ind$ we denote the index of an operator.
        
        By Atiyah-Singer index theorem,
        \begin{align}
            \ind(d_A + d_A^*) &= \int_{T^*M} \ch\big(\dd(d_A + d_A^*)\big) \mathcal{T}(TM)\\
            &= \int_{T^*M} \ch(\mathcal{E}) \ch\big(\dd(d + d^*)\big) \mathcal{T}(TM)\\
            &= \rk(\mathcal{E}) \int_{T^*M} \ch\big(\dd(d + d^*)\big) \mathcal{T}(TM) = \rk(\mathcal{E}) \chi(M).
        \end{align}
        Here, $\mathcal{T}$ denotes the Todd class and $\ch$ denotes the Chern character\footnote{More explicitly, these are given for a vector bundle $V$ over $M$ with curvature two-form $\Omega$ and $w = -\frac{\Omega}{2\pi i}$, by: $\ch(V) = \Tr \exp w$ and $\mathcal{T}(V) = \det \Big(\frac{w}{1 - \exp(-w)}\Big)$. Here we apply the Taylor series at zero to forms.}. The letter $\dd$ denotes the \emph{difference bundle}. Since $(\mathcal{E}, A)$ is flat by assumption, we have $\ch(\mathcal{E}) = \rk(\mathcal{E})$. The transition to the second line is justified since the principal symbol of $d_A + d_A^*$ is equal to $\sigma(d + d^*) \otimes Id_{\mathcal{E}}$, so that
        \[\dd(d_A + d_A^*) =  \dd\big(\sigma(d + d^*) \otimes Id\big) = [G_1 \otimes \mathcal{E}] - [G_2 \otimes \mathcal{E}] = ([G_1] - [G_2]) \cdot [\mathcal{E}] \in K^{comp}(T^*M).\]
        Here $G_1$ and $G_2$ are certain vector bundles over a one point compactification of $T^*M$ and $K^{comp}$ denotes the suitable $K$-theory. Since $\ch$ is multiplicative over the $K$-theory, we get the product of characters. The last equality follows from the Atiyah-Singer index theorem for the operator $d + d^*: \Omega^{odd} \to \Omega^{even}$ and the non-twisted Hodge theory.
\end{proof}

\section{Meromorphic continuation of $\zeta_{A}(s)$}\label{section:zeta}

%\todo[inline]{G: I have not checked this section in full detail yet, but I think I would prefer to use $d_{A}$
%instead of $\nabla$ throughout.}
%{\color{blue} M: That's absolutely fine; I agree.}

We devote this section to showing meromorphic continuation of $\zeta_A(s)$ given by \eqref{DefRuelleZeta} for an arbitrary (possibly non-flat, non-unitary) $A$. We note that the meromorphic continuation of the Ruelle zeta function was first established in \cite{GiLiPo} and later by \cite{DyZw16}, and that here we follow the latter microlocal approach. %We will be using results of Dyatlov-Guillarmou \cite{DyG16}; in fact, their results hold if the trapped set of the flow only is hyperbolic, i.e. admits a splitting of the tangent bundle in the flow direction, stable and unstable bundles. 

Let $(M, g)$ be a compact Riemannian manifold and $\mathcal{E}$ a Hermitian vector bundle over $M$ equipped with a connection $A$ and an endomorphism valued function $\Phi$. Also assume $M$ admits an Anosov flow $\varphi_t$ with generator $X$. We consider the first order operator $P = -i\iota_X d_A + \Phi$.

Let us denote by $\alpha_{x, t}$ the parallel transport (with respect to $P$) in the fibers of $\mathcal{E}$ along integral curves of $\varphi_t$
\begin{align}\label{eq:paralleltransport}
\alpha_{x, t}: \mathcal{E}(x) \to \mathcal{E}(\varphi_t(x)).
\end{align}
Recall now that the propagator $e^{-itP}$ is the one parameter family of operators, defined by solving the first order PDE in $(t, x)$ for $u \in C^\infty(M; \E)$
\begin{equation}\label{eq:propagator}
\Big(\frac{\partial}{\partial t} + iP\Big)(e^{-itP} u) = 0. %\quad \text{ and } \quad e^{-itP}\eta \to \eta \quad\text{ as }\quad t \to 0.
\end{equation}
Then the solution $u(t, x) = \big(e^{-itP}u\big)(t, x) \in C^\infty(\mathbb{R} \times M; \E)$ (we pull back $\E$ to $\mathbb{R} \times M$) and we have
\begin{equation}\label{eq:propagator'}
	\big(e^{-itP}u\big) (t, x) = u(t, x) = \alpha_{\varphi_{-t}x, t} u(\varphi_{-t}x).
\end{equation}
This follows by a computation in local coordinates. In fact, in a local coordinate system $U \ni x$ over which $\E|_U \cong U \times \mathbb{C}^m$ is trivial and for small $t$, we have
\begin{equation}\label{eq:parallel}
\big(\partial_t + A(\partial_t) + i \Phi(\varphi_t x) \big) \alpha_{x, t} = 0.
\end{equation}
We write $A$ for the matrix of one forms associated to $d_A = d + A$ and identify $\alpha_{x, t}$ with a matrix. Then we may compute, using chain rule
\begin{align*}
	\partial_tu(t, x) &= -\big(A(X(x)) + i\Phi(x)\big) \alpha_{\varphi_{-t}x, t}u(\varphi_{-t} x) - (X\alpha)_{\varphi_{-t}x, t} u(\varphi_{-t}x) - \alpha_{\varphi_{-t}x, t} Xu(\varphi_{-t}x)\\
	&= -iP(\alpha_{\varphi_{-t}x, t}u)(t, x) + X(\alpha_{\varphi_{-t}x, t} u)(t, x)  - (X\alpha_{\varphi_{-t}x, t}) u(\varphi_{-t}x) - \alpha_{\varphi_{-t}x, t} Xu(\varphi_{-t}x)\\
	&= -iP u(t, x).
\end{align*}
Here we used \eqref{eq:parallel} in the first equality, the definition of $P$ in the second and the chain rule in the last one. We thus obtain \eqref{eq:propagator'} for small $t$ and by iteration we obtain it for all $t$. As a consequence, we obtain for any $f \in C^\infty(M)$ and $u \in C^\infty(M; \mathcal{E})$
\begin{equation}\label{eq:chainrule}
    e^{-itP}(f u) = f \circ \varphi_{-t} \cdot e^{-itP}u.
\end{equation}
%Moreover, the operator $P$ commutes with $e^{-itP}$ and by \eqref{
%Thus by \eqref{eq:propagator'}, we obtain that the Schwartz kernel of $e^{-itP}$ is a distribution in $\mathcal{D}'(\mathbb{R} \times M \times M)$ that has the wavefront set contained in the conormal bundle to the submanifold $\{y = \varphi_{-t}x\}$.
%To see this, first note that for $f \in C^\infty(M)$ and $u \in C^\infty(M; \mathcal{E})$:

%The relation to parallel transport is given by
%\begin{equation}
%	\big(e^{-itP}u\big)()
%\end{equation}

%This follows by showing that the right hand side solves \eqref{eq:heat}. Therefore  $e^{-itP}u(\varphi^t(x))$ depends only on $u(x)$, as can be seen by writing $u$ as a sum of terms of the form $fv$ where $f(x) = 0$ and $v$ a section. Thus $\alpha_{x, t}$ is well-defined.

%Moreover, since we are working locally, we assume the bundle is trivial and the connection is given by the matrix of one forms $A$. Then it is easy to see that for any $x$:
%\begin{align}
%    \frac{\partial}{\partial t} \big(e^{-t(\nabla_X + i\Phi)}u (\varphi^t x)\big) &= -(\nabla_X + i\Phi)(e^{-t(\nabla_X + i\Phi)}u)(\varphi^tx) + X (e^{-t(\nabla_X + i\Phi)}u) (\varphi^tx)\\
%    &= -\big(A(\partial_t) + i\Phi\big)(e^{-t(\nabla_X + i\Phi)}u)(\varphi^tx)
%\end{align}
%In particular if $\Phi = 0$, we recover the parallel transport with respect to the connection $\nabla$, which justifies the use of the name parallel transport for $\alpha_{x, t}$.

Denote by $\mathcal{P}_{x, t}$ the linearised Poincar\'e map for any time $t$ and point $x \in M$
\begin{align*}
\mathcal{P}_{x, t} = (d\varphi_t(x))^{-T}: \Omega^1_0(x) \to \Omega^1_0(\varphi_t x),
\end{align*}
where for $x \in M$ and $k \in \mathbb{N}$, we define the subbundle of differential forms in the kernel of $\iota_X$
\[\Omega^k_0 = \Omega^k \cap \ker \iota_X.\]
We write $-T$ for the inverse transpose. Note $\Lapl_X$ acts on sections of $\Omega^k_0$ for any $k$. Also, we have that $\varphi_t^*$ is a one parameter family of maps acting on $\Omega_0^k$ for any $k$, such that we may write $(\varphi_t)^* = e^{t \Lapl_X}$. So we obtain that, by definition of $\varphi_{-t}^*$ for any $\eta$ a smooth $k$-form (cf. \eqref{eq:propagator'})
%on the space of forms (acting by pullback) and that it is generated by $\Lapl_X = \frac{\partial}{\partial t}|_{t = 0} (\varphi_t)^*$, the Lie derivative. Therefore 
\begin{align}\label{eq:propagator''}
\wedge^k \mathcal{P}_{x, t}(\eta(x)) &= e^{-t\Lapl_X} \eta(\varphi_tx).
\end{align}
Here $\wedge^k \mathcal{P}_{x, t}$ is the exterior product of maps acting on $\Omega_0^k$. %We state the former claim once more for convenience:
%\begin{align}\label{eq:semigpaction'}
%\alpha_{x, t}(u(x)) &= e^{-itP}u(\varphi^tx)
%\end{align}
Given a closed orbit $\gamma$ with period $T$, we consider a point $x_0 \in \gamma$ and define 
\[\Tr \alpha_{\gamma} := \Tr \alpha_{x_0, T}.\]
Since the maps $\alpha_{\varphi_t x_0, T}$ are conjugate for varying $t$, the trace is independent of $\gamma$. Similarly, we define
\[\det(\id - \mathcal{P}_\gamma) := \det(\id - \mathcal{P}_{x_0, T}).\]
%The operator $P$ extends its action to vector valued differential forms in the usual way, by respecting the Lebnitz rule. We denote the corresponding operator by $P_k$, acting on sections of the bundle $D_k = \Omega_0^k \otimes \E$.

In what follows, for technical purposes we assume that we have a constant $\beta \in \mathbb{N}$ such that
\begin{align}\label{eq:detsign}
|\det(\id - \mathcal{P}_\gamma)| = (-1)^\beta \det(\id - \mathcal{P}_\gamma).
\end{align}
This happens in particular if $E_s$ and $E_u$ are orientable, where $\beta = \dim E_s$. This assumption may be removed by using a suitable twist with an orientation bundle (see \cite{DyG16, DyZw16, GiLiPo} for details).

We will denote by $\gamma^\#$ a general primitive periodic orbit and if $\gamma$ is an arbitrary periodic orbit, then $l_\gamma^\#$ will denote the period of the primitive periodic orbit corresponding to $\gamma$.

%The starting point is Theorem 4 of \cite{DyG16} which says: %(they consider even the more complicated, non-compact case!):

\begin{theorem}\label{merotrace}
Define for $\re s \gg 1$% and every $k$
\begin{align}\label{eqntrace}
F_P(s) := \sum_{\gamma \in \mathcal{G}}\frac{e^{-s l_\gamma}l_{\gamma}^\# \Tr \alpha_\gamma}{|\det(\id - \mathcal{P}_\gamma)|} %\Tr \wedge^k \mathcal{P}_\gamma
\end{align}
where the sum is over all periodic trajectories. Then $F_P(s)$ extends meromorphically to all $s \in \mathbb{C}$. The poles of $F_P(s)$ are precisely $s \in \mathbb{C}$, where $is$ a Pollicott-Ruelle resonance of $P$. Moreover, the poles are simple with residues equal to the Pollicott-Ruelle multiplicity $m_P(is)$. %(see previous section).
\end{theorem}
\begin{proof}
	We give only a sketch of the proof here, as it follows from the work \cite{DyZw16}. The sum \eqref{eqntrace} converges by Lemma \cite[Lemma 2.2]{DyZw16} and as $\lVert{\alpha_{\gamma}}\rVert \leq C e^{Cl_\gamma}$ for some $C > 0$. Observe that by \eqref{eq:propagator'}, we have that the Schwartz kernel $K$ of the propagator $e^{-itP}$, as a distribution $K(t, y, x) \in \mathcal{D}'(\mathbb{R} \times M \times M)$ satisfies $WF(K) \subset N^*S$, where $S = \{(t, \varphi_t(x), x) : x \in M,\, t \in \mathbb{R}\}$ and $N^*S$ denotes the conormal bundle of $S$. Therefore, Guillemin's trace formula \cite[Appendix B]{DyZw16} applies to give, for $t > 0$
	\[\Tr^\flat e^{-itP}|_{C^\infty(M; \mathcal{E})} = \sum_{\gamma \in \mathcal{G}} \frac{l_{\gamma}^\# \Tr \alpha_{\gamma} \delta(t - l_\gamma)}{|\det(\id - \mathcal{P}_\gamma)|}. \]
	All that we are left to do, is to note that the remainder of the proof in \cite[Section 4]{DyZw16} is not sensitive to changing $\varphi_{-t}^*$ to a general propagator $e^{-itP}$ for $P$ as above. This completes the proof.
	
	Alternatively, the whole statement follows from the more general work \cite[Theorem 4]{DyG16} on open systems.
\end{proof}
%(see Guillemin notes)
%The proof is based on the use of the Atiyah-Bott-Guillemin trace formula in the case of vector valued operators (see equation (5.6) of \cite{DyG16}). Note that the proof of meromorphicity of the zeta function in \cite{DyZw16} considers vector bundles up until the Guillemin's formula; it is feasible then, that we may prove \eqref{eqntrace} without citing the more general work \cite{DyG16}. %Since the proof of meromorphicity in \cite{DyZw16} considers vector bundles up until the Guillemin's formula, we could probably deduce the meromorphic extension directly using results of \cite{DyZw16} without quoting the more general work of \cite{DyG16}. %Maybe worth writing down in full to deepen understanding of proof of meromorphicity.

We now prove the meromorphic extension of the zeta function using the meromorphic continuation of the trace above.

\begin{prop}
The zeta function $\zeta_A(s)$ is given by
\begin{align}
\zeta_A(s) = \prod_{\gamma^\#} \det{\big(\id - \alpha_{\gamma^\#} e^{-sl_{\gamma}^\#}\big)}
\end{align}
for large $\re s$ and holomorphic in that region. Moreover, it has a meromorphic extension to the whole of $\mathbb{C}$ and the poles and zeros of the extension are determined by Pollicott-Ruelle resonances of $P = -i \iota_X d_A + \Phi$ acting on differential forms with values in $\E$.
\end{prop}
\begin{proof}
We follow the now standard procedure of writing $\log \zeta_A$ as an alternating sum of traces of maps between bundles of differential forms with values in a vector bundle (see \cite[eq. (2.5)]{DyZw16}, originally due to Ruelle \cite{Ru76}). We write for large $\re s$

\begin{align}\label{eq:formalcompute}
\begin{split}
\log \zeta_A(s) &= \sum_{\gamma^{\#}} \log \det (\id - \alpha_{\gamma^{\#}} e^{-sl_{\gamma}^\#}) = \sum_{\gamma^\#} \Tr \log (\id - \alpha_{\gamma^\#} e^{-sl_\gamma^\#})\\
 &= -\sum_{\gamma^\#, j} \frac{\Tr \big(\alpha_{\gamma^\#}^j\big) e^{-jsl_{\gamma}^{\#}}}{j} = -\sum_\gamma \Tr(\alpha_\gamma) e^{-sl_\gamma} \frac{l_{\gamma}^\#}{l_\gamma}\\
&= \sum_{k = 0}^{n - 1} (-1)^{k + \beta + 1}\sum_{\gamma} \frac{\Tr(\wedge^k \mathcal{P}_\gamma)\Tr(\alpha_\gamma) e^{-sl_\gamma}}{|\det(\id - \mathcal{P}_\gamma)|} \frac{l_{\gamma}^\#}{l_\gamma}\\
&= \sum_{k = 0}^{n - 1} (-1)^{k + \beta} g_k(s).
\end{split}
\end{align}
We used the formula $\log \det (\id + A) = \Tr \log (\id + A)$ that works for $\lVert{A}\rVert$ small enough, the fact that there is a $C > 0$ such that $\lVert{\alpha_\gamma}\rVert \leq C e^{Cl_\gamma}$ and \cite[Lemma 2.2.]{DyZw16}. The function $g_k$ is defined as
\[g_k(s) = -\sum_{\gamma} \frac{\Tr(\wedge^k \mathcal{P}_\gamma)\Tr(\alpha_\gamma) e^{-sl_\gamma}}{|\det(\id - \mathcal{P}_\gamma)|} \frac{l_{\gamma}^\#}{l_\gamma}.\]
Also, we used the identity
\begin{align*}
\det(\id - \mathcal{P}_\gamma) = \sum_{k = 0}^{n - 1} (-1)^k \Tr(\wedge^k \mathcal{P}_\gamma),
\end{align*}
which comes from linear algebra. Introduce then
\begin{align}\label{eqnmercont1}
F_k(s) := -g_k'(s) = -\sum_\gamma \frac{\Tr(\wedge^k \mathcal{P}_\gamma)\Tr(\alpha_\gamma) e^{-sl_\gamma} l_{\gamma}^\#}{|\det(\id - \mathcal{P}_\gamma)|}.
\end{align}%\frac{\zeta_{k, A}'}{\zeta_{k, A}}
%where the sum is now over all closed trajectories. \footnote{Here $\Omega_0^k = \ker \iota_X|_{\Omega^k}$.}
This is reminiscent of the equation \eqref{eqntrace}. In fact, consider the vector bundle $\E_k := \Omega^k_0 \otimes \mathcal{E}$. We extend the action of $P$ on $\E$ to the action on $\E_k$ by the Leibnitz rule and denote the associated first order differential operator by $P_k$. We have, for $w \in C^\infty(M; \Omega_0^k)$ and $s \in C^\infty(M; \E)$
\begin{equation}\label{eq:tensorpde}
P_k (s \otimes w) = (-i \iota_X d_A + \Phi) (s \otimes w) = Ps \otimes w + s \otimes (-i \Lapl_X w).
\end{equation}
Then we observe that, by using \eqref{eq:tensorpde}
\begin{equation}\label{eq:tensorpropagator}
	(\partial_t + iP_k) (e^{-itP}s \otimes e^{-t\Lapl_X} w) = 0.
\end{equation}
Introduce the parallel transport $\beta_{k, x, t}: \E_k(x) \to \E_k(\varphi_t x)$ along the fibers of $\E_k$. Then by \eqref{eq:propagator'}, \eqref{eq:propagator''} and \eqref{eq:tensorpropagator}
\begin{multline}\label{eq:beta}
\beta_{k, x, t} \big(s(x) \otimes w(x) \big) = e^{-itP_k}\big(s \otimes w\big)(\varphi_t x) = e^{-itP}s(\varphi_t x) \otimes e^{-t \Lapl_X} w(\varphi_t x)\\
= \alpha_{x, t} (s(x)) \otimes \wedge^k \mathcal{P}_{x, t} (w(x)).
\end{multline}
We claim that for $k = 0, 1, \dotso, n-1$
\[F_{P_k}(s) = F_k(s).\]
%Here we extend the action of $P$ to the bundle $\E_k$ in the usual way, by using the Leibnitz rule (note that here we use that $\Lapl_X = \iota_X d$ on $\Omega_0$):
%where $\Lapl_X$ is the usual Lie derivative along $X$. 
%Here we used \eqref{eq:tensorpde} to %Note that we used $P_k = -i \Lapl_X \otimes Id + Id \otimes P$ on forms in the kernel of $\iota_X$, so $e^{-itP_k} = e^{-t\Lapl_X \otimes Id} e^{Id \otimes -itP}$. %\beta_{k, x, t} \big(w(x) \otimes s(x) \big) in the kernel of P_k along \varphi^t?
%We also used equations \eqref{eq:propagator} and \eqref{eq:propagator'}.
To see this, observe that along a periodic orbit $\gamma$ of period $l_\gamma$ by \eqref{eq:beta} we have
\[\Tr(\beta_{k, \gamma}) = \Tr(\alpha_\gamma \otimes \wedge^k \mathcal{P}_\gamma) = \Tr(\alpha_{\gamma}) \cdot \Tr(\wedge^k \mathcal{P}_\gamma).\]
Here we write $\beta_{k, \gamma} = \beta_{k, x_0, l_\gamma}$, where $x_0$ is any point on $\gamma$. The trace $\Tr \beta_{k, \gamma}$ is independent of $x_0$. This proves the claim.

By Theorem \ref{merotrace} and an elementary argument, for each $k$ there exists a holomorphic function $\zeta_{k, A}(s)$ such that
\[\frac{\zeta_{k, A}'}{\zeta_{k, A}} = -F_k(s) = g_k'(s).\]
Thus by \eqref{eq:formalcompute} we obtain the following factorisation
\begin{align}\label{eq:zetafactorisation}
\zeta_A(s) = \prod_{k = 0}^{n - 1} \zeta_{k, A}^{(-1)^{k + \beta}}(s).
\end{align}
%Left to prove: the poles of $\frac{\zeta'_{k, A}}{\zeta_{k, A}}$ are simple and residues are integral; thus $\zeta_{k, A}$ exists and is holomorphic. But this is contained in the statement of Theorem \ref{merotrace}. \footnote{Once we are at this stage, we could also use Lemma 4.2. from \cite{DyZw16} to prove that the holomorphic extension of $\zeta_{k, A}$. Therefore by the expansion equation \eqref{expansion} and applying the trace formula to it, since $\Pi$ is projection to $\ker(P_k - \lambda_0)^{J}$ and since the flat trace of a finite rank operator is equal to the usual trace (in a suitable space $H_{sG}$), we pick up the simple pole with coefficient $\rk \Pi$.}
%This finally implies that $\zeta_{k, A}(s)$ have a meromorphic continuation and we have the following factorisation:
%\begin{align}\label{eq:zetafactorisation}
%\zeta_A(s) = \prod_{k = 0}^{n - 1} \zeta_{k, A}^{(-1)^{k + \beta}}(s)
%\end{align}
By Theorem \ref{merotrace}, $s \in \mathbb{C}$ is a zero of $\zeta_{k, A}(s)$ precisely when $is$ is a Pollicott-Ruelle resonance of $P_k$ and the multiplicity of the zero is equal to the Pollicott-Ruelle multiplicity at $is$.
%where $\zeta_{k, A}$ are entire. The order of vanishing of $\zeta_{k, A}$ at each $s_0$ is equal to the multiplicity of the Pollicott-Ruelle resonance $m_{P_k}(s_0)$ of the operator $P_k$ acting on $D_k$.
\end{proof}

For convenience we re-state the factorisation above for $3$-manifolds.
% $n = \dim M = 3$ and $\beta = \dim E_s = 1$ and

\begin{corollary}\label{cor:3dfactor}
Consider a closed $3$-manifold $(M, g)$ with an Anosov flow $X$. Let $\mathcal{E}$ be a 
%Hermitian 
vector bundle over $M$ equipped with a 
%unitary 
connection $A$ and a 
%Hermitian 
potential $\Phi$. Then, assuming $E_s$ is orientable, we have the factorisation, where $\zeta_{k, A}$ is entire for $k = 0, 1, 2$
\begin{align}\label{eq:3dfactor}
\zeta_A(s) = \frac{\zeta_{1, A}(s)}{\zeta_{0, A}(s) \zeta_{2, A}(s)}.
\end{align}
Moreover, the order of zero at a point $s$ of $\zeta_A(s)$ is equal to
\begin{align}
m_{P_1}(is) - m_{P_0}(is) - m_{P_{2}}(is),
\end{align}
where $m_{P_k}(is)$ denotes the Pollicott-Ruelle resonance multiplicity at $is$ of the operators $P_k = -i\iota_Xd_A + \Phi$ acting on sections of the vector bundle $\E_k = \Omega_0^k(M) \otimes \mathcal{E}$ for $k = 0, 1, 2$.
\end{corollary}

%\todo{G: Do we need to assume $\dim E_{s}=1$?}
%{\color{blue} M: I was referring to \eqref{eq:detsign}. If we don't assume that, it seems we can get $\zeta = \frac{\zeta_0 \zeta_2}{\zeta_1}$ hypothetically. If the flow is volume preserving, then I believe we must have $\dim E_s = \dim E_u = 1$.}
%\todo[inline]{G: My point is that for an Anosov flow on 3-manifold $E^{s,u}$ cannot be 0-dimensional, so the assumption is superfluous.}
%{\color{blue} M: Is that because of topology? In theory we could suspend a ``contractive" diffeo to get one of these guys.}

\section{Resonant spaces}\label{section:resonances} In this section we prove:

\begin{theorem}\label{thm:generalresonances} Let $(M,\Omega)$ be a closed 3-manifold with volume form $\Omega$ and let $\varphi_{t}$ be a volume preserving Anosov flow. Let $\mathcal E$ be a Hermitian vector bundle equipped with a unitary flat connection $A$. Then
\begin{enumerate}
\item $\text{\rm dim\,Res}_{0,A}(0)=\text{\rm dim\,Res}_{2,A}(0)=b_{0}(M,\mathcal E)$.
\item If $[\omega]\neq 0$, $\text{\rm dim\,Res}_{1,A}(0)=b_{1}(M,\mathcal E)-b_{0}(M,\mathcal E)$.
\item If $[\omega]=0$, then
\[\text{\rm dim\,Res}_{1,A}(0)=\left\{\begin{array}{ll}
b_{1}(M,\mathcal E)&\mbox{\rm if}\;\mathcal H (X)\neq 0\\
b_{1}(M,\mathcal E)+b_{0}(M,\mathcal E)&\mbox{\rm if}\;\mathcal H (X)=0.\\
\end{array}\right.\] 
\end{enumerate}
Moreover, $k$-semisimplicity holds for $k=0,2$.

\label{thm:dimensionA}

\end{theorem}

In particular, as a consequence we obtain

\begin{proof}[Proof of Theorem \ref{thm:volpres}]
	This is a direct consequence of Theorem \ref{thm:generalresonances} applied to trivial bundle $\E = M \times \mathbb{C}$ and the trivial connection $d_A = d$.
\end{proof}

We break down the proof of Theorem \ref{thm:generalresonances} into the following subsections. 

\subsection{Smooth invariant 1-forms} We first show that smooth resonant 1-forms are zero. The idea is that an invariant $1$-form decays along the stable direction in the future and in the unstable direction in the past, so must vanish. This first subsection is quite general and holds in any dimensions for any unitary connection $A$ and Hermitian matrix field $\Phi$. Recall that $\Omega_{0}^k=\Omega^k\cap \text{ker}\,\iota_{X}$.

\begin{lemma}\label{lemma:ressmoothtrivial}
        We have that
	\begin{equation}
		\text{\rm Res}_{1, A, \Phi}(0) \cap C^\infty(M; \Omega_0^1 \otimes \mathcal{E}) = \{0\}.
	\end{equation}
\end{lemma}
\begin{proof}
        We start by proving the following formula, which holds for any $u \in C^\infty(M; \Omega^k \otimes \E)$
        \begin{equation}\label{eq:formula}
            \alpha_{x, t}\big(u_x(\xi^k)\big) = e^{-t(\iota_X d_A + i\Phi)}u_{\varphi_tx}\big((\wedge^k d\varphi_t) \xi^k\big).
        \end{equation}
        Here $\xi^k \in \Lambda^k_xM$ is a $k$-vector and $x$ is any point in $M$. The definitions of $\alpha_{x, t}$ are given in \eqref{eq:paralleltransport} and \eqref{eq:propagator'}.
        
        Note firstly that it suffices to prove the claim above for $u = s \otimes w$, where $w$ a $k$-form and $s$ a section of $\mathcal{E}$, since we can write $u$ as a sum of such terms near $x$ and a term which is zero close to $x$. But this follows from equation \eqref{eq:beta} and by the definition of the map $\mathcal{P}_{x, t}$.
        
        If $u \in \text{Res}_{1, A, \Phi}(0) \cap C^\infty(M; \Omega_0^1 \otimes \mathcal{E})$ we must have $(-i\iota_{X}d_{A} + \Phi) u = 0$ and $\iota_X u = 0$. This further implies $e^{-t(\iota_{X}d_{A} + i\Phi)} u = u$, since $(\partial_t + \iota_{X}d_{A} + i\Phi)u = 0$. Then by \eqref{eq:formula} for $k = 1$ and $\xi \in E_s(x)$
        \begin{equation}
            |u_x(\xi)| = |\alpha_{x, t}u_x(\xi)| = |u_{\varphi_t x}(d\varphi_t \xi)| \lesssim |d\varphi_t \xi|_g \lesssim e^{-\lambda t}, \quad t > 0.
        \end{equation}
        Here we used that $\alpha_{x, t}$ is a unitary isomorphism\footnote{This can be shown as follows. Fix $x \in M$ and take two parallel sections $u_1$ and $u_2$ of $\E$ along the orbit $\{\varphi_tx : t \in \mathbb{R}\}$, solving locally in some trivialisation $(\partial_t + A(\partial_t) + i \Phi)u_j = 0$ for $j = 1, 2$. Then $\partial_t \langle{u_1, u_2}\rangle_{\E(\varphi_tx)} = \langle{(\partial_t + A(\partial_t)) u_1, u_2}\rangle + \langle{u_1, (\partial_t + A(\partial_t)) u_2}\rangle = -i \langle{\Phi u_1, u_2}\rangle + i\langle{u_1, \Phi u_2}\rangle = 0$, as $d_A$ is unitary and $\Phi$ is Hermitian. Therefore the parallel transport preserves inner products and $\alpha_{x, t}$ is unitary.}, the Anosov property of $X$ and that $t > 0$ in the last inequality. By taking the limit $t \to \infty$, we get $u$ is zero in the direction of $E_s$. Similarly, we get that $u$ is zero in the direction of $E_u$, so $u$ is zero.
\end{proof}

\begin{rem}\rm
    The above method shows that for an arbitrary smooth $k$-form $u \in \text{Res}_{k, A, \Phi}(0)$, we have $u|_{\wedge^k E_u} = 0$ and $u|_{\wedge^k E_s} = 0$, and more generally one could compare rates of contraction and expansion to obtain vanishing on larger subspaces. Other components can be non-zero, as can be seen e.g. below from the computation for $\text{Res}_{2, A}(0)$ for $A$ flat.
\end{rem}

\subsection{$\text{\rm Res}_{0, A}(0)$ and $\text{\rm Res}_{2, A}(0)$} Recall that $\omega=i_{X}\Omega$ and assume from now on that $A$ is flat.

\begin{lemma}\label{lemma:k=02noncontact}
	We have:
	\begin{align}\label{k=0}
		\text{\rm Res}_{0, A}(0) &= \{s \in C^\infty(M; \mathcal{E}):\; d_{A}s = 0\}=H_{A}^{0}(M,\mathcal E),\\
		\text{\rm Res}_{2, A}(0) &= \{s\,\omega :\; s \in C^\infty(M; \mathcal{E}), \, d_{A}s = 0\}.\label{k=2}%wH_A^0(M; \E) =
	\end{align}
	Moreover, $k$-semisimplicity holds for $k=0,2$.
\end{lemma}
\begin{proof} We distinguish the cases $k = 0$ or $2$.

    \textbf{Case $k = 0$.} If $s \in \text{Res}_{0, A}(0)$, then $s \in C^{\infty}(M; \mathcal{E})$ by Lemma \ref{lemma2}.  Since $A$ is flat, $d^2_{A}s=0$ and therefore $d_{A} s \in \text{Res}_{1, A}(0) \cap C^{\infty}(M, \Omega_0^1\otimes \mathcal{E})$ and by Lemma \ref{lemma:ressmoothtrivial} we have $d_{A} s = 0$. So in this case we get a bijection with the parallel sections of $\mathcal{E}$. 
% or \ref{lemma:ressmoothtrivialnonflat}

    For semisimplicity, consider $s \in \mathcal{D}'_{E_u^*}(M; \mathcal{E})$ with $\iota_{X}d_{A} s =: v \in \text{Res}_{0, A}(0)$. Then $v \in C^{\infty}(M; \mathcal{E})$ by Lemma \ref{lemma2} and $v$ is parallel by the previous paragraph. For $u \in C^\infty(M; \mathcal{E})$ parallel, since $d_{A}$ is unitary, we have:
    \begin{equation}
	    \int_M \langle{\iota_{X}d_{A}s, u}\rangle_{\mathcal{E}}\,\Omega = \int_M X\langle{s, u}\rangle_{\E}\,\Omega = 0.
    \end{equation}
    By picking $u = v$, we get $v = 0$ and so $s \in \text{Res}_{0, A}(0)$.
    %For $\text{Res}_{0, A}(0)$, the proof in Lemma \ref{k=02} carries over, due to the fact we have Lemma \ref{lemma:ressmoothtrivial} and $\nabla_X$ is skew-Hermitian.
    
    \textbf{Case $k = 2$.} For $u \in \text{Res}_{2, A}(0)$, we may write
$u = s\omega$ for some distributional section $s \in \mathcal{D}'_{E_u^*}(M; \E)$. Then $\iota_{X}d_{A} u = 0$ implies $\iota_{X}d_{A} s = 0$, as $\Lapl_X \Omega = d\omega = 0$. By the analysis of $\text{Res}_{0, A}(0)$, we immediately get that $s$ is parallel.
    
    For semisimplicity, assume $\iota_{X}d_{A} u = v \in \text{Res}_{2, A}(0)$ with $u \in \mathcal{D}'_{E_u^*}(M; \Omega_0^2 \otimes \E)$. So $u = s \omega$ for some $s \in \mathcal{D}'_{E_u^*}(M; \E)$ and $v = s' \omega$ with $s'$ smooth and parallel. Therefore $s' = \iota_{X}d_{A}s \in \text{Res}_{0, A}(0)$ and by semisimplicity in the $k = 0$ case, we obtain $s' = 0$.
\end{proof}

\begin{rem}\rm
In the proof of Lemma \ref{lemma:k=02noncontact}, the fact that $J(0) = 1$ in the case $k = 0$ also holds for $A$ non-flat and unitary. To see this, consider the spectral theoretic inequality, that holds for $\varphi \in C^\infty(M; \E)$
		\begin{equation}
			\lVert{(P - \lambda)\varphi}\rVert_{L^2} \cdot \lVert{\varphi}\rVert_{L^2} \geq \big|\im \big\langle{(P - \lambda)\varphi, \varphi}\big\rangle_{L^2}\big| = |\im \lambda| \lVert{\varphi}\rVert^2_{L^2}.
 		\end{equation}
 		Here we used that $P = P^*$ on $L^2$. Therefore $\lVert{R(\lambda)}\rVert_{L^2 \to L^2} \leq \frac{1}{|\im \lambda|}$ for $\im \lambda > 0$, which implies $J(0) = 1$. %An existence statement analogous to Proposition \ref{prop:existence} could then be obtained by asking that $f$ is orthogonal to co-resonant states at zero.
\end{rem}

\subsection{$\text{\rm Res}_{1, A}(0)$} Recall that $H^0_{A}(M;\mathcal E)$ is the space of parallel sections (i.e. smooth sections $s$ of $\mathcal E$ such that $d_{A}s=0$). We start with a solvability result along the lines of \cite[Proposition 3.3.]{DyZw17}.

\begin{prop}\label{prop:existence}
	Assume $X$ preserves a smooth volume form $\Omega$ and $A$ is unitary and flat. Let $f \in C^\infty(M; \mathcal{E})$ and assume $\int_M \langle{f, s}\rangle_{\E} \Omega = 0$ for all $s \in C^\infty(M; \mathcal{E})$ parallel. Then there exists $u \in \mathcal{D}'_{E_u^*}(M; \mathcal{E})$ such that $\iota_{X}d_{A} u = f$.
\end{prop}

%\todo[inline]{G: I have added that $A$ is flat since this proof seems to rely on the computation of $\text{Res}_{0,A}(0)$ later on; is this correct?}
%{\color{blue} M: We use Lemma 4.4. to get the last equality in (2.11) and for Lemma 4.4. it seems we need flat. I will re-write this proof anyway since it needs some optimisation.}
\begin{proof}
		%Recall we we have the following expansion by \eqref{expansion} for $\lambda_0 = 0$:
		%\begin{equation}
		%		R(\lambda) = R_H(\lambda) - \sum_{i = 1}^{J(0)} \frac{P^{j-1}\Pi_0}{\lambda^j}
		%\end{equation}
		Let us denote $P = -i \iota_X d_A$. By Lemma \ref{lemma:k=02noncontact} we have the $0$-semisimplicity and so $J(0) = 1$. Thus by \eqref{expansion} near zero, where $\Pi = \Pi_0$
		\begin{equation*}
		 	R(\lambda) = R_H(\lambda) - \frac{\Pi}{\lambda}.
		\end{equation*}
		%Recall that by equation \eqref{expansion} for $P = -i\iota_{X}d_{A}$ acting on $\mathcal{D}'_{E_u^*}(M; \mathcal{E})$ and since $(P - \lambda) R(\lambda) = \id$, we must have close to zero
		Therefore, by applying $P - \lambda$ to this equation we obtain close to zero
		\begin{equation}\label{eq:resolvent+}
				(P - \lambda) R_H(\lambda) + \Pi_0 = \id.
		\end{equation}
		%Here we used the fact that $J(0) = 0$ and so $(P - \lambda)^{-1} = R_H(\l)$ on $\mathcal{D}'_{E_u^*}(M; \mathcal{E})$, since both sides are holomorphic. 
		We introduce $u := -iR_H(0)f$, which lies in $\mathcal{D}'_{E_u^*}(M; \mathcal{E})$ by the mapping properties of $R_H(\lambda)$ in \eqref{expansion}.
%\todo[inline]{G: don't we want to refer to (2.4) instead of (2.2)?} 
Then, assuming $\Pi_0 f = 0$ we have by \eqref{eq:resolvent+}, evaluated at $\lambda = 0$
		\begin{equation*}
				f = f - \Pi_0f = PR_H(0)f = (iP)\big(-iR_H(0)f\big) = \iota_{X}d_{A} u.
		\end{equation*}
		Now we prove that $\Pi_0 f  = 0$. %Recall the following inequality from spectral theory, for $u \in C^\infty(M; \mathcal{E})$:
%		\begin{equation}
%			\lVert{(P - \lambda)u}\rVert_{L^2} \cdot \lVert{u}\rVert_{L^2} \geq |\im \langle{(P - \lambda)u, u}\rangle| = -|\im \lambda| \lVert{u}\rVert^2
% 		\end{equation}
%		 Then observe, since $\lVert{R(\lambda)}\rVert_{op} \leq \frac{1}{|\im \lambda|}$, that we have $J(0) = 1$. Therefore, near zero
%		 \begin{equation}
%		 	R(\lambda) = R_H(\lambda) - \frac{\Pi}{\lambda}
%		 \end{equation}
%		 It follows from $P \Pi_0 = 0$, that we have $\im(\Pi_0) \subset \ker(P)$. 
		 %By using the semisimplicity we prove below for $P_0$ in Lemma \ref{lemma:k=02noncontact}, we have $\ker(P^{J(0)}) = \ker(P)$. 
		 %We also recall the proof of Lemma 2.2. from \cite{DyZw17}, which says $\Pi_{\lambda_0} u = u$ for $u \in \text{Res}_{0, A}(\lambda_0)$ for each $\lambda_0$. 
		 By Lemma \ref{lemma:genres} and Lemma \ref{lemma:k=02noncontact}, we get
		\begin{equation*}
				\ran(\Pi_0) = \ker(P|_{\mathcal{D}'_{E_u^*}(M; \E) })= \text{Res}_{0, A}(0) = H_A^0(M; \mathcal{E}).
		\end{equation*}
		Since $X$ is volume-preserving and $A$ is unitary, we have $P^* = P$. Therefore $\ran \Pi_0' = H_A^0(M; \E)$ analogously, where $\Pi_0'$ denote the spectral projector of $-P$ w.r.t. the flow $-X$. Now Lemma \ref{lemma:adjoints} gives $\Pi_0^* = \Pi_0'$ and so for any $g \in C^\infty(M; \E)$
		%Note that from equation \eqref{eq:laplace} by analytic continuation we have for all $\lambda$ that $R(\lambda)^* = -R_{-P}(-\bar{\lambda})$; \footnote{The adjoint is here meant in the sense of an integral transform, i.e. $\langle{R(\lambda)f, g}\rangle = \langle{f, R(\lambda)^*g}\rangle$ for all $f, g \in C^\infty(M; \mathcal{E})$ and $\langle{\cdot, \cdot}\rangle$ denotes the distributional pairing.} here $R_{-P}$ denotes the resolvent with respect to the operator $-P$. Then by Lemma \ref{lemma:k=02noncontact} and the argument as above, we have that the range of $\Pi^*$ also equals $H^0_A(M; \mathcal{E})$. The conclusion follows, since $\Pi_0^*g \in H_A^0(M; \mathcal{E})$ for all $g \in C^\infty(M; \mathcal{E})$ and by the hypothesis
		\begin{equation*}
			\langle{\Pi_0f, g}\rangle_{L^2} = \langle{f, \Pi_0^*g}\rangle_{L^2} = 0.
		\end{equation*}
		Thus $\Pi_0f = 0$, which concludes the proof.
		%Note that we have for $u \in L^2(M; \mathcal{E})$ and for an orthonormal basis $f_i$ for $i = 1, \dotso, d$ in $L^2$ of $H_A^0(M; \mathcal{E})$:
		%\begin{equation}
		%		\Pi u = \sum_{i = 1}^d f_i\int_M \langle{u, f_i}\rangle dvol_M
		%\end{equation}
		%\footnote{What is exactly meant by $\Pi^*$, i.e. the adjoint?}
		%How do they prove $R(\lambda) = R_H(\lambda) - \frac{\Pi_0}{\lambda}$ at zero, i.e. $J(0) = 1$? 
%This finally implies that $\Pi f = 0$, since $f$ is orthogonal to parallel sections.
\end{proof}

We proceed with

\begin{lemma}\label{lemma:T} There is a linear map $T:\text{\rm Res}_{1, A}(0)\to H^0_{A}(M;\mathcal E)$ such that
$d_{A}u=T(u)\omega$, where $u\in \text{\rm Res}_{1, A}(0)$. The map $T$ satisfies the following:
\begin{enumerate}
\item if $[\omega]\neq 0$ or $\mathcal H(X)\neq 0$, then $T$ is trivial;
\item if $\mathcal H(X)=0$, then $T$ is surjective.
\end{enumerate}
\end{lemma}

\begin{proof} Let $u \in \text{Res}_{1, A}(0)$. Since $A$ is flat, $d_{A}^2=0$ and hence
 $d_{A} u \in \text{Res}_{2, A}(0)$ and so $d_{A} u = s \omega$ with $s$ parallel and smooth, by Lemma \ref{lemma:k=02noncontact}. If we set $T(u)=s$, this defines a linear map such that $d_{A}u=T(u)\omega$.

Next note that given parallel sections $p,q\in H^0_{A}(M;\mathcal E)$, the inner product $\langle q,p\rangle_{\E}$ is a constant function on $M$. By Lemma \ref{lemma:regularity} there is a smooth $v$ such that $d_{A}u=d_{A}v$.
We write
\[d\langle T(u),v\rangle_{\E} = \langle T(u),d_{A}v\rangle_{\E} = \|{T(u)}\|^2\omega\]
and observe that the left hand side is exact. Hence we must have $T\equiv 0$ if $[\omega]\neq 0$.

If $[\omega]=0$, we set $\omega=d\tau$ and thus
\[d_{A}(u-T(u)\tau)=0.\]
Using Lemma \ref{lemma:regularity}, we can write $u-T(u)\tau=\eta+d_{A}F$, where $\eta$ is a smooth 1-form with $d_{A}\eta=0$ and $F\in \mathcal D'_{E^{*}_{u}}(M;\mathcal E)$. Contracting with $X$ and taking (pointwise) inner product with $T(u)$ we derive
\begin{equation}
-\|T(u)\|^2\tau(X)=\varphi(X)+X\langle T(u),F\rangle_{\E}
\label{eq:aux1}
\end{equation}
where $\varphi$ is the smooth, closed 1-form $\varphi:=\langle T(u),\eta\rangle$.
But note that
\[\int_{M}\varphi(X)\Omega=\int_{M}\varphi\wedge d\tau=-\int_{M}d(\varphi\wedge\tau)=0.\]
Hence integrating \eqref{eq:aux1} yields
\[-\|T(u)\|^2\mathcal H(X)=0\]
and therefore $T\equiv 0$ if $\mathcal H(X)\neq 0$ thus showing item (1) in the lemma.

To show item (2) assume $\mathcal H(X)=0$ and let $s$ be a parallel section. We shall show that there is 
$u\in  \text{Res}_{1, A}(0)$ with $T(u)=s$.
Note that for any parallel section $p$
\[\int_{M}\langle s\tau(X),p\rangle_{\E}\,\Omega=\langle s, p\rangle_{\E}\,\mathcal H(X)=0.\] 
By Proposition \ref{prop:existence} there is an $F\in \mathcal D'_{E^{*}_{u}}(M;\mathcal E)$ such that
$\iota_{X}d_{A}F=s\tau(X)$ and hence $u:=s\tau-d_{A}F\in \text{Res}_{1, A}(0)$ and
$T(u)=s$ as desired.

\end{proof}

\begin{lemma}\label{lemma:k=1noncontact}
	There is an injection
	\begin{equation}
			\text{\rm Ker\,T} \hookrightarrow H_A^1(M; \mathcal{E}).
	\end{equation}
	The injection can be described as follows: let $u \in \text{\rm Ker\,T}$. Then there exists $F \in \mathcal{D}'_{E_u^*}(M; \mathcal{E})$, such that
	\begin{equation}
		u - d_{A}F \in C^\infty(M; \mathcal{E} \otimes \Omega^1)
	\end{equation}
	and also $d_{A}(u - d_{A}F) = 0$. The injection map is given by 
	\begin{equation}
		S: u \in \text{\rm Ker\,T} \mapsto [u - d_{A}F] \in H_A^1(M; \mathcal{E}).
	\end{equation}
An element $[\eta] \in H_A^1(M; \mathcal{E})$ is in the image of $S$ iff
\[\int_{M}\langle p,\eta(X)\rangle_{\E}\,\Omega=0\]
for any parallel section $p$.

\label{lemma:mapS}

\end{lemma}
\begin{proof}
  Let $u \in \text{Ker\,T}$, so that $d_{A}u = 0$. By Lemma \ref{lemma:regularity} there is $F \in \mathcal{D}'_{E_u^*}(M; \mathcal{E})$ such that $u - d_{A}F \in C^{\infty}(M; \Omega^1 \otimes \mathcal{E})$. We claim that the class $[u - d_{A}F] \in H_A^1(M; \mathcal{E})$ is independent of our choice of $F$. Suppose there is a $G$ such that $u - d_{A}G$ is smooth and $d_{A}$-closed. Then $d_{A}(F-G) \in C^\infty(M; \Omega^1 \otimes \mathcal{E})$, so by Lemma \ref{lemma:regularity} (or ellipticity), $F-G$ is smooth and thus $u-d_{A}F$ and $u-d_{A}G$ belong to the same class.
	
	For injectivity, we assume that $u - d_{A}F$ is exact; so without loss of generality assume $u=d_{A}F$. Then $\iota_X u = 0$ implies $d_{A}F(X) = 0$, so by Lemma \ref{lemma:k=02noncontact} we have $F$ smooth and parallel, so $u = 0$. 

If $[\eta]$ is in the image of $S$, then $\eta=u-d_{A}F$ for some $F \in \mathcal{D}'_{E_u^*}(M; \mathcal{E})$.
Contracting with $X$, we see that $\eta(X)=-d_{A}F(X)$ and hence $\langle p,\eta(X)\rangle_{\E}=-X\langle p,F\rangle_{\E}$.
Integrating gives
\[\int_{M}\langle p,\eta(X)\rangle_{\E}\,\Omega=0.\]
Conversely, if the last integral is zero for all $p$, Proposition \ref{prop:existence} gives $F \in \mathcal{D}'_{E_u^*}(M; \mathcal{E})$ such that $-\eta(X)=d_{A}F(X)$ and $u:=\eta + d_{A}F\in \text{Ker\,T}$
and $Su=[\eta]$.

\end{proof}

And finally we can compute the rank of $S$ in terms of whether $X$ is null-homologous or not.

\begin{lemma}\label{lemma:S} We have:
\begin{enumerate}
\item $\text{\rm dim\,S(Ker\,T)}=b_{1}(M,\mathcal E)$ if $[\omega]=0$;
\item $\text{\rm dim\,S(Ker\,T)}=b_{1}(M,\mathcal E)-b_{0}(M,\mathcal E)$ if $[\omega]\neq 0$.
\end{enumerate}
\end{lemma}

\begin{proof} If $X$ is null-homologous, we write $\omega=d\tau$. We use Lemma \ref{lemma:mapS}
to show that $S$ is surjective. Consider $\eta \in H^1_A(M; \E)$ and $p \in H_A^0(M; \E)$. Since the $1$-form $\varphi:=\langle p,\eta\rangle$ is closed we have
\[\int_{M}\varphi(X)\Omega=\int_{M}\varphi\wedge d\tau=-\int_{M}d(\varphi\wedge\tau)=0\]
and item (1) follows.

Suppose now $[\omega]\neq 0$. We define a map $W:H^{1}_{A}(M,\mathcal E)\to (H_{A}^{0}(M,\mathcal E))^*$
by
\[W([\eta])(p):=\int_{M}\langle p,\eta(X)\rangle_{\E}\,\Omega.\]
By Lemma \ref{lemma:mapS} the image of $S$ coincides with the kernel of $W$. Thus, to prove item (2) it suffices to show that
$W$ is surjective.
By Poincar\'e duality there is a closed 1-form $\varphi$ such that
\[\int_{M}\varphi\wedge\omega\neq 0.\]
If $p$ and $q$ are parallel sections we compute
\[W([q\varphi])(p)=\langle p,q\rangle_{\E} \int_{M}\varphi(X)\,\Omega=\langle p,q\rangle_{\E} \int_{M}\varphi\wedge\omega\]
and hence $W$ is onto.
\end{proof}

We are now in shape to put the ingredients together and prove

\begin{proof}[Proof of Theorem \ref{thm:generalresonances}]
The theorem follows directly after applying Lemmas \ref{lemma:T} and \ref{lemma:S}.
\end{proof}

Putting together the material from this section and Section \ref{section:zeta} we obtain

\begin{proof}[Proof of Corollary \ref{corollary:zetaA}]
	The order of vanishing of $\zeta(s)$ is equal to $m_1(0) - m_0(0) - m_2(0)$ by Corollary \ref{cor:3dfactor}. By Theorem \ref{thm:generalresonances} we have that $m_0(0) = m_2(0) = b_0(M, \E)$ and $m_1(0) \geq \dim \res_{1, A}(0)$, which concludes the proof.
\end{proof}

Moreover, we obtain

\begin{proof}[Proof of Corollary \ref{corollary:zeta}]
	This is a direct consequence of Corollary \ref{corollary:zetaA} applied to the case $\E = M \times \mathbb{C}$ and the trivial connection $d_A = d$.
\end{proof}

\section{Examples}\label{section:ex}

In this section we consider a few non-contact examples of Anosov flows on the unit tangent bundle of a surface.
They illustrate the various cases in Theorem \ref{thm:volpres} and give specific deformations for Theorem \ref{thm:perturbationflow}.

%\todo[inline]{G: Notation $M$ vs $\Sigma$}

\subsection{Structural equations.}\label{subsec:structureeqns} As a general reference for structural equations, see \cite[Chapter 7]{ST76}. For this section assume $(\Sigma, g)$ is a compact oriented negatively curved surface. Let $X$ be the geodesic vector field on the unit sphere bundle $S\Sigma$. Denote by $\pi: S\Sigma \to \Sigma$ the footpoint projection. Then, there are $1$-forms $\alpha$, $\beta$ and $\psi$ on $S\Sigma$ defined by, for $\xi \in T^*_{(x, v)} S\Sigma$
\begin{align}\begin{split}
	\alpha_{(x, v)}(\xi) &= \langle{v, d\pi(\xi)}\rangle_x,\\
    \beta_{(x, v)}(\xi) &= \langle{d\pi(\xi), iv}\rangle_x,\\
    \psi_{(x, v)}(\xi) &= \langle{\mathcal{K}(\xi), iv}\rangle_x.
    \end{split}
\end{align}
The $1$-form $\alpha$ is called the contact form. From the defining equation one obtains $\iota_X \alpha = 0$ and $\iota_X d\alpha = 0$, and $\Omega = -\alpha \wedge d\alpha$ is a volume form. Also, here $\mathcal{K}: TT\Sigma \to T\Sigma$ is the \emph{connection map}, i.e. projection along the horizontal subbundle, and $\psi$ is called the \emph{connection $1$-form}. The expression $iv$ denotes the vector $v$ rotated by an angle of $\frac{\pi}{2}$ (we fix an orientation). Explicitly, 
\begin{equation}
    \mathcal{K}_{(x, v)}(\xi) := \frac{DZ}{dt}(0) \in T_x \Sigma,
\end{equation}
where $\big(\gamma(t), Z(t)\big)$ is an arbitrary local curve in $T\Sigma$ with initial data $\big(\gamma(0), Z(0)\big) = (x, v)$ and $\big(\dot{\gamma}(0), \dot{Z}(0)\big) = \xi$; $\frac{D}{dt}$ denotes the Levi-Civita derivative along the curve. One can then show that $\{\alpha, \beta, \psi\}$ form a co-frame on $S\Sigma$, such that the following \emph{structural equations} (see \cite[p. 188]{ST76}) hold
\begin{align}\label{eq:structure1}
\begin{split}
    d\alpha &= \psi \wedge \beta,\\
    d\beta &= -\psi \wedge \alpha,\\
    d\psi &= -K\alpha \wedge \beta.
    \end{split}
\end{align}
From this, we deduce the following properties
\begin{equation}\label{eq:structure2}
    \iota_X \beta = \iota_X \psi = 0, \quad \iota_X d\beta = \psi, \quad \iota_X d\psi = -K\beta.
\end{equation}
Furthermore, there is a natural choice of metric on $S\Sigma$, called the \emph{Sasaki metric}. It is defined by the splitting
\[T_{(x,v)} S\Sigma = \mathbb{H}(x, v) \oplus \mathbb{V}(x, v) = \ker(\mathcal{K}(x, v)|_{S\Sigma}) \oplus \ker\big(d\pi(x, v)\big)\]
into \emph{horizontal} and \emph{vertical} subspaces, respectively. Then the new metric is defined as:
\begin{equation}
    \langle\langle{\xi, \eta}\rangle\rangle := \langle{\mathcal{K}(\xi), \mathcal{K}(\eta)}\rangle + \langle{d\pi(\xi), d\pi(\eta)}\rangle.
\end{equation}
It follows after some checking from relations \eqref{eq:structure1} and the definitions that $\{\alpha, \beta, \psi\}$ are an orthonormal co-frame for $T^*S\Sigma$ with respect to the Sasaki metric. This also yield an orthonormal dual frame $\{X, H, V\}$, respectively. We record the structural equations \eqref{eq:structure1} for these vector fields
\begin{align}\label{eq:structure3}
	\begin{split}[H, V] &= X,\\
	[V, X] &= H,\\
	[X, H] &= KV.
	\end{split}
\end{align}
Here $V$ is the generator of rotations in the vertical fibres.

%Perhaps a more natural interpretation of the pairing \eqref{eq:pairinghyp} in our context, would be to 
We now use the Hodge star operator $\ast$ with respect to the Sasaki metric on $S\Sigma$ to write $\Lapl_X^* = -\ast \Lapl_X \ast$ on one forms. We also have an extra structure given by
\begin{equation}\label{eq:hodge}
	\alpha \wedge Ju = \ast u
\end{equation}
for $u$ section of $\Omega_0^1$. Here $J:\Omega_0^1 \to \Omega_0^1$ is the (dual) almost complex structure associated to the symplectic form $d\alpha$ on $\ker \alpha = \spann \{V, H\}$ and is given by
\[J(u_2 \beta + u_3\psi) = u_3 \beta - u_2 \psi, \quad J^2 = -\id.\] 
Therefore $(\Lapl_X^*)^k u = 0$ for some $k \in \mathbb{N}$ is equivalent to $\Lapl_X^k Ju = 0$ and we obtain
\begin{equation}\label{eq:cores1}
	\res_{-i\Lapl_X^*, \Omega_0^1}(0) = J^{-1} \res_{i\Lapl_X, \Omega_0^1}(0).
\end{equation}
In the next section we use this relation together with time-changes to derive an explicit expression for co-resonant states at zero.

\subsection{Time-reversal and resonant spaces}\label{subsec:reversal} Here we consider the action under pullback of the time-reversal map $R: S\Sigma \to S\Sigma$, given by $R(x, v) = (x, -v)$. We first collect the information on this action on the orthonormal frames and co-frames given in \eqref{eq:structure1} and \eqref{eq:structure3}.

%Recall that the co-frame on $SM$ is defined by $\alpha_{(x, v)}(\xi) = \langle{v, d\pi(\xi)}\rangle_x$, $\beta_{(x, v)}(\xi) = \langle{d\pi (\xi), iv}\rangle_x$ and $\psi_{(x, v)} = \langle{\mathcal{K}(\xi), iv}\rangle_x$, where $\mathcal{K}$ is the connection map.

\begin{prop}
	We have $R^*\alpha = -\alpha$, $R^*\beta = -\beta$ and $R^*\psi = \psi$. Similarly, we have $R^*X = -X$, $R^*H = -H$ and $R^*V = V$.
\end{prop}
\begin{proof}
We consider the co-frame case first. Simply observe that 
\[R^*\alpha_{(x, v)}(\xi) = \langle{-v, d\pi dR \xi}\rangle_x = -\alpha_{(x, v)}(\xi)\]
so $R^*\alpha = -\alpha$. Similarly 
\[R^*\beta_{(x, v)}(\xi) = \langle{-iv, d\pi dR \xi}\rangle_x = -\beta_{(x, v)}(\xi)\]
so $R^*\beta = -\beta$. Finally, recall that $\mathcal{K}(\xi) = \frac{DZ}{dt}(0)$, where $c(t) = \big(\gamma(t), Z(t)\big)$ is any curve in $T\Sigma$ with $\dot{c}(0) = \xi$. Therefore 
\[\mathcal{K}(dR \xi) = - \frac{DZ}{dt}(0) = -\mathcal{K}(\xi)\]
since $\tilde{c}(t) = \big(\gamma(t), -Z(t)\big)$ is the curve adapted to $-dR \xi$. Now we easily see that $R^*\psi = \psi$ from the definition.

%For the frame case, we start by noting 
%\[R^*V(x, v) = dR^{-1} \frac{\partial \rho_t}{\partial t}\Big|_{t = 0}(x, -v) = \frac{\partial R^{-1} \rho_t}{\partial t}\Big|_{t = 0}(x, -v) = V(x, v)\] 
The frame case follows from the co-frame case, since contractions commute with pullbacks. 
%More explicitly, write $X(x, v) = \big(\dot{\gamma}(0), \ddot{\gamma}(0)\big)$, where $\gamma(t)$ is the adapted geodesic. Therefore
%\[R^*X(x, v) = dR^{-1} \frac{\partial }{\partial t}\Big|_{t = 0} \big(\gamma(-t), -\dot{\gamma}(-t)\big) = \frac{\partial}{\partial t}\Big|_{t = 0} \big(\gamma(-t), \dot{\gamma}(-t)\big) = -\big(\dot{\gamma}(0), \ddot{\gamma}(0)\big) = -X(x, v)\]
\end{proof}

Now note that in any unit sphere bundle $SN$ over an Anosov manifold $(N, g_1)$, the pullback by $R$ swaps the stable and unstable bundles. More precisely, we have
\[R^*E_{u, s}^X = E_{u, s}^{R^*X} = E_{u, s}^{-X} = E_{s, u}, \quad R^*E_0 = E_0.\]
The upper index denotes the vector field with respect to which we are taking the stable/unstable bundles. This follows from the fact that $R$ intertwines the flows of $X$ and $-X$. Thus we also have
\[R^*E_{u, s}^{*} = E_{s, u}^{*}, \quad R^*E_0^* = E_0^*.\]

The upshot is of course that $R^*$ is an isomorphism between resonant and co-resonant spaces, i.e. the ones with the wavefront set in $E_u^*$ and in $E_s^*$.
\begin{prop}\label{prop:symmetricpairing}
	The pairing \eqref{eq:pairingfull} between resonant and co-resonant states is equivalent to the pairing on
	\begin{equation}\label{eq:pairingsymmetric}
		\res_{-i\Lapl_X, \Omega_0^1}(0) \times \res_{-i\Lapl_X, \Omega_0^1}(0), \quad (u, v) := \int_{S\Sigma} u \wedge \alpha \wedge R^*\overline{v}.
	\end{equation}
	The pairing \eqref{eq:pairingsymmetric} is Hermitian (i.e. conjugate symmetric).
\end{prop}
\begin{proof}
	We first claim that
\begin{equation}\label{eq:rescorestimereverse}
	\res_{-i\Lapl_X^*, \Omega_0^1}(0) = J^{-1} R^* \res_{-i\Lapl_X, \Omega_0^1}(0).
\end{equation}
This is obtained from \eqref{eq:cores1} and by observing that $v \in \res_{i\Lapl_X, \Omega_0^1}(0)$ if and only if $R^*v \in \res_{-i\Lapl_X, \Omega_0^1}(0)$, since $R^*$ commutes with $\iota_X$ and $d$, and as $R^*$ swaps $E_u^*$ and $E_s^*$ by the discussion above. Thus by another application of \eqref{eq:hodge}, we obtain \eqref{eq:pairingsymmetric}. %iff $v = J^{-1}R^* v'$ for some $v' \in \res_{\Lapl_X, \Omega_0^1}(0)$. 

For the symmetry part, observe that $R$ is orientation-preserving and
\[(u, v) = \int_{SM} u \wedge \alpha \wedge R^*\overline{v} = -\int_{SM} R^*u \wedge \alpha \wedge \overline{v} = \overline{(v, u)}.\]
\end{proof}

\subsection{Magnetic flows} These flows are determined by a smooth function $\lambda\in C^{\infty}(\Sigma)$.
The relevant vector field is $X_{\lambda}:=X+\lambda V$. A calculation using the structure equations shows
\[\iota_{X_{\lambda}}\Omega=-d\alpha+\lambda \alpha\wedge\beta=-d\alpha+\lambda\pi^*\sigma,\]
where $\sigma$ is the area form of $g$. If $\Sigma$ has negative Euler characteristic, then $K\sigma$ generates
$H^{2}(\Sigma)$ and thus there is a constant $c$ and a 1-form $\gamma$ such that
\[\lambda\sigma=cK\sigma+d\gamma.\]
Therefore
\[\iota_{X_{\lambda}}\Omega=-d\alpha+\lambda\pi^*\sigma=d(-\alpha-c\psi+\pi^*\gamma),\]
and hence $X_{\lambda}\in {\mathcal X}_{\Omega}^{0}$. If $X$ is Anosov and $\lambda$ is small, $X_{\lambda}$ remains Anosov. In general these flows are {\it not} contact, cf. \cite{DP}.

\subsection{Explicit flows with $[\omega]\neq 0$}\label{subsec:explicitWnonzero}

In this subsection, we construct explicit volume-preserving non null-homologous Anosov flows that are close to the geodesic flow on a compact oriented negatively surface $(\Sigma, g)$. Let $\theta \neq 0$ be a {\it harmonic} 1-form on $\Sigma$. At the level of $S\Sigma$ this can be seen in terms of two equations
\begin{align}\label{eq:thetaharmonic}
\begin{split}
X(\theta)+HV(\theta)&=0,\\
H(\theta)-XV(\theta)&=0.
\end{split}
\end{align}%Let $\pi:SM\to M$ be the canonical projection. 
This first is zero divergence, the second is $d\theta=0$. To check these equations one can argue as follows. We will use that $d\pi_{(x,v)}\big(X(x,v)\big)=v$ and
$d\pi_{(x,v)}\big(H(x,v)\big)=iv$. Given $\theta$, we consider $\pi^*\theta$ and note (using the standard formula for $d$ applied to $\pi^*\theta$):
\[d(\pi^*\theta)(X,H)=X\pi^*\theta(H)-H\big(\pi^*\theta(X)\big)-\pi^*\theta\big([X,H]\big).\]
By the structural equations, the term $[X,H]$ is purely vertical, hence it is killed by $\pi^*\theta$. Now one can check that $\pi^*\theta(H)(x,v)=\theta(iv)=V(\theta)=-(\ast \theta)(v)$ and $\pi^*\theta(X)=\theta(v)$.
Finally since
\[d(\pi^*\theta)(X,H)=\pi^*d\theta(X,H)=d\theta\big(d\pi(X),d\pi(H)\big)=d\theta(v,iv)\]
one obtains that $d\theta=0$ if and only if $H(\theta)-XV(\theta)=0$. The form $\theta$ has zero divergence if and only if $\ast\theta$ is closed so the first equation also follows.

We consider the vector field $Y:=\theta X+V(\theta)H$. This vector field is dual to the 1-form on $S\Sigma$ given
by $\pi^*\theta=\theta\alpha+V(\theta)\beta$. This form is closed as well as
$\varphi:=-V(\theta)\alpha+\theta\beta$ which is the pull back $\pi^*(\ast \theta)$.
We can easily check that $\varphi(Y)=0$ and $\pi^*\theta(Y)=[\theta]^2+[V(\theta)]^2$.

The flows we wish to consider are of the form $X_{\varepsilon}=X+\varepsilon Y$, where $X$ is the Anosov geodesic vector field and $\varepsilon$ is small so that it remains Anosov. Using the above we observe

\begin{itemize}
\item $X_{\varepsilon}$ preserves the volume form $\Omega=\alpha\wedge\beta\wedge\psi$. This is thanks to the fact
that $\theta$ has zero divergence.
\item $[\iota_{X_{\varepsilon}}\Omega]\neq 0$ for $\varepsilon\neq 0$. This is because  $\pi^*\theta(Y)=[\theta]^2+[V(\theta)]^2\geq 0$ and hence if $\theta$ is not trivial 
\begin{align}\label{eq:Wnonzero}
	\int_{S\Sigma}\pi^*\theta(X_{\varepsilon})\Omega=\varepsilon \int_{S\Sigma}\pi^*\theta(Y)\Omega\neq 0.
\end{align}

\end{itemize}

What we will prove in the coming sections is that $X_\varepsilon$ has a splitted resonance for one-forms near zero, and the semisimplicity does not break down.

\section{Perturbations}\label{section:perturbations}

In this section we study the behaviour of the Pollicott-Ruelle multiplicities under small deformations
and start with the proof of Theorem \ref{thm:localresonance}.

\subsection{Uniform anisotropic Sobolev spaces}\label{subsec:anisotropic}

We start by laying out the necessary tools to study perturbations of Anosov flows and associated anisotropic Sobolev spaces. We will follow the recent approach of Guedes Bonthonneau \cite{YGB19}, who constructs a uniform weight function that works in a neighbourhood of the initial vector field. For brevity, we will only outline the necessary details. We refer the reader to \cite{FaSj11} for more details in the case of a fixed vector field, and to \cite{DGRS18} for an alternative construction of a weight function that works for perturbed vector fields.
The use of anisotropic spaces in hyperbolic dynamics has its origins in the works of many authors, cf. \cite{Ba,BaT, 
BKL,BL,Li1,GoLi}.

Let $M$ be compact and $X_0$ an Anosov vector field. By \cite[Section 2]{YGB19}, there exists a $0$-homogeneous weight function $m \in C^\infty(T^*M \setminus 0)$ that applies to all flows with generators $\lVert{X - X_0}\rVert_{C^1} < \eta$, for some $\eta > 0$, in the sense to be explained. It satisfies, for all such $X$
\[m = 1 \text{ near } E_u^*, \quad m = - 1 \text{ near } E_s^*, \quad  X_*m \leq 0.\]
Here $X_*$ is the symplectic lift of $X$ to $T^*M$. We set $G(x, \xi) \sim m(x, \xi) \log (1 + |\xi|)$ for all $|\xi|$ large. The anisotropic Sobolev spaces are defined as, for $r \in \mathbb{R}$
\begin{equation}\label{eq:defanisotropicspace}
	\mathcal{H}_{h, rG} = \Op_h (e^{-rG}) L^2(M).
\end{equation}
Here $h > 0$ and $\Op_h$ denotes a semiclassical quantisation on $M$; we write $\Op := \Op_1$. We will write $\mathcal{H}_{rG} = \Op(e^{-rG})L^2(M)$. Frequently we consider a smooth vector bundle $\mathcal{E}$ over $M$ and in that case we consider the corresponding spaces $\mathcal{H}_{h, rG} = \Op_h(e^{-rG \times \id_{\E}}) L^2(M; \mathcal{E})$. We will write
\[\mathcal{H}_{h, rG + k \log \langle{\xi}\rangle} = \Op_{h}(e^{-rG}) H^k(M; \E).\]
We will use the special notation $\mathcal{H}_{rG, k} := \mathcal{H}_{1, rG + k \log \langle{\xi}\rangle} = \Op(e^{-rG})H^k(M; \E)$. We remark that the spaces $\mathcal{H}_{h, rG}$ for varying $h$ are all the same as sets, equipped with a family of distinct, but equivalent norms.

Let $X_\varepsilon$ be a smooth family of Anosov vector fields on $M$. Consider also a smooth family of differential operators $P_\varepsilon$ with principal symbol $\sigma(X_\varepsilon) \times \id_{\E}$. We will consider any $Q \in \Psi^{-\infty}(M; \E)$ compactly microsupported, self-adjoint operator, elliptic in the neighbourhood of the zero section in $T^*M$. Introduce now the spaces 
\[\mathcal{D}_{h, rG}^\varepsilon := \{u \in \mathcal{H}_{h, rG} : P_\varepsilon
u \in \mathcal{H}_{h, rG}\}\]
and equip them with the norm $\lVert{u}\rVert_{\mathcal{D}_{h, rG}^\varepsilon}^2 = \lVert{u}\rVert_{\mathcal{H}_{h, rG}}^2 + \lVert{hP_\varepsilon u}\rVert_{\mathcal{H}_{h, rG}}^2$. Completely analogously with $\mathcal{H}_{rG}$, we introduce $\mathcal{D}_{rG}^\varepsilon$, and also $\mathcal{D}_{rG, k}^\varepsilon$ for an integer $k$.

Then \cite[Lemma 9]{YGB19} states

\begin{lemma}\label{lemma:anisotropic}
	There exists an $\varepsilon_0 > 0$ such that the following holds. Given any $s_0 > 0$, $k \in \mathbb{Z}$ and $r > r(s_0)+|k|$, there is $h_k > 0$ such that for $0 < h < h_k$, $\im s > -s_0$, $|\re s| < h^{-\frac{1}{2}}$ and $|\varepsilon| < \varepsilon_0$,
\[P_\varepsilon - h^{-1}Q - s: \mathcal{D}_{h, rG + k\log \langle{\xi}\rangle}^\varepsilon \to \mathcal{H}_{h, rG + k\log\langle{\xi}\rangle}\]
is invertible and the inverse is bounded as $O(1)$ independently of $\varepsilon$.
\end{lemma}%on $\mathcal{H}_{h,rG+k \log \langle{\xi\rangle}}$, with image equal to $\mathcal{D}_{rG}^\varepsilon$.

Here $r(s)$ is a non-increasing function of $\im s$, so that $r(s) > r_{P_\varepsilon}(\im s)$ for all $\varepsilon\in (-\varepsilon_0, \varepsilon_0)$. Also, here $r_{P_\varepsilon}(s_0)$ represents a certain threshold (see \cite[p. 4.]{YGB19}) depending on $P_\varepsilon$ such that for $r$ bigger than this quantity the resolvent $(P_\varepsilon -h^{-1}Q - s)^{-1}: \mathcal{H}_{h, rG} \to \mathcal{H}_{h, rG}$ is holomorphic and $(P_\varepsilon - s)^{-1}: \mathcal{H}_{rG} \to \mathcal{H}_{rG}$ admits a meromorphic extension, to $\im s > - s_0$ and $|\re s| \leq h^{-\frac{1}{2}}$.

\subsection{Pollicott-Ruelle multiplicities are locally constant} In this section we prove, using the construction of anisotropic Sobolev spaces in the previous section, that in some fixed bounded region, the sums of multiplicities of resonances are locally constant. Observe that under the assumptions in Lemma \ref{lemma:anisotropic}, we have the following factorisation property
\begin{equation}\label{eq:factorisation}
	(P_\varepsilon - s)(P_\varepsilon - h^{-1}Q - s)^{-1} = \id + h^{-1}Q(P_\varepsilon - h^{-1}Q - s)^{-1}.
\end{equation}
This holds for $s$ in $\Omega_{h, s_0} : = \{s : \im s > -s_0, \,\, |\re s| < h^{-\frac{1}{2}}\}$. Introduce the notation 
\[D(\varepsilon, s) := h^{-1}Q(P_\varepsilon - h^{-1}Q - s)^{-1}.\] 
Since $Q$ is smoothing, we have that $D(\varepsilon, s)$ is of trace class, and moreover since for any $\varepsilon, \varepsilon'$
\[D(\varepsilon, s) - D(\varepsilon', s) = h^{-1}Q(P_{\varepsilon'} - h^{-1}Q - s)^{-1} (P_{\varepsilon'} - P_{\varepsilon}) (P_{\varepsilon} - h^{-1}Q - s)^{-1},\]
we have that $\varepsilon \mapsto D(\varepsilon, s)$ is continuous with values in holomorphic maps from $\Omega_{h, s_0 + 1}$ to $\Lapl(\mathcal{H}_{rG}, \mathcal{H}_{rG})$. Here $\mathcal{L}(A, B)$ denotes the space of bounded operators from $A$ to $B$, with the operator norm.

Then $P_\varepsilon - s: \mathcal{D}_{rG}^\varepsilon \to \mathcal{H}_{rG}$ are an analytic family of Fredholm operators for $\im s > -s_0$. Consider now a resonance $s_1$ of $P = P_0$, and a simple closed curve $\gamma$ around $s_1$ containing no resonances on itself or in its interior except $s_1$, such that $\gamma \subset \Omega_{h, s_0}$. The fact that $D(\varepsilon, s)$ is continuous, allows to say that for $\varepsilon$ small, a neighbourhood of $\gamma$ still contains no resonances of $P_\varepsilon$. Introduce the family of projectors
\[\Pi_\varepsilon := \frac{1}{2\pi i} \oint_{\gamma} (s - P_\varepsilon)^{-1} ds.\]

Our first aim is to prove

\begin{lemma}\label{lemma:constrank}
	The ranks of $\Pi_\varepsilon$ are locally constant, i.e. there is an $\varepsilon_1 > 0$ s.t. $\rk \Pi_\varepsilon$ is constant for $\varepsilon \in (-\varepsilon_1, \varepsilon_1)$.
\end{lemma}
\begin{proof}
	We first claim that, for $\varepsilon$ small enough
	\begin{equation}\label{eq:argumentprinciple}
		\frac{1}{2\pi i} \Tr \oint_\gamma \partial_s \big(\id + D(\varepsilon, s)\big)^{-1} \big(\id + D(\varepsilon, s)\big) ds =  -\rk \Pi_\varepsilon.
	\end{equation}
	The left hand side is well defined by the generalised argument principle \cite[Theorem C.11.]{DyZwbook}, since the contour integral is a finite rank operator. To prove the equality in \eqref{eq:argumentprinciple}, we apply the residue theorem for meromorphic families of operators. Use \eqref{eq:factorisation} to obtain the left hand side of \eqref{eq:argumentprinciple} is equal to
	\begin{align*}
		&\frac{1}{2\pi i} \Tr \oint_\gamma \big((s - P_\varepsilon)^{-1} + (P_\varepsilon - h^{-1}Q - s)(P_\varepsilon - s)^{-2}\big)(P_\varepsilon - s)(P_\varepsilon - h^{-1}Q - s)^{-1} ds\\
		&= -\frac{1}{2\pi i} \Tr \oint_\gamma (P_\varepsilon - h^{-1}Q - s)^{-1}ds + \frac{1}{2\pi i} \Tr \oint_\gamma (P_\varepsilon - h^{-1}Q - s)(P_\varepsilon - s)^{-1}(P_\varepsilon - h^{-1}Q - s)^{-1} ds.
	\end{align*}
	The first integrand in the second line above vanishes, since $(P_\varepsilon - h^{-1}Q - s)^{-1}$ is holomorphic; the second one is equal to $-\Tr \Pi_\varepsilon = -\rk \Pi_\varepsilon$, by the cyclicity of traces. This shows \eqref{eq:argumentprinciple}.
	
	Now recall by Jacobi's formula that we have
	\begin{multline*}
	\frac{1}{2\pi i} \Tr \oint_\gamma \partial_s \big(\id + D(\varepsilon, s)\big)^{-1} \big(\id + D(\varepsilon, s)\big) ds = -\frac{1}{2\pi i} \oint_\gamma \Tr \Big((\id + D(\varepsilon, s))^{-1} \partial_s D(\varepsilon, s)\Big) ds\\
	= -\frac{1}{2\pi i} \oint_\gamma \frac{\partial_s \det\big(\id + D(\varepsilon, s)\big)}{\det\big(\id + D(\varepsilon, s)\big)} ds.
	\end{multline*}
	Here we used integration by parts, and that $\partial_s D(\varepsilon, s)$ is a smoothing operator to commute trace and integration. In particular, continuity of $\varepsilon \mapsto D(\varepsilon, s)$ as above and so of the Fredholm determinant $\varepsilon \mapsto \det\big(\id + D(\varepsilon, s)\big)$ and its derivative $\varepsilon \mapsto \partial_s\det\big(\id + D(\varepsilon, s)\big)$, implies that for $\varepsilon$ small enough the integrand changes by a small margin, and since the integral is integer valued, we obtain the claim.\footnote{Alternatively, one may apply the generalised Rouch\'e's theorem \cite[Theorem C.12.]{DyZwbook} to conclude that the sums of \emph{null multlicities} (in the sense of Gohberg-Sigal theory, cf. \cite[Appendix C]{DyZwbook}) over the resonances in the interior of $\gamma$ of operators $\id + D(\varepsilon, s)$ for small enough $\varepsilon$ are constant. By \eqref{eq:argumentprinciple}, we know that these sums of null multiplicities are equal to $\rk \Pi_\varepsilon$, which proves the claim.}
\end{proof}

Note that apriori projections $\Pi_\varepsilon$ are continuous only as functions of $\varepsilon$ with values in $\mathcal{L}(\mathcal{H}_{rG, 1}, \mathcal{H}_{rG})$ and $\mathcal{L}(\mathcal{H}_{rG}, \mathcal{H}_{rG, -1})$, if the resolvents $(P_\varepsilon - s)^{-1}$ are. The maps $\Pi_\varepsilon: \ran \Pi_0 \to \ran \Pi_\varepsilon$ are isomorphisms for small $\varepsilon$ by Lemma \ref{lemma:constrank}. We will show $\varepsilon \mapsto \Pi_{\varepsilon} \in \mathcal{L}(\mathcal{H}_{rG}, \mathcal{H}_{rG})$ is continuous; we follow the argument in \cite[Appendix A]{ChDa19}. Pick a basis $\varphi^j \in \mathcal{H}_{rG, 1}$, $j = 1, \dotso, k = \rk \Pi_0$ of $\ran \Pi_0$, and define $\varphi_\varepsilon^j := \Pi_\varepsilon \varphi^j$; $\varepsilon \mapsto \varphi_{\varepsilon}^j \in \mathcal{H}_{rG}$ is continuous. Define also $\widetilde{\varphi}_\varepsilon^j = \Pi_0 \Pi_\varepsilon \varphi^j$ and note $\varepsilon \mapsto \widetilde{\varphi}_\varepsilon^j \in \mathcal{H}_{rG}$ is also continuous. Let $\nu_\varepsilon^j$ be the dual basis in $\ran \Pi_0$ of $\widetilde{\varphi}_\varepsilon^j$; then $\varepsilon \mapsto \nu_\varepsilon^j \in (\ran \Pi_0)'$ is continuous. Here the prime denotes the dual.  Finally, let $l_\varepsilon^j : = \nu_\varepsilon^j \circ \Pi_0 \circ \Pi_\varepsilon$, continuous as a map $\varepsilon \mapsto l_\varepsilon^j \in \mathcal{H}_{rG}'$. Then we may write
\[\Pi_\varepsilon = \sum_{j = 1}^k \varphi_\varepsilon^j \otimes l_\varepsilon^j.\]
By construction, this map is continuous $\mathcal{H}_{rG} \to \mathcal{H}_{rG}$, for $r > r(s_0) + 1$.

One may further bootstrap this argument as in \cite{ChDa19}, to re-obtain \cite[Lemma 10]{YGB19}:
\begin{lemma}\label{lemma:Pismooth}
	For $r > r(s_0) + k + 1$ and $\varepsilon$ small enough, $\varepsilon \mapsto \Pi_\varepsilon$ is a $C^k$ family of bounded operators on $\mathcal{H}_{rG}$.
\end{lemma}

We are now in good shape to prove some of the basic perturbation statements from the introduction.

\begin{proof}[Proof of Theorem \ref{thm:localresonance}, (1) \& (2)]
		If $X_0 \in \mathcal{X}_\Omega^0$ has non-zero helicity, then for $\varepsilon$ small enough, 
$\mathcal{H}(X_{\varepsilon}) \neq 0$ and we may assume by Lemma \ref{lemma:constrank} that $m_{1, X_{\varepsilon}}(0) \leq m_{1, X_0}(0) = b_1(M)$. Thus by Theorem \ref{thm:volpres}, $\dim \res_{-i\Lapl_{X_{\varepsilon}}, \Omega_0^1}(0) = b_1(M) = m_{1, X_{\varepsilon}}(0)$, so that $X_{\varepsilon}$ is $1$-semisimple, which proves $(1)$.%\footnote{Here we denote the resonant space of $k$-forms at zero with respect to a flow $Y$ by $\res_{k, Y}(0)$.}
		
		The proof of $(2)$ is completely analogous to the proof above and we omit it.
\end{proof}

\section{Proof of Theorem \ref{thm:perturbationflow}}\label{section:closetocontact} In this section we discuss what happens with semisimplicity if we perturb an arbitrary contact Anosov flow. For this purpose, consider $M$ a closed orientable $3$-manifold, and a contact Anosov flow $X$ on $M$. This implies there is a contact $1$-form $\alpha$, such that $\Omega = -\alpha \wedge d\alpha$ is a volume form, $\alpha(X) = 1$ and $\iota_X d\alpha = 0$.

We consider a frame $\{X_{1},X_{2}\}$ of $ \ker \alpha$ (such a frame exists since $M$ is parallelizable) such
that $d\alpha(X_{1},X_{2})=-1$. The dual co-frame $\{\alpha,\alpha_{1},\alpha_{2}\}$ to $\{X,X_{1},X_{2}\}$  satisfies
\[d\alpha = \alpha_2 \wedge \alpha_1, \quad \Omega = - \alpha \wedge d\alpha = \alpha \wedge \alpha_1 \wedge \alpha_2.\]
Next, consider a Riemannian metric $g$ on $M$, making $\{X,X_{1},X_{2}\}$ an orthonormal frame. Observe that
$\Omega^1 = \mathbb{R}\alpha \oplus \Omega_0^1$ and for any $u =  u_1 \alpha_1 + u_2 \alpha_2 \in \mathcal{D}'(M; \Omega_0^1)$, we have for the action of the Hodge star $\ast$ of $g$
\begin{equation}\label{eq:hodge'}
\ast u = u_1 \alpha_2 \wedge \alpha + u_2 \alpha \wedge \alpha_1 = \alpha \wedge (u_2 \alpha_1 - u_1 \alpha_2).
\end{equation}
We introduce the complex structure $J: \Omega_0^1 \to \Omega_0^1$ given by
\[J u := u_2 \alpha_1 - u_1 \alpha_2,\]
so that $\ast u = \alpha \wedge J u$. In particular, we have $\Lapl_X^*u = -\ast \Lapl_X \ast u = 0$ if and only if
\begin{equation}\label{eq:liecoresonant}
\Lapl_X J u = 0.
\end{equation}

Let $Y\in \mathcal{X}_{\Omega}$. Since $Y$ preserves $\Omega$ we may consider the winding cycle map associated to $Y$:
\[W_Y: H^1(M) \to \mathbb{C}, \quad W_Y(\theta) := \int_{M} \theta(Y) \Omega.\]
Clearly $Y$ is null-homologous iff $W_{Y}\equiv 0$. The next lemma characterizes the property of $Y$ being null-homologous in terms of a distinguished resonant state of $X$. Let $\Pi$ denote the spectral projector at zero of $-i\Lapl_{X}$ acting on $\Omega^{1}$ (cf. \eqref{eq:projector}). 
%\todo{G:Is this a good way to say what $\Pi$ is} {\color{blue} M: you think that it's imprecise? I've added a reference. Or you think it's wrong to call it that way?}
%\todo[inline]{G: no, no, I just wasn't sure if I was using the correct terminology.}
Set
\[u:=\Pi \Lapl_{Y}\alpha\in \text{Res}_{-i\Lapl_{X},\Omega^{1}}(0).\]
%\todo[inline]{G: Need to be a little careful here as in principle $\Lapl_{Y}\alpha$ may not be in $\Omega_{0}^{1}$.}
%{\color{blue} M: yeah this is proved at the start of section 7.1., but I've inserted a separate proof in the Lemma below, let me know if that's ok. It's a bit dumb to repeat it though...}

%{\color{blue} M: I've kept writing $\res_{-i\Lapl_X^*}$ through the paper for resonant states with wavefront set in $E_s^*$, as $(-i\Lapl_X)^* = i\Lapl_X^*$ has the same principal symbol as $-i\Lapl_X$. So if we want to consider wavefront set contained in $E_s^*$, we want $-i\Lapl_X^*$. It's not perfect, but I think it works. See e.g. Lemma \ref{lemma:pairingsemisimple}. I've also added "real" in front of theta below.}
%\todo[inline]{G: Fine}

\begin{lemma} We have $\iota_X u = 0$. Let $\theta$ be a (real) smooth closed 1-form and let $\psi \in \mathcal{D}'_{E_s^*}(M)$ be such that
$v:= (J)^{-1}(\theta + d\psi) \in \res_{-i\Lapl_X^*, \Omega_0^1}(0)$. Then
\[\langle u,v\rangle_{L^2}=-W_{Y}(\theta).\]
In particular, $Y$ is null-homologous iff $u=0$.
\label{lemma:uandY}

\end{lemma}

\begin{proof} We may write for some $a, a_1, a_2 \in C^\infty(M)$
\[Y = aX + a_1X_1 + a_2X_2\]
and a calculation shows
\begin{equation}\label{eq:lieY}
\Lapl_Y \alpha = (\iota_Y d + d \iota_Y) \alpha = a_1 \iota_{X_1} d\alpha + a_2 \iota_{X_2} d\alpha + da.
\end{equation}
Therefore, we have
\begin{equation}
	\iota_X u = \Pi \iota_X \Lapl_Y\alpha = \Pi Xa = X \Pi a = 0.
\end{equation}
In the previous equation we used that $\Pi a$ is constant by Theorem \ref{thm:volpres} and that $\Pi$ commutes with $X$. Next we compute using that $*v=\alpha\wedge(\theta+d\psi)$
\begin{align}\label{eq:thetai}
\begin{split}
	\langle{\Lapl_Y \alpha, v}\rangle_{L^2} &= \int_{M} (a_1 \iota_{X_1} d\alpha + a_2 \iota_{X_2} d\alpha + da) \wedge \alpha \wedge (\theta + d\psi )\\
	&= -\int_{M} (a_1 \iota_{X_1} + a_2 \iota_{X_2})(\theta + d\psi) \Omega\\
	&= -\int_{M}  \iota_Y(\theta + d\psi) \Omega = -\int_{M} \iota_Y \theta \Omega = -W_{Y}(\theta).
	\end{split}
\end{align}
Here we used the graded commutation rule for contractions, integration by parts and the following facts: $\theta + d\psi$ is closed, $\iota_X(\theta+ d\psi) = 0$ and $Y$ is volume preserving. By Lemma \ref{lemma:adjoints} it follows that $\Pi^* v = v$. By this and the computation in \eqref{eq:thetai}, it follows that
\[\langle{u, v}\rangle_{L^2} = \langle{\Pi\Lapl_Y \alpha, v}\rangle_{L^2} = \langle{\Lapl_Y \alpha, v}\rangle_{L^2} = -W_{Y}(\theta)\]
as desired.
Clearly, the relation $\langle u,v\rangle=-W_{Y}(\theta)$ implies that if $u=0$, then $Y$ is null-homologous.
If $Y$ is null-homologous, then $\langle u,v\rangle=0$ for all $v$. Since 1-semisimplicity holds for $X$, Lemma \ref{lemma:pairingsemisimple}, implies $u=0$ and the lemma is proved.

\end{proof}

The next lemma provides important information about the pairing between resonant and co-resonant states in the contact case.

\begin{lemma}\label{lemma:pairings}
	Let $\theta$ be a smooth closed one form on $M$. Let $\varphi \in \mathcal{D}'_{E_u^*}(M)$ and $\psi \in \mathcal{D}'_{E_s^*}(M)$ be such that \begin{align}\label{eq:rescond}
	\begin{split}
	u &= \theta + d\varphi \in \res_{-i\Lapl_X, \Omega_0^1}(0),\\
	v &= (J)^{-1}(\theta + d\psi) \in \res_{-i\Lapl_X^*, \Omega_0^1}(0).
	\end{split}
\end{align}
	Then
	\[\re \langle{u, v}\rangle_{L^2} = \re \int_{M} (\theta + d\varphi) \wedge \alpha \wedge (\overline{\theta} + d\overline{\psi}) \leq 0\]
	with equality if and only if $\theta$ is exact, or in other words $u = v = 0$.
\end{lemma}
\begin{proof}
	By \eqref{eq:rescond} we have $\iota_X u = 0$ and $\iota_X v = 0$, so $X\varphi = X\psi = -\theta(X)$. We have the following chain of equalities
	\begin{align}\begin{split}\label{eq:computation}
		\re \langle{u, v}\rangle_{L^2} &= -\int_{M} \re(\theta \wedge \overline{\theta}) \wedge \alpha - \re \int_{M} \varphi d\alpha \wedge \overline{\theta}\\
		&= \re\int_{M} \varphi \overline{\theta}(X) \Omega = -\re \langle{\varphi, X\varphi}\rangle_{L^2} = \im \langle{-iX\varphi, \varphi}\rangle_{L^2}.
		\end{split}
	\end{align}
	Here we used $X\varphi = -\theta(X)$, $\re (\theta \wedge \overline{\theta}) = 0$ and integration by parts.
	
	Assume now $\re \langle{u, v}\rangle_{L^2} \geq 0$. By the computation in \eqref{eq:computation}, Lemma \ref{lemma2} implies $\varphi \in C^\infty(M)$, so $u \in C^\infty(M; \Omega_0^1)$ and Lemma \ref{lemma:ressmoothtrivial} implies $u \equiv 0$ and $\theta$ exact, so also $v \equiv 0$. This completes the proof.
\end{proof}

\subsection{Constructing the splitting resonance.} Let $Y\in {\mathcal X}_{\Omega}$ such that $Y$ is {\it not} null-homologous and consider a perturbation of $X$
\[X_{\varepsilon}=X+\varepsilon Y.\]

%\todo{G: Do we need an ``i"?}
Consider a simple closed curve $\gamma$ around zero, so that no resonances of $-i\Lapl_{X_\varepsilon}$ on $\Omega^1(M)$ cross the curve $\gamma$, for small enough values of the parameter $\varepsilon$. Consider the family of projectors given by 
\begin{equation}\label{eq:projepsilon}
\Pi_\varepsilon := \Pi_{\Lapl_{X_\varepsilon}} = \frac{1}{2\pi i} \oint_\gamma (\lambda + i\Lapl_{X_\varepsilon})^{-1} d\lambda.
\end{equation}
By Lemma \ref{lemma:Pismooth}, the $\Pi_\varepsilon$ are $C^k$ in $\varepsilon$ in suitable topologies. More precisely, we have $\varepsilon \mapsto \Pi_\varepsilon \in \mathcal{L}(\mathcal{H}_{rG}, \mathcal{H}_{rG})$ is $C^k$ for $r > r(0) + k + 1$ (i.e. $r$ large enough).

We will construct the splitting resonant state ``by hand". For that purpose, consider
\[t_\varepsilon = \Lapl_{X_\varepsilon} \Pi_\varepsilon \alpha = \varepsilon \Pi_\varepsilon \Lapl_Y \alpha.\]
Here we used that $\Pi_\varepsilon$ commutes with $\iota_{X_\varepsilon}$ and $d$, which follows since the integral defining $\Pi_\varepsilon$ does so. 
%Now we compute using $X_1, X_2 \in \ker \alpha$ and $\iota_Xd\alpha = 0$
%\begin{equation}\label{eq:lieY}
%\Lapl_Y \alpha = (\iota_Y d + d \iota_Y) \alpha = a_1 \iota_{X_1} d\alpha + a_2 \iota_{X_2} d\alpha + da
%\end{equation}
Our candidate for the splitting resonance is
\[u_\varepsilon := \Pi_\varepsilon \Lapl_Y \alpha.\]
Firstly, we note that $\iota_{X_\varepsilon} u_\varepsilon = 0$, which follows from
\[\iota_{X_\varepsilon} t_\varepsilon = \Lapl_{X_\varepsilon} \Pi_\varepsilon (1 + \varepsilon \alpha(Y)) = 0.\]
This is because $\Pi_\varepsilon f = \frac{1}{\vol(M)} \int_{M} f \Omega$ is constant, which follows from Theorem \ref{thm:volpres}. We also understand that $\Pi_\varepsilon$ acts on forms of any degree, and is given by the expression \eqref{eq:projepsilon}. This implies directly that $\iota_{X_\varepsilon} u_\varepsilon = 0$ for $\varepsilon \neq 0$, and then by continuity we have $\iota_{X_\varepsilon} u_\varepsilon = 0$ for all $\varepsilon$.

Fix now $\varepsilon \neq 0$. Then either exactly one resonance ``splits" by Lemma \ref{lemma:constrank} and Theorem \ref{thm:volpres}, so we must have $\Lapl_{X_\varepsilon} t_\varepsilon = \mu_\varepsilon t_\varepsilon$ for some $\mu_\varepsilon \neq 0$ and thus $\Lapl_{X_\varepsilon} u_\varepsilon = \mu_\varepsilon u_\varepsilon$, or a resonant state does not split, in which case $\Lapl_{X_\varepsilon}t_\varepsilon = 0$ and so $\Lapl_{X_\varepsilon} u_\varepsilon = 0$. Also, we clearly have $\Lapl_{X} u_0 = 0$. 
Therefore, there exists a function $\lambda_\varepsilon$ such that for each small enough $\varepsilon$
\begin{equation}
	\Lapl_{X_\varepsilon} u_\varepsilon = \lambda_\varepsilon u_\varepsilon.
\end{equation}
Hence we may write
\[\lambda_\varepsilon = \frac{\langle{\Lapl_{X_\varepsilon} u_\varepsilon, u^*}\rangle}{\langle{u_\varepsilon, u^*}\rangle},\]
where $u^*$ is a co-resonant 1-form at zero, such that $\langle{u_0, u^*}\rangle \neq 0$. Such a one form exists by Lemma \ref{lemma:pairingsemisimple}. Therefore, for $\varepsilon$ small enough and by continuity the above expression makes sense, so we conclude that $\lambda_\varepsilon$ is in $C^2$ for $\varepsilon$ in an interval around zero. Note that $\lambda_0 = 0$ and that by Lemma \ref{lemma:uandY}, $u_0\neq 0$ since $Y$ is not null-homologous.

\subsection{Proving that $\lambda_\varepsilon \neq 0$.}
We dedicate this subsection to proving that $\lambda_\varepsilon \neq 0$ for $\varepsilon \neq 0$ and we achieve this by looking at the second order derivatives of $\lambda_\varepsilon$ in $\varepsilon$. Recall we have a $C^2$ family of resonant one forms $u_\varepsilon = \Pi_\varepsilon \Lapl_Y \alpha$, corresponding to resonances $-i\lambda_\varepsilon$ for the flow $X + \varepsilon Y$ such that
\begin{align}\label{eq:family}
\begin{split}
	\iota_{X + \varepsilon Y} du_\varepsilon &= \lambda_\varepsilon u_\varepsilon,\\
	\iota_{X + \varepsilon Y} u_\varepsilon &= 0.
	\end{split}
\end{align}
We will denote $u_0$ by $u$ and $\lambda_0$ by $\lambda$, and we apply the same principle to the derivatives of $\lambda$ and $u$ at zero. We want to linearise the equations \eqref{eq:family}, by taking derivatives in $\varepsilon$.

\textbf{First linearisation of \eqref{eq:family}.} We take the first derivative of \eqref{eq:family} to get
\begin{align}\label{eq:lin1}
\begin{split}
	\iota_Y du_\varepsilon + \iota_{X + \varepsilon Y} d\dot{u}_\varepsilon &= \dot{\lambda}_\varepsilon u_\varepsilon + \lambda_\varepsilon \dot{u}_\varepsilon,\\
	\iota_Y u_\varepsilon + \iota_{X + \varepsilon Y} \dot{u}_\varepsilon &= 0.
	\end{split}
\end{align}
Evaluating \eqref{eq:lin1} at $\varepsilon = 0$, we get the system
\begin{align}\label{eq:lin10}
\begin{split}
	\iota_Y du + \iota_{X} d\dot{u} &= \dot{\lambda} u,\\
	\iota_Y u + \iota_{X} \dot{u} &= 0.
	\end{split}
\end{align}
This further simplifies, since $u$ is a resonant state at zero, so by Lemma \ref{lemma:T} we have $du = 0$.  %Pick now a co-resonant state $u^*$ in the kernels of $\Lapl_{X}^*$ and $\iota_X$, such that $\langle{u, u^*}\rangle \neq 0$. This may be done by Lemma \ref{lemma:pairingsemisimple} and since we know semisimplicity holds by \cite{DyZw17}. 
By \eqref{eq:hodge'} we may write $\ast u^* = \alpha \wedge w$, where $w = Ju^*$ and we have $\Lapl_X w = 0$ and $\iota_X w = 0$. Similarly as before, since $w \in \mathcal{D}'_{E_s^*}(M; \Omega_0^1)$ we have $dw = 0$. Therefore, by taking inner product with $u^*$ in \eqref{eq:lin10}, we get
\begin{align*}
	\dot{\lambda} \langle{u, u^*}\rangle &= \langle{\iota_X d\dot{u}, u^*}\rangle = \int_{M} \iota_X d\dot{u} \wedge \alpha \wedge w\\
	&= -\int_{M} d \dot{u} \wedge w = -\int_{M} \dot{u} \wedge dw = 0.
\end{align*}
This implies $\dot{\lambda} = 0$.

\textbf{Second linearisation of \eqref{eq:family}.} By taking the $\varepsilon$ derivative of \eqref{eq:lin1} we get
\begin{align}\label{eq:lin2}
\begin{split}
	2\iota_Y d\dot{u}_\varepsilon + \iota_{X + \varepsilon Y} d \ddot{u}_\varepsilon &= \ddot{\lambda}_\varepsilon u_\varepsilon + 2\dot{\lambda}_\varepsilon \dot{u}_\varepsilon + \lambda_\varepsilon \ddot{u}_\varepsilon,\\
	2\iota_Y \dot{u}_\varepsilon + \iota_{X + \varepsilon Y} \ddot{u}_\varepsilon &= 0.
	\end{split}
\end{align}
We evaluate the equation \eqref{eq:lin2} at $\varepsilon = 0$ to get
\begin{align}\label{eq:lin20}
\begin{split}
	2\iota_Y d\dot{u} + \iota_{X} d \ddot{u} &= \ddot{\lambda} u,\\
	2\iota_Y \dot{u} + \iota_{X} \ddot{u} &= 0.
	\end{split}
\end{align}
Consider the same co-resonant state $u^*$ as above. %such that $\langle{u, u^*}\rangle \neq 0$, $\ast u^* = \alpha \wedge v$, where $\Lapl_X v = 0$, $\iota_X v = 0$ and $dv = 0$. 
Pairing \eqref{eq:lin20} with $u^*$ yields
\begin{align}\label{eq:secondderivative}
	\ddot{\lambda} \langle{u, u^*}\rangle = 2\int_{M} \iota_Y d\dot{u} \wedge \alpha \wedge w + \int_{M} \iota_X d \ddot{u} \wedge \alpha \wedge w.
\end{align}
Now the second integral above is equal to $-\int_{M} d\ddot{u} \wedge w = 0$, by integration by parts.

The first integral is a bit trickier and it is equal to
\begin{multline}\label{eq:secondvarcomp}
%\begin{split}
	\int_{M} \iota_Y d\dot{u} \wedge \alpha \wedge w =  \int_{M} (a_1 \iota_{X_1} + a_2 \iota_{X_2}) d\dot{u} \wedge \alpha \wedge w = \int_{M} (a_1 \iota_{X_1} + a_2 \iota_{X_2}) w  d \dot{u} \wedge \alpha\\
	= \int_{M} w(Y) d\dot{u} \wedge \alpha.
%	\end{split}
\end{multline}
Here we used that $\iota_X d\dot{u} = 0$ by the first linearisation analysis and $\iota_X w = 0$. Note that $\iota_X d\dot{u} = 0$ also implies that $d \dot{u} \wedge \alpha$ is $X$-invariant, so the integral $\int_{M} w(Y) d\dot{u} \wedge \alpha$ may be interpreted as ``some winding cycle".

Observe that $WF(d \dot{u}) \subset WF(\dot{u}) \subset E_u^*$. This follows by differentiating $\Pi_\varepsilon$ at zero to deduce
\[\dot{\Pi}_0 = \frac{1}{2\pi i} \oint_\gamma (\lambda + i\Lapl_X)^{-1} (-i\Lapl_Y) (\lambda + i\Lapl_X)^{-1} d\lambda = i\big(R_H(0) \Lapl_Y \Pi_0 + \Pi_0 \Lapl_Y R_H(0)\big).\]
At this point we recall that $(-i\Lapl_X - \lambda)^{-1} = R_H(\lambda) - \frac{\Pi_0}{\lambda}$. Since $\Pi_0, R_H(0)$ extend to maps $\mathcal{D}'_{E_u^*}(M; \Omega^1) \to \mathcal{D}'_{E_u^*}(M; \Omega^1)$, we have that $\dot{u} = \dot{\Pi}_0 \Lapl_Y \alpha \in \mathcal{D}'_{E_u^*}(M; \Omega^1)$.

By Theorem \ref{thm:volpres} it follows that $d\dot{u} \wedge \alpha = c \Omega$ for some constant $c$. In fact, we have
\begin{align}\label{eq:c} 
\begin{split}
c \vol(M) &= \int_{M} d \dot{u} \wedge \alpha = \int_{M} \dot{u} \wedge d\alpha = -\int_{M} \dot{u}(X) \Omega\\ 
&= \int_{M} u(Y) \Omega = W_Y(u).
\end{split}
\end{align}
In these lines we used the second equation of \eqref{eq:lin10} and $\iota_X u = 0$. Combining \eqref{eq:c}, \eqref{eq:secondderivative} and \eqref{eq:secondvarcomp} we have
\begin{align}\label{eq:secondderivative'}
	\ddot{\lambda} \langle{u, u^*}\rangle = 2c \int_{M} w(Y) \Omega = 2cW_Y(w) = \frac{2W_Y(u) W_Y(w)}{\vol(M)}.
\end{align}
Next we choose a special $u^*$.  Namely, if we write $u=\theta+d\varphi$
for some (real) smooth closed 1-form $\theta$ and $\varphi \in \mathcal{D}'_{E_u^*}(M)$, then we choose $u^*=v$ as in 
Lemma \ref{lemma:pairings}. This ensures that $\langle u,u^*\rangle < 0$ and moreover, by Lemma \ref{lemma:uandY} we have
\[\langle u,u^*\rangle=-W_{Y}(\theta)< 0.\]
Hence \eqref{eq:secondderivative'} simplifies to
\[\ddot{\lambda} = \frac{-2W_Y(\theta)}{\vol(M)} < 0.\]
By the symmetry of the Pollicott-Ruelle resonance spectrum, we have that $\lambda_\varepsilon$ is real, since otherwise we would contradict Lemma \ref{lemma:constrank}. We conclude by Taylor's theorem 
\[\lambda_\varepsilon = \varepsilon^2 \Big(-\frac{W_Y(\theta)}{ \vol(M)} + O(\varepsilon)\Big).\]
In particular $\lambda_\varepsilon$ is negative (so non-zero) for sufficiently small $\varepsilon \neq 0$. Therefore, the resonance $-i\lambda_\varepsilon$ of $-i\Lapl_{X_\varepsilon}$ splits to the upper half-plane and $0$ is a strict local maximum for $\lambda_\varepsilon$.  This completes the proof of Theorem \ref{thm:perturbationflow}.

We conclude this section with:

\begin{proof}[Proof of the first part of Corollary \ref{cor:zeta'}]
	By Corollary \ref{cor:3dfactor}, the order of vanishing of the Ruelle zeta function at zero is equal to $m_1(0) - m_0(0) - m_2(0)$. By Theorem \ref{thm:volpres}, we know $m_2(0) = m_0(0) = 1$ and by Theorem \ref{thm:perturbationflow} and Lemma \ref{lemma:constrank} we have $m_1(0) = b_1(M) - 1$ for small enough non-zero $\varepsilon$. This concludes the proof.
\end{proof}

%\begin{rem}\rm {\color{blue} M: Leave or not?}
%	Lemma \ref{lemma:pairings} actually gives a kind of Hodge duality statement for contact flows, i.e. it shows that for $u^* = J^{-1}R^*u$ we have $\langle{u, u^*}\rangle < 0$ unless $u = 0$.
	
%	On arbitrary manifolds $M$, the proof of Lemma \ref{lemma:pairings} generalises to contact Anosov flows and forms in $\res_{-i\Lapl_X} (0) \cap \ker d$. Note also that $H^1(M) \hookrightarrow \res_{\Lapl_X. \Omega_0^1}(0)$, but we do not always have $d \res_{\Lapl_X. \Omega_0^1}(0) = 0$ due to the work \cite{KW19}, in the case $\dim M = 5$. So we may write in any dimension
	%\[\res^{(1)}_{\Lapl_X, \Omega_0^1}(0) = \ker d \oplus A^{(1)}\]
	%where $A^{(1)}$ is some complement with no closed non-trivial forms.
	
%	More precisely, we may write $\ast u = \pm \alpha \wedge (d\alpha)^{(n - 1)} \wedge J^*u$ for $u$ section on $\Omega_0^1$ and $J^*:\Omega_0^1 \to \Omega_0^1$ a suitable complex structure, given by pulling back via the almost complex structure $J$ w.r.t. $(\ker \alpha, d\alpha)$ and the Riemannian inner product. Then the proof above gives that the pairing of resonances and co-resonances on the $\ker d$ part is non-degenerate. %How to exploit this some more?
%\end{rem}

\section{Time changes}\label{section:timechange}
In this section we consider the transformation $X \mapsto \widetilde{X} = f X$, where $X$ is an Anosov vector field and $f > 0$ a positive smooth function and call such a transformation a \emph{time-change}. By \cite[Lemma 2.1.]{DMM86}, we have that $\widetilde{X}$ is also Anosov and moreover, its stable and unstable bundle $\widetilde{E}^s$ and $\widetilde{E}^u$ are given by
\begin{equation}\label{eq:timechange}
	 \widetilde{E}^s = \{Z + \theta(Z)X: Z \in E^s\}.%, \quad \widetilde{E}^u = \{Z + \theta(Z)X: Z \in E^u.\}
\end{equation} 
Here the continuous one form $\theta$ is given by solving $\Lapl_X(f^{-1}\theta) = f^{-2}df$. Therefore, we notice that $\widetilde{E}_u^* = (\widetilde{E}^s \oplus \mathbb{R}\widetilde{X})^* = E_u^*$ and $\widetilde{E}_s^* = (\widetilde{E}^u \oplus \mathbb{R}\widetilde{X})^* = E_s^*$, where we used \eqref{eq:timechange}. This means that the resonant states associated to the flow $f X$ lie in suitable spaces $\mathcal{D}'_{E_u^*}$, which will be very convenient. %We first need a Lemma about the mapping properties of $P$ at resonances.

We begin by recasting Lemma \ref{lemma:pairingsemisimple} to the case of one forms and consider a time-change.

\begin{prop}\label{prop:semisimplicitypairing}
	Let $X$ be an Anosov flow on a manifold $M$ and let $f > 0$ be a positive smooth function. Then $\Lapl_{fX}$ acting on $\Omega_0^1$ is semisimple at zero if and only if the pairing
	\begin{equation}\label{eq:pairinglie}
		\res^{(1)}_{-i\Lapl_X, \Omega_0^1}(0) \times \res^{(1)}_{-i\Lapl_X^*, \Omega_0^1}(0) \to \mathbb{C}, \,\, (u, v) \mapsto \big\langle{u/f, v}\big\rangle_{L^2(M; \Omega^1)}
	\end{equation}
	is non-degenerate. %Equivalently, $\Lapl_{fX}$ is semisimple at zero if and only if
%	\begin{equation}\label{eq:condsemisimpl'}
%		\Big(\frac{1}{f} \cdot \res_{\Lapl_X, \Omega_0^1}^{(1)}(0)\Big) \cap \ker \Pi_0 = \{0\}
%	\end{equation}
\end{prop}
\begin{proof}
	Let us determine the appropriate resonant spaces of $\Lapl_{fX}$ and $\Lapl_{fX}^*$ at zero. Note first that $\ker \Lapl_{fX} = \ker \Lapl_X$ on $\mathcal{D}'_{E_u^*}(M; \Omega_0^1)$, since time-changes preserve the $E_u^*$ set. Next, we compute $\Lapl_{fX}^* = \Lapl_X^*(f\cdot)$ on $\Omega_0^1$, with respect to a fixed smooth inner product (e.g. given by a metric). Therefore, we have
	\[\res_{-i\Lapl_{fX}^*, \Omega_0^1}^{(1)}(0) = \frac{1}{f} \res_{-i\Lapl_{X}^*, \Omega_0^1}^{(1)}(0).\]
	Thus the non-degeneracy of the pairing between resonances and co-resonances is equivalent to the non-degeneracy of \eqref{eq:pairinglie} and applying Lemma \ref{lemma:pairingsemisimple} finishes the proof.
	
%	For the second statement, the condition $\Pi_0(u/f) = 0$ for some $u \in \res_{-i\Lapl_X, \Omega_0^1}^{(1)}(0)$, is equivalent to 
%	\[\langle{u/f, \Pi_0^* \varphi}\rangle = 0\]
%	for all $\varphi \in C^\infty(M; \Omega_0^1)$. By Lemma \ref{lemma:adjoints} and since $\Pi_0^* = \Pi'_0$ and these operators have finite rank so $\im \Pi_0' = \im \Pi_0'|_{C^\infty}$, the non-degeneracy of \eqref{eq:pairinglie} is seen to be equivalent to the condition \eqref{eq:condsemisimpl'}. This finishes the proof.
\end{proof}
%\blue{ M: We need to fix the notation for the resonance spaces: most people write $\res_X(\lambda)$ for distributions in the kernel of $(X + \lambda)^k$ for some $k$. But \cite{DyZw17, FaSj11} and some other works use the complex operator $P = -iX$ instead.. But this is not a problem, especially at zero.}

\subsection{Time-changes of the geodesic flow on a hyperbolic surface.}\label{subsec:time-changes} The aim of this subsection is to explicitly specify the equations for one forms in the kernel of $\Lapl_X$ on the unit sphere bundle $M = S\Sigma$ of a closed hyperbolic surface $\Sigma$. We start by considering the case of general variable curvature and use the orthonormal frame $\{\alpha, \beta, \psi\}$ constructed in Subsection \ref{subsec:structureeqns}.

Let $u \in \mathcal{D}'(M; \Omega_0^1)$. Then $u = b \beta + f \psi$ for some $b, f \in \mathcal{D}'(M)$ and we have
\begin{align*}
	%0 &= db \wedge \beta + b d\beta + df \wedge \psi + f d\psi\\
	%0 &= \big(X(b)\alpha + V(b) \psi\big) \wedge \beta - b\psi \wedge \alpha + \big(X(f) \alpha + H(f)\beta\big) \wedge \psi - fK \alpha \wedge \beta\\
	du = \alpha \wedge \beta \big(X(b)- fK)\big) + \beta \wedge \psi \big(H(f) - V(b)\big) + \alpha \wedge \psi \big(b + X(f)\big).
\end{align*}
Therefore, $du = 0$ implies that
\begin{align}\label{eq:res1}
\begin{split} X(b) &= Kf,\\
X(f) &= -b,\\
H(f) &= V(b).
\end{split}
\end{align}
The first two equations are coming from $\iota_X d u = 0$. The third is an additional one, which we know holds if $u \in \mathcal{D}'_{E_u^*}(M; \Omega_0^1)$ and $\iota_X du = 0$; it can be explained as an additional horocyclic invariance (cf. \cite{FaGu17} and below).

%The third one comes from knowing that $du = 0$ forand it can be explained as the additional horocyclic invariance (see below, cf. \cite{FaGu17}). , and assume $u \in \mathcal{D}'_{E_u^*}(M; \Omega_0^1)$.

Now we specialise to $K = -1$, i.e. the case of hyperbolic surfaces. Then in $\{\beta, \psi\}$ co-frame spanning $\Omega_0^1$, the operator $\Lapl_X$ may be written as
\[\Lapl_X = X \times Id + \begin{pmatrix}
0 & 1\\
1 & 0
\end{pmatrix}\]
%acting on the bundle $\Omega_0^1$ spanned by $\beta$ and $\psi$. 
and the first two equations in \eqref{eq:res1} then read %$\Lapl_X u = 0$, where $u = b \beta + f \psi$ for some $b, f \in \mathcal{D}'_{E_u^*}(M)$ is the same as 
\begin{align*}
	(X - 1) (b - f) &= 0,\\
	(X + 1) (b + f) &= 0.
\end{align*}
Thus $f = -b$ as there are no resonances with positive imaginary part, since $X$ is volume-preserving.\footnote{This can be seen from \eqref{eq:laplace}, since $e^{-itP} = \varphi_{-t}^*$ is an isometric isomorphism on $L^2(M)$ and so the integral defining the resolvent converges for $\im \lambda > 0$.} %Therefore, $(X - 1)^2 b = 0$. 
The third equation in \eqref{eq:res1} now gives $U_-b = 0$, where $U_- = H + V$ is the horocyclic vector field spanning $E_u$. Now we may also write, where the adjoint is with respect to the Sasaki metric on $S\Sigma$ %(or equivalently $\Omega$)
\[\Lapl_X^* = -X \times Id + \begin{pmatrix}
0 & 1\\
1 & 0
\end{pmatrix}.\]
Therefore $\Lapl_X^* v = 0$, where $v = b'\beta + f'\psi$ for some $b', f' \in \mathcal{D}'_{E_s^*}(M)$, is the same as %(note the difference in the minus sign to above)
\begin{align*}
	(-X + 1) (b' + f') &= 0,\\
	(-X - 1) (b' - f') &= 0.
\end{align*}
Since we are looking at the vector field $-X$, no resonances with positive imaginary part gives $f' = -b'$ and so $(X + 1) b' = 0$. The third equation in \eqref{eq:res1} then reads $U_+ b' = 0$, where $U_+ = H - V$ spans the $E_s$ bundle.

Therefore, we have
\begin{align}\label{eq:correspondence}
\begin{split}
	\res_{-i\Lapl_X}^{(1)}(0) &= \big\{b(\beta - \psi) \in \mathcal{D}'(M) : (X - 1)b = 0, \,\, (H + V)b = 0\big\},\\
	\res_{-i\Lapl_X^*}^{(1)}(0) &= \big\{b(\beta - \psi) \in \mathcal{D}'(M) : (X + 1)b = 0, \,\, (H - V)b = 0\big\}.
	%&= (\beta - \psi) \res_{X}^{(1)}(-1)
	% &= (\beta - \psi) \res_{X^*}^{(1)}(-1)
\end{split}
\end{align}
Note that we may drop the wavefront set conditions, since they follow from the equations being satisfied. We remark that since we know $-i\Lapl_X$ at $0$ is semisimple by \cite{DyZw17}, then so is $-iX$ at $-i$ by the correspondence \eqref{eq:correspondence} and $\dim \res_{-iX}(-i) = b_1(M)$. Alternatively, we may use \cite[Theorem 1]{GHW18} to deduce semisimplicity even at the special point $-i$ for hyperbolic surfaces.%Now the resonant and co-resonant pairing \eqref{eq:pairinglie} is equivalent to the resonant and co-resonant pairing between resonant functions of $X$ at $1$:
\begin{prop}
Let $f \in C^\infty(M)$ and $f > 0$. Semisimplicity for $-i\Lapl_{fX}$ at zero acting on $\Omega_0^1$ is equivalent to the non-degeneracy of the following pairing
\begin{equation}\label{eq:pairinghyp}
\res_{-iX}^{(1)}(-i) \times \res_{iX}^{(1)}(-i), \,\, (b_1, b_2) \mapsto \langle{b_1/f, b_2}\rangle_{L^2(M)}.
\end{equation}
\end{prop}
\begin{proof}
	The proof is based on the correspondence \eqref{eq:correspondence} and Proposition \ref{prop:semisimplicitypairing}. Then for $b_1(\beta - \psi) \in \res_{-i\Lapl_{fX}}^{(1)}(0)$ and $\frac{b_2}{f}(\beta - \psi) \in \res_{-i\Lapl_{fX}^*}^{(1)}(0)$, we have
	\[\big\langle{b_1(\beta - \psi), b_2/f(\beta - \psi)}\big\rangle_{L^2(M; \Omega^1)} = 2\big\langle{b_1, b_2/f }\big\rangle_{L^2(M)}.\]
	This proves that the pairing \eqref{eq:pairinghyp} is equivalent to the pairing \eqref{eq:pairinglie}, which finishes the proof.
\end{proof}
In the next sections, we would like to find out more about the pairing \eqref{eq:pairinghyp}, similar to \cite{DFG15, GHW18}, where a pairing formula for generic resonances is proved. %and (b) an alternative way to prove semisimplicity for $X$ at $1$ directly.

\begin{rem}\rm\label{rem:speconeformshyperbolic}
	Using the decomposition $u = a \alpha + b \beta + f \psi$, by \eqref{eq:res1} it may be seen that $(\Lapl_X + s)u = 0$ is equivalent to $(X + 1 + s)(b + f) = 0$, $(X - 1 + s)(b - f) = 0$ and $(X + s)a = 0$. This enables us to determine the resonance spectrum of $\Lapl_X$ on one forms from the resonance spectrum of $X$ on functions, using the works of \cite{DFG15, GHW18}. In particular, for $\re s > -1$ we obtain $b + f = 0$, which suffices to determine the spectrum on the left in Figure \ref{fig:splitting}. The small and large eigenvalues in this figure are in the sense of \cite{BMM16}.
\end{rem}

\subsection{Reduction to distributions on the boundary.} We follow the notation from \cite[Section 3]{DFG15}. We consider the hyperboloid model 
\[\mathbb{H}^2 = \{x = (x_0, x_1, x_2) = (x_0, x') \in \mathbb{R}^3 : \langle{x, x}\rangle_{\mathcal{M}} = x_0^2 - x_1^2 - x_2^2 = 1, \, x_0 > 0\}\]
of hyperbolic geometry, equipped with the Riemannian metric $-\langle{\cdot, \cdot}\rangle_{\mathcal{M}}$, restricted to $T\mathbb{H}^2$. Here $\langle{\cdot, \cdot}\rangle_{\mathcal{M}}$ is called the Lorentzian metric. We also consider the action the isometry group $G = PSO(1, 2)$ of $\mathbb{H}^2$, consisting of matrices preserving the Lorentzian metric, orientation and the sign of $x_0$. This action extends to an action on the unit sphere bundle $S\mathbb{H}^2$, since $G$ consists of isometries and in fact $G \ni \gamma \mapsto \gamma \cdot (1, 0, 0, 0, 1, 0) \in S\mathbb{H}^2$ is a diffeomorphism. We also have explicitly%Then we may write for the unit sphere bundle $S\mathbb{H}^2$
\begin{equation}\label{eq:identify}
	S\mathbb{H}^2 = \{(x, \xi) \in \mathbb{H}^2 : x, \xi \in \mathbb{R}^3,\, \langle{\xi, \xi}\rangle_{\mathcal{M}} = -1,\, \langle{x, \xi}\rangle_{\mathcal{M}} = 0\}.
\end{equation}
We will write $\varphi_t$ for the geodesic flow on $S\mathbb{H}^2$ and $X$ for the geodesic vector field. In the identification \eqref{eq:identify}, we may write
\[X = \xi \cdot \partial_x + x \cdot \partial_\xi.\]
Therefore the geodesic flow on $S\mathbb{H}^2$ may be explicitly written as
\begin{equation}\label{eq:geodesicflowhyperboloid}
	\varphi_t(x, \xi) = (x \cosh t + \xi \sinh t, x\sinh t + \xi \cosh t).
\end{equation}
We may compactify $\mathbb{H}^2$ to the closed unit ball $\overline{B^2}$ by embedding it with the map $\psi_0(x) = \frac{x'}{x_0 + 1}$ and we call $S^1$ bounding $B^2$ the \emph{boundary at infinity}. Note that to a point $\nu \in S^1$ we may associate a ray $\{(s, s\nu) : s > 0\}$, which is asymptotic to the hyperboloid ray $\{(\sqrt{1 + s^2}, s\nu): s > 0\}$. The action of $G$ extends to an action on the boundary at infinity $S^1$ as follows. Let $\gamma \in G$ and $\nu \in S^1$. Then the matrix action on $\mathbb{R}^3$
\begin{equation}
	\gamma \cdot (1, \nu) = N_\gamma(\nu) \big(1, L_\gamma(\nu)\big)
\end{equation}
%\todo{G: define $N_{\gamma}$ and $L_{\gamma}$}
defines an action of $\gamma \in G$ on $S^1$ via $L_\gamma$. It also defines the multiplicative map $N_\gamma: S^1 \to \mathbb{R}_+$. 

%Let $M = \Gamma \backslash \mathbb{H}^2$ be a compact hyperbolic surface, where $\Gamma \subset PSO(1, 2)$ is a discrete subgroup. Then we may identify the unit sphere bundle as $SM = \Gamma \backslash S\mathbb{H}^2$.
Denote by $\pi:S\mathbb{H}^2 \to \mathbb{H}^2$ the footpoint projection. We will consider the mappings
\begin{equation}\label{eq:B+-}
	B_\pm(x, \xi): S\mathbb{H}^2 \to S^1,\quad B_\pm(x, \xi) = \lim_{t \to \pm\infty}\pi(\varphi_t(x, \xi)).
\end{equation}
%\todo{G: Is $G=\Gamma$?}
The limit in \eqref{eq:B+-} is interpreted as the point of intersection of the geodesic starting at $x$ and with tangent vector $\xi$ with the boundary at infinity. We introduce also 
\begin{equation}\label{eq:Phi+-}
\Phi_\pm: S\mathbb{H}^2 \to \mathbb{R}_+, \quad \Phi_\pm(x, \xi) := x_0 \pm \xi_0 > 0.
\end{equation}
In fact, then we can write for any $(x, \xi) \in S\mathbb{H}^2$
\begin{equation}\label{eq:x+-xi}
x \pm \xi = \Phi_\pm(x, \xi) \big(1, B_\pm(x, \xi)\big).
\end{equation}
%This follows from the explicit formula for the geodesic flow on $S\mathbb{H}^2$
The maps $B_\pm$ and $\Phi_\pm$ have nice interactions with the geodesic vector field $X$ and the horocyclic vector fields $U_\pm$, defined in Subsection \ref{subsec:time-changes}. By this we mean that 
\begin{equation}\label{eq:B+-derivatives}
	dB_\pm \cdot X = 0, \quad U_\pm B_\pm = 0.
\end{equation} 
The first equation holds since $B_\pm$ is constant along $X$ and the second one since $B_\pm$ is constant along horospheres. We also have
\begin{equation}\label{eq:phiderivatives}
	X \Phi_\pm = \pm \Phi_\pm, \quad U_\pm \Phi_\pm = 0.
\end{equation} 
Here, the first equation follows from $\Phi_\pm(\varphi_t(x, \xi)) = e^{\pm t} \Phi_\pm(x, \xi)$, which is true by equation \eqref{eq:geodesicflowhyperboloid}. The second one also follows from a computation. Finally, since $\langle{x + \xi, x - \xi}\rangle_{\mathcal{M}} = 2$ and by \eqref{eq:x+-xi}, for $(x, \xi) \in S\mathbb{H}^2$, we have %and the subscript $M$ denotes the Minkowski metric
\begin{equation}\label{eq:+-identity}
	\Phi_+(x, \xi) \Phi_-(x, \xi) \big(1 - B_+(x, \xi) \cdot B_-(x, \xi)\big) = 2.
\end{equation}
The maps $\Phi_\pm$ and $B_\pm$ are $G$-equivariant in the following sense. We have
\begin{equation}\label{eq:+-invariance}
	B_\pm\big(\gamma \cdot (x, \xi)\big) = L_\gamma\big(B_\pm(x, \xi)\big), \quad \Phi_\pm\big(\gamma \cdot(x, \xi)\big) = N_\gamma\big(B_\pm(x, \xi)\big) \Phi_\pm(x, \xi).
\end{equation}
Now the Jacobian of the map $L_\gamma: S^1 \to S^1$ may be computed explicitly and is given by
\begin{equation}\label{eq:jacobianG}
	\big\langle{dL_\gamma(\nu) \cdot \zeta_1, dL_\gamma(\nu) \cdot \zeta_2}\big\rangle_{\mathbb{R}^2} = 
	N_\gamma(\nu)^{-2} \langle{\zeta_1, \zeta_2}\rangle_{\mathbb{R}^2}, \quad \zeta_1, \zeta_2 \in T_\nu S^1.
\end{equation}
%Here $\zeta_1, \zeta_2 \in T_\nu S^1$. 

Consider $\Sigma = \Gamma \backslash \mathbb{H}^2$ a compact hyperbolic surface, where $\Gamma \subset PSO(1, 2)$ is a discrete subgroup. Then we may identify the unit sphere bundle as $S\Sigma = \Gamma \backslash S\mathbb{H}^2$. We introduce the space of boundary distributions as
\begin{equation}\label{eq:bd0}
	\Bd^0(\lambda) = \{w \in \mathcal{D}'(S^1) \mid L_\gamma^* w(\nu) = N_\gamma^{-\lambda}(\nu) w(\nu), \,\, \gamma \in \Gamma, \,\, \nu \in S^1\}.
\end{equation}
The generator $X$ of the geodesic flow descends to $S\Sigma$ and we define the \emph{first band resonant states} by 
\[\res_X^0(\lambda) = \{u \in \mathcal{D}'_{E_u^*}(S\Sigma) \mid (X + \lambda)u = 0, \,\, U_-u = 0\}.\]
We similarly introduce the first band co-resonant states via (cf. Subsection \ref{subsec:rescores})
\[\res_{X^*}^0(\lambda) = \{u \in \mathcal{D}'_{E_s^*}(S\Sigma) \mid (X - \overline{\lambda})u = 0, \,\, U_+u = 0\}.\]
Then we have the correspondence, valid for all $\lambda \in \mathbb{C}$ proved in \cite[Lemma 5.6]{DFG15}, that we prove here for completeness. Note that by $\Phi_\pm^\lambda$ for $\lambda \in \mathbb{C}$ we simply mean the exponentiation of the function $\Phi_\pm > 0$ by the exponent $\lambda$.
\begin{lemma}\label{lemma:boundarydistributions}
	Let $\pi_\Gamma: S\mathbb{H}^2 \to S\Sigma$ be the natural projection. Then
	\begin{equation}
		\pi^*_\Gamma \res_X^0(\lambda) = \Phi_-^\lambda B_-^* \Bd^0(\lambda).
	\end{equation}
	Similarly we have, for the space of co-resonant states
	\begin{equation}
		\pi^*_\Gamma \res_{X^*}^0(\lambda) = \Phi_+^{\overline{\lambda}} B_+^* \Bd^0(\overline{\lambda}).
	\end{equation}
	We also have $\overline{\Bd^0(\lambda)} = \Bd^0(\overline{\lambda})$.
\end{lemma}
%\todo[inline]{G: What is $\Phi^{\lambda}_{\pm}$?}
\begin{proof}
	Let $w \in \Bd^0(\lambda)$ and put $v = \Phi_-^\lambda B_-^*w \in \mathcal{D}'(S\mathbb{H}^2)$ (pullback of distributions under submersions is well-defined, see \cite[Corollary 7.9]{GrSj94}). We use now the invariance properties $\Phi_\pm$ and $B_\pm$ given by \eqref{eq:+-invariance} to prove $v$ is $\Gamma$-invariant. For $\gamma \in \Gamma$ we have
	\[\gamma^* v = (\gamma^*\Phi_-)^\lambda \gamma^* B_-^* w = B_-^* (N_\gamma)^\lambda \Phi_-^\lambda B_-^* L_\gamma ^* w = \Phi_-^\lambda B_-^* w = v.\] 
	Thus $v$ is $\Gamma$-invariant and descends to $\mathcal{D}'(SM)$.
	%Then using $B_\pm\big(\gamma \cdot(x, \xi)\big) = L_\gamma\big(B_\pm(x, \xi)\big)$ and $\Phi_\pm\big(\gamma \cdot (x, \xi)\big) = N_\gamma\big(B_\pm(x, \xi)\big) \Phi_\pm(x, \xi)$, together with the invariance property of $w$, we get that $v$ is $\Gamma$-invariant and so descends to $\mathcal{D}'(SM)$.

	Now using equations \eqref{eq:B+-derivatives} and \eqref{eq:phiderivatives}, we obtain directly that $(X + \lambda)v = 0$ and $U_- v = 0$. This proves $\Phi_-^\lambda B_-^* \Bd^0(\lambda) \subset \pi_\Gamma^*\res_X^0(\lambda)$ (the wavefront set condition on $v$ follows from \cite[Chapter 7]{GrSj94}). The other direction follows by reversing the steps above and noting that a function (distribution) invariant by $X$ and $U_-$ is immediately a pullback by $B_-$. The statement about co-resonant states follows similarly.
\end{proof}

We now introduce the set of coordinates $(\nu_-, \nu_+, s) \in (S^1 \times S^1)_\Delta \times \mathbb{R}$ on $S\mathbb{H}^2$, yielding a diffeomorphism $F: (S^1 \times S^1)_\Delta \times \mathbb{R} \to S\mathbb{H}^2$, and given by identification
\begin{equation}\label{eq:coords}
	(\nu_-, \nu_+, s) = \Big(B_-(x, \xi), B_+(x, \xi), \frac{1}{2} \log \frac{\Phi_+(x, \xi)}{\Phi_-(x, \xi)}\Big).
\end{equation}
Here $(S_1 \times S_1)_\Delta$ denotes the torus $S^1 \times S^1$ without the diagonal $\Delta$. The coordinates \eqref{eq:coords} can be interpreted as: $(\nu_-, \nu_+)$ parameterizes the geodesic $\gamma$ starting at $\nu_-$ and ending at $\nu_+$ and $s$ is the parameter on this geodesic, such that $\gamma(-s)$ is the point on $\gamma$ closest to $e_0 = (1, 0, 0)$ (or $0$ in the disk model). The geodesic flow in these coordinates is simply $\varphi_t: (\nu_-, \nu_+, s) \mapsto (\nu_-, \nu_+, s + t)$.

The coordinates \eqref{eq:coords} enable us to write a product of distributions in resonant and co-resonant spaces more explicitly, but we first require an explicit computation of the Jacobian of the change of coordinates $(x, \xi) \to (\nu_-, \nu_+, s)$.

\begin{lemma}\label{lemma:jacobian}
	For the coordinate system introduced in \eqref{eq:coords}, we have the equality
	\begin{equation}\label{eq:Jacobian}
		F^*(dx d\xi) = \frac{2 d\nu_- d\nu_+ ds}{|\nu_- - \nu_+|^2} . %\frac{1}{2} \Phi_+(x, \xi) \Phi_-(x, \xi) ds d\nu_- d\nu_+
	\end{equation}
\end{lemma}
\begin{proof}
This is the content of \cite[Theorem 8.1.1]{Ni}.
\end{proof}

\begin{rem}\rm
	The Jacobian popping up in Lemma \ref{lemma:jacobian} is well-known and the current in \eqref{eq:Jacobian} is called the \emph{Liouville current}. 
%There should be a more elegant proof of \eqref{eq:Jacobian}, e.g. using the group invariance of the form $\frac{F^*(dx d\xi)}{ds}$ on $(S^1 \times S^1)_\Delta$, but we were not able to find a good reference for this.
\end{rem}

We now prove that the invariant distributions formed as products of resonant and co-resonant states have a very nice form in the coordinates \eqref{eq:coords}.

\begin{prop}\label{prop:+-representation}
	Let $w_1 \in \Bd^0(\lambda)$ and $w_2 \in \Bd^0(\overline{\lambda})$, and consider the invariant distributions $v_1 = \Phi_-^\lambda B_-^*w_1$ and $v_2 = \Phi_+^{\overline{\lambda}} B_+^* w_2$ constructed in Lemma \ref{lemma:boundarydistributions}. Then the product distribution in $(\nu_-, \nu_+, s)$ coordinates takes the form\footnote{Formally, by \eqref{eq:+-product} we mean an equality in the sense of $0$-currents. More explicitly, we mean an equality in the sense of distributions $\big\langle{2^{2\lambda + 1}\frac{w_1(\nu_-) \overline{w}_2(\nu_+)}{|\nu_- - \nu_+|^{2(\lambda + 1)}}, f}\big\rangle_{(S^1 \times S^1)_\Delta \times \mathbb{R}} = \big\langle{v_1\overline{v}_2, f \circ F^{-1}}\big\rangle_{S\mathbb{H}^2}$.}
	\begin{align}\label{eq:+-product}
	F^*\big((v_1\overline{v}_2)(x, \xi) dx d\xi\big) = 2^{2\lambda + 1}\frac{w_1(\nu_-) \overline{w}_2(\nu_+)}{|\nu_- - \nu_+|^{2(\lambda + 1)}}d\nu_- d\nu_+ ds.
	\end{align}
	In particular, for $\lambda = -1$ the product $F^*(v_1 \overline{v}_2)$ extends to a distribution on $S^1 \times S^1 \times \mathbb{R}$.
\end{prop}
\begin{proof} 
By definition, we have the following expression for the product $v_1 \overline{v}_2$
\begin{equation}\label{eq:product}
	\big(v_1 \overline{v}_2\big)(x, \xi) = \big(\Phi_-(x, \xi) \Phi_+(x, \xi)\big)^\lambda B_-^*w_1(x, \xi) B_+^*\overline{w}_2(x, \xi).
\end{equation}
Now changing the coordinates to $(\nu_-, \nu_+, s)$ given in \eqref{eq:coords} and by using the identity \eqref{eq:+-identity} we get
\begin{equation}\label{eq:product'}
	F^*(v_1\overline{v}_2\big)(\nu_-, \nu_+, s) = 2^\lambda (1 - \nu_- \cdot \nu_+)^{-\lambda} w_1(\nu_-) \overline{w}_2(\nu_+) = 2^{2\lambda} \frac{w_1(\nu_-) \overline{w}_2(\nu_+)}{|\nu_- - \nu_+|^{2\lambda}}.
\end{equation}
Using the Jacobian computation in Lemma \ref{lemma:jacobian}, we establish \eqref{eq:+-product}.
%We are interested in the case $\lambda = -1$. Then
%\[\big(v_1\overline{v}_2\big)(\nu_-, \nu_+, s) = \frac{1}{4} |\nu_- - \nu_+|^2 w_1(\nu_-) \overline{w}_2(\nu_+)\] %\in \mathcal{D}'(S\mathbb{H}^2)\]

%This finally gives, using \eqref{eq:product} for $\lambda = -1$
%\begin{equation}\label{eq:lambda=-1}
%	(v_1\overline{v}_2)(x, \xi) dx d\xi = \frac{1}{2} w_1(\nu_-) \overline{w}_2(\nu_+) ds d\nu_- d\nu_+
	 %&= 1/4 (P(x, \nu_-) + P(x, \nu_+)) w_1(\nu_-) \bar{w}_2(\nu_+) ds d\nu_- d\nu_+
	 %(v_1\bar{v}_2)(x, \xi) dx d\xi &= 1/4 (\Phi_+(x, \xi) + \Phi_-(x, \xi)) w_1(\nu_-) \bar{w}_2(\nu_+) ds d\nu_- d\nu_+\\
	 %&= 1/4 (P(x, \nu_-) + P(x, \nu_+)) w_1(\nu_-) \bar{w}_2(\nu_+) ds d\nu_- d\nu_+
%\end{equation}

In the special case $\lambda = - 1$, using \eqref{eq:+-product} we may write 
\begin{equation}\label{eq:-1product}
	F^*\big(v_1\overline{v}_2 (x, \xi) dx d\xi\big) = \frac{1}{2} w_1(\nu_-) \overline{w}_2(\nu_+) ds d\nu_- d\nu_+.
\end{equation}
In particular, for $\lambda = -1$ the distribution $F^*(v_1 \overline{v}_2)$ extends to a distribution on the space $S^1 \times S^1 \times \mathbb{R}$.
%More generally, for any $\lambda$ we have
%\begin{align*}
%	(v_1\overline{v}_2)(x, \xi) dx d\xi = 2^{2\lambda + 1}\frac{w_1(\nu_-) \overline{w}_2(\nu_+)}{|\nu_- - \nu_+|^{2(\lambda + 1)}}d\nu_- d\nu_+ ds
%\end{align*}
%$\Phi_-(x, \xi) + \Phi_+(x, \xi) = 2x_0 \geq 2$
%This explains the appearance of the $\lambda + 1$ factor in the exponent of the denominator, present in the \cite{AZ07} paper for Patterson-Sullivan distributions.

%For $\lambda = -1$, the point is that the singular part of the measure $dx d\xi$ cancels with the coefficients of the distribution $v_1 \overline{v}_2$, so in particular we obtain a distribution that extends to the whole of $S^1 \times S^1 \times \mathbb{R}$. 
\end{proof}

\begin{rem}\rm
The distributions in \eqref{eq:+-identity} are called Patterson-Sullivan type distributions. See \cite{AZ07} for more details, where the particular case of $\lambda = -\frac{1}{2} + ir_j$ is studied, in connection to eigenvalues of $\Delta$ on $\Sigma$ with eigenvalue $\frac{1}{4} + r_j^2$. Note however there is an extra factor of $|\nu_- - \nu_+|^2$ compared to \eqref{eq:product'}, obtained by changing coordinates according to \eqref{eq:coords}.%There is however an offset by a factor of $|\nu_- - \nu_+|^2$, which comes from the change of coordinates $dxd\xi \to Jd\nu_-d\nu_+ ds$.% (see next section). 
\end{rem}

\subsection{Construction of a time-change that is not semisimple on one-forms}

Here we construct a smooth, positive function on the unit sphere bundle $S\Sigma$ of a compact hyperbolic surface $\Sigma = \Gamma \backslash \mathbb{H}^2$ such that under a time-change of the geodesic flow, the action of the Lie derivative on resonant $1$-forms at zero is not semisimple. We establish a few auxiliary lemmas first. We denote by $\pi_\Gamma: \mathbb{H}^2 \to \Gamma \backslash \mathbb{H}^2$ the associated projection.

%First recall that $\Bd^0(-1) = \{w \in \mathcal{D}'(S^1) : L^*_\gamma w (\nu) = N_\gamma(\nu) w(\nu)\}$. Since $L_\gamma^* d\nu = N^{-1}_\gamma(\nu) d\nu$ by \eqref{eq:jacobianG}, we have that $w(\nu)d\nu$ is invariant under the action of $\Gamma$ for $w \in \Bd^0(-1)$. Also recall that $\Gamma \cong \pi_1(M)$, so by the results above we know that the space of invariant $1$-currents on $S^1$ is of dimension $2g$, where $g$ is the genus, which is quite nice.\footnote{Does this hold in higher dimensions? E.g., is the dimension of invariant $2$-currents on $S^2$ by $\pi_1(M)$ is equal to $b_1(M)$ for a hyperbolic $3$-manifold? Also, can we prove that invariant $1$-currents on $S^1$ by $\pi_1(M)$ is of dimension $2g$ with other, more direct means?} %Moreover, we know that $\int_{S^1} w(\nu) d\nu = 0$ for any $w \in \Bd^0(-1)$. One way to see this is by \cite[Lemma 5.11.]{DFG15} that holds for the particular value $\lambda = -1$ and gives $\langle{\pi_*v_1, \pi_*v_2}\rangle = 0$ where $v_1$ is a resonance and $v_2$ a co-resonance at $-1$.

%\footnote{And even if \cite[Lemma 5.11.]{DFG15} is false, i.e. isn't true for this particular set of indices, we know that $\pi_*: \res_{X, \Omega_0^1}(-1) \to \Eig(0) = \{consts.\}$ and so we obtain $2g - 1$ elements in the kernel.}

\begin{lemma}\label{lemma:wproperties}
	Let $w \in \Bd^0(-1)$. Then $w(\nu) d\nu$ is $\Gamma$-invariant and we have
	\[\int_{S^1} w(\nu) d\nu = 0.\]
\end{lemma}
\begin{proof}
	For the first claim, recall that by \eqref{eq:jacobianG} we have $L_\gamma^*d\nu = N_\gamma^{-1}(\nu) d\nu$ for any $\gamma \in G$. Therefore, by \eqref{eq:bd0} we have also $L_\gamma^*(wd\nu) = wd\nu$ for any $\gamma \in \Gamma$ which gives the required property.
	
	The second property is a direct consequence of the works \cite{DFG15} or \cite{GHW18} on pairings. Note that \cite[Lemma 5.11.]{DFG15} proves a pairing formula, which for $\lambda = -1$ gives
\begin{equation}\label{eq:push}
	\langle{\pi_* v_1, \pi_* v_2}\rangle_\Sigma = 0
\end{equation}
	for all $v_1$ resonance states at $-1$ and $v_2$ co-resonant states at $-1$. Here $\pi_*$ maps first band resonant and co-resonant states at $-1$ to eigenfunctions of $\Delta$ on $\Sigma$ at zero by \cite[Lemma 5.8.]{DFG15}, so $\pi_*v_1$ and $\pi_* v_2$ are constants. Using the time-reversal map $R$ from Subsection \ref{subsec:reversal} we may identify resonant and co-resonant states, i.e. we have $R^*: \res_X^0(-1) \to \res_{X^*}^0(-1)$ an isomorphism. Moreover, we claim that $\pi_* R^*v = \pi_* v$ for any $v \in \res_X^0(-1)$. For this recall the connection $1$-form $\psi$ on $S\Sigma$ (dual to the vertical fibre), and observe that $\pi_*v = \pi_*(v \psi)$. Then for any two form $\theta$ on $\Sigma$
	\[\langle{\pi_* (R^*v\psi), \theta}\rangle_{\Sigma} = \int_{S\Sigma} R^*v \psi \wedge \pi^*\theta = \langle{\pi_*(v \psi), \theta}\rangle_{\Sigma}.\]
	Here we used $R^*\psi = \psi$ and $\pi \circ R = \pi$. By applying \eqref{eq:push} to $v_2 = R^*v_1$, we obtain that $\pi_*$ is zero on both resonant and co-resonant states.
	
	Alternatively, this follows directly from the proof of \cite[Theorem 1]{GHW18} (more precisely, see \cite[p. 19]{GHW18} and the start of discussion of $\lambda_0 = -n$ case).
\end{proof}
Next we prove an auxiliary lemma that relies on the dynamics of the action of $\Gamma$ on $S^1$.

\begin{lemma}\label{lemma:kernelw}
	Let $w \in \Bd^0(-1)$ and let $(\nu_-, \nu_+) \in S^1 \times S^1$ with $\nu_- \neq \nu_+$. Then there exists a $\varphi \in C^\infty(S^1)$, such that
	\begin{itemize}
		\item[1.] $\varphi \geq 0$.
		\item[2.] $\varphi(\nu_+) \neq 0$.
		\item[3.] $\varphi$ vanishes in a neighbourhood of $\nu_-$.
		\item[4.] $\langle{w, \varphi}\rangle_{S^1} = 0$.
	\end{itemize}
\end{lemma}
\begin{proof}
	We denote by $B_\varepsilon(A)$ the $\varepsilon$-neighbourhood of a set $A$. Let $\varphi_\varepsilon \in C^\infty(S^1)$ be a non-negative function with $\varphi_\varepsilon = 1$ outside $B_\varepsilon(\nu_-)$ and $\varphi_\varepsilon = 0$ in $B_{\varepsilon/2}(\nu_-)$; assume also $0 \leq \varphi_\varepsilon \leq 1$. Here $\varepsilon > 0$ is a small enough positive number. If $\langle{w, \varphi_\varepsilon}\rangle = 0$ for some $\varepsilon$, we are done by setting $\varphi = \varphi_\varepsilon$. %\todo{G: Is $w=w_{2}$?} 
	If not, then we may assume $\langle{w, \varphi_\varepsilon}\rangle > 0$ for every $\varepsilon > 0$. For, assume $\langle{w, \varphi_\varepsilon}\rangle > 0$ and $\langle{w, \varphi_\delta}\rangle < 0$ for some $\varepsilon, \delta > 0$. Then if we take $s = - \frac{\langle{w, \varphi_\varepsilon}\rangle}{\langle{w, \varphi_\delta}\rangle} > 0$, we have $\langle{w, \varphi_\varepsilon + s\varphi_\delta}\rangle = 0$ and so we are done by setting $\varphi = \varphi_\varepsilon + s\varphi_\delta$.
	
	Next, we may w.l.o.g. assume $\langle{w, \varphi_\varepsilon}\rangle > 0$ for all $\varepsilon > 0$ small enough. By Lemma \ref{lemma:wproperties} we have $\langle{w, 1}\rangle = 0$, which implies $\langle{w, 1 - \varphi_\varepsilon}\rangle < 0$. The invariance of $w(\nu) d\nu$ under the action of $\Gamma$ following from Lemma \ref{lemma:wproperties} then yields, that for any $\psi \in C^\infty(S^1)$
	\begin{equation}\label{eq:changeofcoord}
		\langle{w, \psi}\rangle = \int_{S^1} L_\gamma^*(w(\nu) d\nu) \psi = \int_{S^1} w(\nu) \psi \circ L_{\gamma^{-1}}(\nu) d\nu = \langle{w, \psi \circ L_{\gamma^{-1}}}\rangle.
	\end{equation}
	Now use that since $\Gamma \cong \pi_1(M)$ has $2g \geq 4$ generators, it is not elementary by \cite[Theorem 2.4.3]{SK92}. Therefore, by \cite[Exercise 2.13]{SK92} we have that $\Gamma$ contains infintely many hyperbolic elements (fixing exactly two elements of $S^1$), no two of which have a common fixed points.
	
	So take $\gamma \in \Gamma$ hyperbolic such that $\nu_-, \nu_+$ are not in the set of fixed points of $\gamma$, which we denote by $\{p_1, p_2\}$. Assume without loss of generality $p_1$ is an attractor and $p_2$ a repeller.
%\footnote{Every hyperbolic element of $PSL(2, \mathbb{R})$ is conjugate to $T: z \mapsto kz$ for $0< k < 1$ or $k > 1$. The fixed points of $T$ are $0$ and $\infty$, and we clearly have the desired behaviour of attracting/repelling.}
	
	By \eqref{eq:changeofcoord} for $\psi = 1 - \varphi_\varepsilon$, we get that $\langle{w, 1 - \varphi_\varepsilon}\rangle = \langle{w, (1 - \varphi_\varepsilon)\circ L_{\gamma^{-1}}}\rangle < 0$. Since $\supp ((1 - \varphi_\varepsilon)\circ L_{\gamma^{-1}}) = L_\gamma(B_\varepsilon(\nu_-))$, we have that for $n \geq N_0$ large enough, $\varphi_{\varepsilon, n} := (1 - \varphi_\varepsilon)\circ L_{\gamma^{-n}}$ has support arbitrarily close to $p_1$, so disjoint from $\nu_-$ and $\nu_+$. Therefore,
	%\[\langle{w, \varphi_{\varepsilon, n}}\rangle = \langle{w, 1 - \varphi_\varepsilon}\rangle < 0, \quad\]
	for $s = -\frac{\langle{w, \varphi_\varepsilon}\rangle}{\langle{w, \varphi_{\varepsilon, n}}\rangle} > 0$, we have
	\[\langle{w, \varphi_\varepsilon + s \varphi_{\varepsilon, n}}\rangle = 0.\]
	Then $\varphi = \varphi_\varepsilon + s \varphi_{\varepsilon, n}$ does the job.
\end{proof}
With this in hand, we can prove the following claim:

\begin{theorem}\label{thm:time-change}
	Let $\Sigma = \Gamma \backslash \mathbb{H}^2$ be a closed hyperbolic surface. Fix $w_2 \in \Bd^0(-1)$ and let $v_2 \in \res_{X^*}^0(-1)$ be the corresponding co-resonant state, according to Lemma \ref{lemma:boundarydistributions}. Then there exists an $f \in C^\infty(S\Sigma)$ with $f > 0$ such that
	\begin{equation}
		\int_{S\Sigma} f v_1 \overline{v}_2 dxd\xi = 0
	\end{equation}
	for all $v_1 \in \res_{X}^0(-1)$. In other words, semisimplicity of the Lie derivative $\Lapl_{-iX/f}$ acting on resonant one forms at zero fails.
\end{theorem}

%\todo[inline]{G: This confuses me now: are $v_1, v_2$ both in the same space or is one resonant and the other co-resonant? Have we used the previous subsection on time reversals at all?}
%\blue{M: No time-reversal is used and I relegated it to a different section. You are right, $v_2$ is a co-resonant state.}
\begin{proof}
	We divide the construction of $f$ into several steps.
	
	\emph{Step 1.} First, fix $(x_0, \xi_0) \in S\mathbb{H}^2$. Denote the corresponding coordinates of $(x_0, \xi_0)$ by $(\nu_{0-}, \nu_{0+}, s_0)$, according to \ref{eq:coords}. By Lemma \ref{lemma:kernelw}, there is a non-negative $\varphi_+ \in C^\infty(S^1)$, non-vanishing at $\nu_{0+}$, vanishing near $\nu_{0-}$ and in the kernel of $w_2$. Now let $\varphi_- \in C^\infty(S^1)$ be such that $\varphi_- \geq 0$, $\varphi_-(\nu_{0-}) \neq 0$ and $\supp(\varphi_+) \cap \supp(\varphi_-) = \emptyset$. %\todo{G: supports of points??}
	Also, let $\psi \in C_0^\infty(\mathbb{R})$ be such that $\psi(s_0) \neq 0$ and $\psi \geq 0$. Set $\chi(\nu_-, \nu_+, s) := \varphi_+(\nu_+) \varphi_-(\nu_-) \psi(s)$. Take any $w_1 \in \Bd^0(-1)$ and denote the corresponding element of $\res_X^0(-1)$ by $v_1$. Then by the computation in Proposition \ref{prop:+-representation} for $\lambda = -1$, we have $F^*\pi_\Gamma^*(v_1 \overline{v}_2 dxd\xi) = \frac{1}{2}w_1(\nu_-) \overline{w}_2(\nu_+) d\nu_-d\nu_+ ds$ and
	\begin{multline}
		\big\langle{\pi_\Gamma^*(v_1 \overline{v}_2 dx d\xi), F_* \chi}\big\rangle_{S\mathbb{H}^2} = \frac{1}{2}\big\langle{w_1(\nu_-) \overline{w}_2(\nu_+) d\nu_-d\nu_+ ds, \chi}\big\rangle_{(S^1 \times S^1)_\Delta \times \mathbb{R}}\\ = \frac{1}{2}\langle{w_1, \varphi_-}\rangle \langle{\overline{w}_2, \varphi_+}\rangle \langle{ds, \psi}\rangle = 0
	\end{multline}
	since $\langle{w_2, \varphi_+}\rangle = 0 $ by the construction. We will denote the $\chi$ above by $\chi_{(x_0, \xi_0)}$ and by $U_{(x_0, \xi_0)}$ a neighbourhood of $(x_0, \xi_0)$ where $F_*\chi_{(x_0, \xi_0)} > 0$. Note that $\chi$ is a function in $C_0^\infty\Big((S^1 \times S^1)_\Delta \times \mathbb{R}\Big)$, by the condition on disjoint supports of $\varphi_-$ and $\varphi_+$ in the construction, and as $\psi \in C_0^\infty(\mathbb{R})$. Therefore we have $F_*\chi$ a function in $C_0^\infty(S\mathbb{H}^2)$.
	
	\emph{Step 2.} Denote by $\mathcal{D} \subset \mathbb{H}^2$ a compact fundamental domain for $\Sigma$. Then $S\mathcal{D}$ is a fundamental domain for $S\Sigma$. %Then
	%\[S \mathcal{D} \subset \bigcup_{(x, \xi) \in S\mathbb{H}^2} U_{(x, \xi)}\]
	By compactness, we have an $N > 0$ and $(x_i, \xi_i) \in S\mathbb{H}^2$ for $i = 1, 2, \dotso, N$ such that 
	\[S \mathcal{D} \subset \bigcup_{(x_i, \xi_i)} U_{(x_i, \xi_i)}.\]
	Define then 
	\[F_*\chi(x, \xi) := \sum_{i = 1}^N F_*\chi_{(x_i, \xi_i)} (x, \xi) \in C_0^\infty\big(S\mathbb{H}^2\big).\]%C_0^\infty\big((S^1 \times S^1)_\Delta \times \mathbb{R}\big)\]
	By the construction, we have
	\begin{equation}\label{eq:zerointegral}
		\big\langle{\pi_\Gamma^*(v_1\overline{v}_2 dxd\xi), F_*\chi}\big\rangle_{S\mathbb{H}^2} = \frac{1}{2}\sum_{i = 1}^N \big\langle{w_1(\nu_-)\overline{w}_2(\nu_+) d\nu_- d\nu_+ ds, \chi_{(x_i, \xi_i)}}\big\rangle_{(S^1 \times S^1)_\Delta \times \mathbb{R}} = 0.
	\end{equation}
	
	\emph{Step 3.} We introduce the push-forward map $\pi_*: C_0^\infty(S\mathbb{H}^2) \to C^\infty(S\Sigma)$, by defining for any $\eta \in C_0^\infty(S\mathbb{H}^2)$
	\begin{equation}
		\pi_* \eta (x, \xi) := \sum_{\gamma \in \Gamma} \eta\big(\gamma \cdot(x_0, \xi_0)\big) \in C^\infty(S\Sigma).
	\end{equation}
	Here $(x_0, \xi_0) \in \pi_\Gamma^{-1}(x, \xi) \subset S\mathbb{H}^2$ is an arbitrary point in the fiber and the definition of $\pi_*$ is independent any choices. Note that the only accumulation points of orbits of $\Gamma$ are on the boundary at infinity $S^1$, so the pushforward is well-defined and sequentially continuous. Note also that $\pi_*$ is dual to $\pi_\Gamma^*$ in the sense of distributions.
	
	Then we observe that $f(x, \xi) := \pi_* F_*\chi (x, \xi) \in C^\infty(S\Sigma)$ satisfies the required properties. Firstly,
	\begin{equation}
		\big\langle{v_1\overline{v}_2 dx d\xi, f}\big\rangle_{S\Sigma} = \big\langle{\pi_\Gamma^*(v_1\overline{v}_2 dxd\xi), F_*\chi}\big\rangle_{S\mathbb{H}^2} = 0
	\end{equation}
	by the equation \eqref{eq:zerointegral} from Step 2 and duality of $\pi_*$ with $\pi_\Gamma^*$. Secondly, we have $f > 0$. To see this, let $(x, \xi) \in S\Sigma$ and denote a lift to $S\mathbb{H}^2$ by $(x_0, \xi_0)$. Then there exists $\gamma' \in \Gamma$ with $\gamma' \cdot (x_0, \xi_0) \in \mathcal{D}$. Therefore, there is an $i \in \{1, 2, \dotso, N\}$ with $\gamma' \cdot (x_0, \xi_0) \in U_{(x_i, \xi_i)}$ and so $F_*\chi_{(x_i, \xi_i)}(\gamma' \cdot (x_0, \xi_0)) > 0$. Hence
	\[f(x, \xi) = \sum_{\gamma \in \Gamma} F_*\chi(\gamma \cdot (x_0, \xi_0)) \geq \sum_{i = 1}^N F_*\chi_{(x_i, \xi_i)}\big(\gamma' \cdot (x_0, \xi_0)\big) \geq F_*\chi_{(x_i, \chi_i)}(\gamma' \cdot (x_0, \xi_0)) > 0.\]
	
This proves the first claim. The final claim now follows directly from the correspondence in \eqref{eq:correspondence} and Proposition \ref{prop:semisimplicitypairing}.
\end{proof}

\begin{rem}\rm
	One may see the element in the kernel of $\Lapl_{X/f}^2$ and not in the kernel of $\Lapl_{X/f}$ constructed in Theorem \ref{thm:time-change} more explicitly. Namely, one such element is given by the formula
	\[u' = -iR_H(0) \big(f u\big).\]
	Here $u \in \res_X^0(-1)$ is an element such that $\int_{S\Sigma} f u v dxd\xi = 0$ for all $v \in \res_{X^*}^0(-1)$ and $R_H(\lambda)$ is the holomorphic part at zero of $(-i\Lapl_X - \lambda)^{-1}$ on one forms. The conclusion follows as $\Pi_0(fu) = 0$ and $-iR_H(0)$ is an inverse to $\Lapl_X$ on $\ker \Pi_0 \cap \mathcal{D}'_{E_u^*}(M; \Omega^1)$. 
\end{rem}
Theorem \ref{thm:time-change}. completes the proof of Theorem \ref{thm:localresonance}. We conclude this section with the following:

\begin{proof}[Proof of the second part of Corollary \ref{cor:zeta'}]
	By Theorem \ref{thm:localresonance} there is a time change $fX$ on the unit sphere bundle $S\Sigma$ of a closed hyperbolic surface $\Sigma$ with $\ker \Lapl_{fX}^2 \neq \ker \Lapl_{fX}$ on $\Omega_0^1(S\Sigma)$. By Theorem \ref{thm:volpres}, for the flow $fX$ we have $m_0(0) = m_2(0) = 1$ and $\dim \res_1(0) = b_1(\Sigma)$, so that $m_1(0) \geq b_1(\Sigma) + 1$. The claim then follows by applying Corollary \ref{cor:3dfactor}.
\end{proof}

%\appendix

%\section{An inconvenient Jacobian}

%n this appendix, we compute the Jacobian of a change of coordinates in the space $S\mathbb{H}^2$, that we require in 

%\todo[inline]{G: check that all references are used. Make them uniform/update.}


\begin{thebibliography}{aa}
\bibitem{A88} T. Adachi, \emph{Meromorphic extension of L-functions of Anosov flows and profinite graphs}, Kumamoto J. Math. {\bf 1} (1988), 9--24.

\bibitem{AS87} T. Adachi, T. Sunada, \emph{Twisted Perron-Frobenius theorem and L-functions}, J. Funct. Anal. {\bf 71} (1987), no. 1, 1--46.

%\bibitem{Ar69} V. I. Arnold, \emph{The one-dimensional cohomologies of the Lie algebra of divergence-free vector fields, and the winding numbers of dynamical systems,} Functional Anal. Appl. {\bf 3} (1969) 319--321.

\bibitem{AZ07} N. Anantharaman, S. Zelditch, \emph{Patterson-Sullivan distributions and quantum ergodicity.} Ann. Henri Poincar\'{e} {\bf 8} (2007), no. 2, 361--426.

\bibitem{AK98} V. I. Arnold, B. A. Khesin, \emph{Topological methods in hydrodynamics}, Applied Math. Sciences {\bf 125}, Springer, New York (1998).

\bibitem{Ba} V. Baladi, {\it Anisotropic Sobolev spaces and dynamical transfer operators: $C^{\infty}$ foliations.,} Algebraic and topological dynamics, 123–135, Contemp. Math., {\bf 385}, Amer. Math. Soc., Providence, RI, 2005.

\bibitem{BaT} V. Baladi, M. Tsujii, {\it Anisotropic Hölder and Sobolev spaces for hyperbolic diffeomorphisms,} Ann. Inst. Fourier (Grenoble) {\bf 57} (2007) 127--154.

\bibitem{BMM16} W. Ballmann, H. Matthiesen, S. Mondal, \emph{Small eigenvalues of closed surfaces,} J. Differential Geom. {\bf 103} (2016), no. 1, 1--13.

\bibitem{BKL}  M. Blank, G. Keller, C. Liverani, {\it Ruelle-Perron-Frobenius spectrum for Anosov maps,} Nonlinearity {\bf 15} (2002) 1905--1973.

\bibitem{BL18} C. Bonatti, R. Langevin, \emph{Un exemple de flot d'Anosov transitif transverse \`a un tore et non conjugu\'e \`a une suspension}, Ergodic Theory Dynam. Systems {\bf 14} (1994), no. 4, 633--643. 

\bibitem{BT} R. Bott, L. W. Tu, \emph{Differential forms in algebraic topology}, Graduate Texts in Mathematics {\bf 82}, Springer-Verlag, New York-Berlin, 1982. xiv+331 pp.

\bibitem{BL} O. Butterley, C. Liverani, {\it Smooth Anosov flows: correlation spectra and stability,} J. Mod. Dyn. {\bf 1} (2007) 301--322.

\bibitem{ChDa19} Y. Chaubet, N. V. Dang, \emph{Dynamical torsion for contact Anosov flows}, arXiv:1911.09931, preprint (2019).

\bibitem{DP} N. Dairbekov, G.P. Paternain, \emph{ Longitudinal KAM-cocycles and action spectra of magnetic flows}, Math. Res. Lett. {\bf 12} (2005) 719--729. 

\bibitem{DGRS18} N.V. Dang, C. Guillarmou, G. Riviere, S. Shen, \emph{Fried Conjecture in small dimensions}, Invent. Math. {\bf 220} (2020), 525--579.

\bibitem{DR17} N. V. Dang, G. Rivi\`ere, \emph{Topology of Pollicott-Ruelle resonant states}, to appear in Ann. Sc. Norm. Sup. di Pisa.

\bibitem{DMM86} R. de la Llave, J. M. Marco, R. Moriy\'on, \emph{Canonical perturbation theory of Anosov systems and regularity results for the Liv\v sic cohomology equation.} Ann. of Math. (2) {\bf 123} (1986), no. 3, 537--611.

\bibitem{DFG15} S. Dyatlov, F. Faure, C. Guillarmou, \emph{Power spectrum of the geodesic flow on hyperbolic manifolds.}, Anal. PDE {\bf 8} (2015), no. 4, 923--1000. 

\bibitem{DyG16} S. Dyatlov, C. Guillarmou, \emph{Pollicott-Ruelle resonances for open systems}, Ann. Henri Poincar\'e {\bf 17} (2016), no. 11, 3089--3146. 

\bibitem{DyZw16} S. Dyatlov, M. Zworski, \emph{Dynamical zeta functions for Anosov flows via microlocal analysis}, Ann. Sci. Éc. Norm. Supér. (4) {\bf 49} (2016), no. 3, 543--577. 

\bibitem{DyZwbook} S. Dyatlov, M. Zworski, \emph{Mathematical theory of scattering resonances.} Graduate Studies in Mathematics, 200. American Mathematical Society, Providence, RI, 2019. xi+634 pp.

\bibitem{DyZw17} S. Dyatlov, M. Zworski, \emph{Ruelle zeta function at zero for surfaces}, Invent. Math. {\bf 210} (2017), no. 1, 211--229.

\bibitem{FaGu17} F. Faure, C. Guillarmou, \emph{Horocyclic invariance of Ruelle resonant states for contact Anosov flows in dimension 3}, Math Research Letters, {\bf 25} (2018), no 5, 1405--1427.

\bibitem{FaSj11} F. Faure, J. Sj\"ostrand, \emph{Upper bound on the density of Ruelle resonances for Anosov flows}, Comm. Math. Phys. {\bf 308} (2011), no. 2, 325--364.

\bibitem{FoHa} P. Foulon, B. Hasselblatt, \emph{Zygmund strong foliations,}
Israel J. Math. {\bf 138} (2003), 157--169. 

\bibitem{F1} D. Fried, \emph{Analytic torsion and closed geodesics on hyperbolic manifolds}, Invent. Math. {\bf 84} (1986), no. 3, 523--540.

\bibitem{GiLiPo} P. Giulietti, C. Liverani, M. Pollicott, \emph{Anosov flows and dynamical zeta functions}, Ann. of Math. (2) {\bf 178} (2013), no. 2, 687--773.

\bibitem{GoLi}  S. Gou\"{e}zel, C. Liverani, {\it Banach spaces adapted to Anosov systems,} Ergodic Theory Dynam. Systems {\bf 26} (2006) 189--217.

\bibitem{GrSj94} A. Grigis, J. Sj\"ostrand, \emph{Microlocal analysis for differential operators. An introduction.} London Mathematical Society Lecture Note Series, {\bf 196.} Cambridge University Press, Cambridge, 1994. iv+151 pp.

\bibitem{YGB19} Y. Guedes Bonthonneau, \emph{Perturbation of Ruelle resonances and Faure--Sj\"ostrand anisotropic space}, Rev. Un. Mat. Argentina {\bf 61} (2020), no. 1, 63--72.

\bibitem{G17} C. Guillarmou, \emph{Invariant distributions and X-ray transform for Anosov flows}, J. Differential Geom, {\bf 105} (2017), no. 2, 177--208.

\bibitem{GHW18} C. Guillarmou, J. Hilgert, T. Weich, \emph{Classical and quantum resonances for hyperbolic surfaces}, Math. Annalen {\bf 370} (2018), Volume 370, Issue 3-4, pp 1231--1275.

\bibitem{HoI+II} L. H\"ormander, \emph{The Analysis of Linear Partial Differential Operators, Volumes I and II}, Springer, 1983.

\bibitem{HoIII+IV} L. H\"ormander, \emph{The Analysis of Linear Partial Differential Operators, Volumes III and IV}, Springer, 1985.

\bibitem{SK92} S. Katok, \emph{Fuchsian groups}, Chicago Lectures in Mathematics, University of Chicago Press, Chicago, IL, 1992. x+175 pp.

\bibitem{KW19} B. K\"uster, T. Weich, \emph{Pollicott-Ruelle resonant states and Betti numbers}, to appear in Comm. Math. Phys (2020).

\bibitem{Li1}  C. Liverani, {\it On contact Anosov flows,} Ann. of Math. {\bf 159} (2004) 1275--1312.

\bibitem{Ni} J. P. Nicholls, \emph{The ergodic theory of discrete groups.} London Mathematical Society Lecture Note Series, 143. Cambridge University Press, Cambridge, 1989. xii+221 pp.

\bibitem{Pnotes} M. Pollicott, \emph{Dynamical zeta functions}, online lecture notes: \url{https://homepages.warwick.ac.uk/~masdbl/grenoble-16july.pdf}

\bibitem{Ru76} D. Ruelle, \emph{Zeta-functions for expanding maps and Anosov flows}, Invent. Math. {\bf 34} (1976), no. 3, 231--242.

\bibitem{Shen} S. Shen, \emph{Analytic torsion, dynamical zeta functions, and the Fried conjecture}, Anal. PDE {\bf 11} (2018) 1--74.

\bibitem{ST76} I.M. Singer, J.A. Thorpe, {\it Lecture notes on elementary topology and geometry,} Undergraduate Texts in Mathematics. Springer-Verlag, New York-Heidelberg, 1976.

\bibitem{Z} M. Zworski, \emph{Commentary on ``Differentiable dynamical systems'' by Stephen Smale,} Bull. Amer. Math. Soc. (N.S.) {\bf 55} (2018) 331--336.
\end{thebibliography}
\end{document}